\documentclass[reqno]{amsart}
\usepackage[left=2.5cm,right=2.5cm,top=3cm,bottom=3cm]{geometry}
\usepackage{amsmath}
\usepackage{amssymb}
\usepackage{amsfonts}
\usepackage{amsthm}
\usepackage{mathtools}
\usepackage{leftindex}
\usepackage{stmaryrd}
\usepackage{amsthm}
\usepackage{cite}
\usepackage{tikz} 
\usetikzlibrary{tqft,calc}
\usepackage{enumitem}
\usepackage{graphicx}
\usepackage{tikz-cd}
\usepackage[hidelinks=true]{hyperref}
\usepackage{ulem}
\usetikzlibrary{cd}
\usetikzlibrary{positioning}
\usepackage[new]{old-arrows}

\newcommand{\IC}{\mathbb{C}}

\newcommand{\g}{\mathfrak{g}}

\renewcommand{\l}{\mathfrak{l}}
\newcommand{\p}{\mathfrak{p}}
\renewcommand{\k}{\mathfrak{k}}
\renewcommand{\t}{\mathfrak{t}}
\newcommand{\z}{\mathfrak{z}}

\renewcommand{\a}{\mathfrak{a}}
\renewcommand{\sl}{\mathfrak{sl}}

\newcommand{\C}{\mathcal{C}}
\newcommand{\D}{\mathcal{D}}

\newcommand {\G}{\mathcal G}
\renewcommand{\H}{\mathcal{H}}

\newcommand{\N}{\mathcal{N}}
\renewcommand{\O}{\mathcal{O}}

\renewcommand{\S}{\mathcal{S}} 

\newcommand{\ad}{\mathrm{ad}}
\newcommand{\Ad}{\mathrm{Ad}}
\renewcommand{\exp}{\mathrm{exp}}

\newcommand{\pr}{\mathrm{pr}}
\newcommand{\rank}{\mathrm{rank}}
\newcommand{\sss}{\mathsf{s}}
\newcommand{\ttt}{\mathsf{t}}
\newcommand{\iii}{\mathsf{i}}
\newcommand{\mmm}{\mathsf{m}}

\newcommand{\sll}[1]{\mkern-4mu\mathbin{/\mkern-5mu/}_{\mkern-4mu{#1}}}
\newcommand{\tto}{\;\substack{\longrightarrow\\[-9pt] \longrightarrow}\;}

\newcommand{\semi}{{\mathsf{s}}}
\newcommand{\nilp}{\mathsf{n}}
\newcommand{\reg}{\mathsf{reg}}
\newcommand{\greg}{\mathfrak{g}_{\reg}}

\newcommand\junk[1]{}

\newcommand{\too}{\longrightarrow}
\newcommand{\fp}[2]{\leftindex_{#1}\times_{#2}\,}

\newcommand {\ol}[1]{\overline{#1}}

\numberwithin{equation}{section}

\newtheorem{theorem}{Theorem}[section]
\newtheorem{proposition}[theorem]{Proposition}
\newtheorem{corollary}[theorem]{Corollary}
\newtheorem{lemma}[theorem]{Lemma}

\newtheorem{mtheorem}{Main Theorem}

\theoremstyle{definition}
\newtheorem{definition}[theorem]{Definition}
\newtheorem{example}[theorem]{Example}
\newtheorem{remark}[theorem]{Remark}

\makeatletter
\@namedef{subjclassname@2020}{\textup{2020} Mathematics Subject Classification}
\makeatother

\begin{document}
	
	\title[Slices for reductive group actions]{Slices for reductive group actions in algebraic\\ and holomorphic symplectic geometry}
	
	\author[Peter Crooks]{Peter Crooks}
	\author[Rebecca Goldin]{Rebecca Goldin}
	\author[Yiannis Loizides]{Yiannis Loizides}
	\address[Peter Crooks]{Department of Mathematics and Statistics\\ Utah State University \\ 3900 Old Main Hill \\ Logan, UT 84322, USA}
	\email{peter.crooks@usu.edu}
	\address[Rebecca Goldin]{Department of Mathematical Sciences \\ George Mason University \\ 4400 University Drive \\ Fairfax, VA 22030, USA}
	\email{rgoldin@gmu.edu}
	\address[Yiannis Loizides]{Department of Mathematical Sciences \\ George Mason University \\ 4400 University Drive \\ Fairfax, VA 22030, USA}
	\email{yloizide@gmu.edu}

	
	\begin{abstract}
		Symplectic slice theorems elucidate the local structure of symplectic manifolds carrying Hamiltonian actions of compact Lie groups. We generalize these theorems in two natural settings. The first is based on the idea that complex reductive algebraic groups are the natural complex-geometric counterparts of compact Lie groups. Using new definitions of \textit{Poisson} and \textit{symplectic slices}, we prove analogues of the classical symplectic slice theorems for Hamiltonian actions of complex reductive algebraic groups. These analogues have versions in the complex-algebraic and holomorphic categories, and make extensive use of \textit{Slodowy slices} and \textit{decomposition classes} in complex reductive Lie algebras. The starting point for our second setting is the fact that Hamiltonian Lie group actions are special cases of Hamiltonian symplectic groupoid actions. We generalize the classical symplectic slice theorems to the latter case.
	\end{abstract}
	\maketitle
	\setcounter{tocdepth}{2}
	\tableofcontents

	
	\section{Introduction}
	\subsection{Context}\label{Subsection: Context} Hamiltonian Lie group actions play a distinguished role at the interface of mathematical physics, Poisson geometry, and symplectic geometry. A recurring paradigm is that the local structure of such actions is elucidated by their \textit{cross-sections} or \textit{slices} \cite{fla-rat:96,hil-nee-pla:94,con-daz-mol:88,SymplecticFibrations,gui-ste:83,ler-mei-tol-woo,ort-rat:02,los:06,bie:97,cro-ray:19,sja:95}. This paradigm and the associated \textit{slice theorems} underlie crucial advances in Poisson and symplectic geometry over the last several decades, including convexity and connectedness results \cite{ler-mei-tol-woo,lin-sja,lan:18,mei:17,osh-sja,sja}, geometric quantization \cite{sep-ngo,chuah,mei-sja}, log symplectic geometry \cite{bra-kie-mir:23,cro-roe:22,mat-mir:23}, $\mathrm{Spin}^c$-geometry \cite{hoc-son,par,par:03,mei-adv}, symplectic contraction \cite{hil-man-mar,lan:20}, symplectic cutting \cite{ler:95,mar-tha:12}, and symplectic implosion \cite{gui-jef-sja:02,hur-jef-sja}. Such theorems are best understood for Hamiltonian actions of compact Lie groups on real Poisson and symplectic manifolds, where fundamental Weyl chambers often play a significant role. 
	
	In contrast, some of the most modern and intriguing Hamiltonian actions occur in the complex-algebraic and holomorphic categories \cite{bie:97,bry-kos:94,gan-gin:02,kal:06,kro,los:06,bor-mac:81,bea:00,chr-gin,moo-tac:11,her-sch-sea:20,boa:01,ati-bie:02,boa:12,dan-kir-swa:13,may:20,may:22,bie:23,bie:17}. We are thereby motivated to formulate and prove slice theorems for Hamiltonian actions of complex reductive algebraic groups, the natural complex-geometric counterparts of compact Lie groups. Versions of this endeavor have attracted considerable interest; see \cite{sch:17,cro-roe:22,los:06,kal:06}. While the transition from compact to complex reductive groups may sound innocuous for developing slice theories, it introduces several new obstacles; one is the lack of direct, complex-geometric analogues of fundamental Weyl chambers. Perhaps the closest analogues of these chambers are \textit{Slodowy slices} \cite{slo:80} in complex reductive Lie algebras; they feature prominently in the study of nilpotent orbits \cite{slo:80,kos:63}, Poisson deformations \cite{lehn}, and quantizations \cite{gan-gin:02,gin:09,pre:07,amb}. The problem is that Slodowy slices generally do not satisfy the definition of a \textit{slice} in the context of Lie group actions on manifolds. It is therefore natural to seek a broader definition of a \textit{slice} in Hamiltonian geometry, with a view to accomplishing the following: witnessing Slodowy slices and more traditional slices as examples, and providing enough flexibility to prove reasonable analogues of the symplectic slice theorems in algebraic and holomorphic symplectic geometry.
	
	There is another context in which to re-examine the classical slice theorems. To this end, one has notions of a \textit{symplectic groupoid} and Hamiltonian actions thereof \cite{mik-wei:88,kar:86,wei:87,cos-daz-wei:87}. These notions have shaped modern research in Poisson and symplectic geometry, including geometric quantization \cite{wei-xu:91,cat-fel:01,wei:912}, Hamiltonian reduction \cite{cro-may:22,mik-wei:88,cat-zam:07,cat-zam:09,bra-fer}, Lie-theoretic Poisson geometry \cite{cra-fer,cra-fer2}, local normal form results \cite{fer-mar,fre-mar:17,bis-bur-lim-mei,bur-lim-mei}, and mathematical quantum field theory \cite{cro-may:24,cro-may:242,cal:15}. It turns out that Hamiltonian $G$-spaces for a Lie group $G$ are in bijective correspondence with Hamiltonian spaces for the \textit{cotangent groupoid} $T^*G\tto\g^*$, where $\g$ is the Lie algebra of $G$. This fact suggests the possibility of generalizing the classical slice theorems to the setting of Hamiltonian symplectic groupoid actions. 
	
	\subsection{Objectives} The purpose of this manuscript is to systematically broaden and generalize the classical slice theorems in Hamiltonian geometry, along the lines outlined above. Our first objective is to adopt an appropriately encompassing definition of a \textit{slice}. This is accomplished by our definitions of \textit{Poisson slices} and \textit{symplectic slices}. We then pursue the following further objectives. 
	\begin{itemize}
		\item[\textup{(i)}] Broaden the classical slice theorems to Hamiltonian actions of complex reductive algebraic groups. As part of this process, define an appropriate role for Slodowy slices.
		\item[\textup{(ii)}] Generalize the classical slice theorems to Hamiltonian actions of symplectic groupoids.
	\end{itemize}
	
	These further objectives are accomplished via Main Theorems \ref{thm:simplestpcs}, \ref{Theorem: General slices}, and \ref{thm:abs-symp-cr-sec} in Section \ref{Subsection: Main results}.
	
	\subsection{Analogies}\label{Subsection: Analogies}
	As indicated above, much of our work is inspired by analogies between Hamiltonian actions in the smooth and complex-algebraic / holomorphic categories. The following table is intended to give extremely rough, preliminary, and largely definition-free indications of the analogies used and derived in this manuscript. Each row consists of a relevant object in the smooth category, its analogue in the complex-algebraic / holomorphic categories, and an indication of whether this analogy is ``direct" or only ``approximate". More precise details can be found in the balance of our introduction.   
	\vspace{10pt}
	
	\begin{center}
		\scalebox{1}{\begin{tabular}{| c | c | c |} 
				\hline
				\textbf{Smooth category} & \textbf{Algebraic / holomorphic categories} & \textbf{Analogy}\\
				\hline
				compact Lie group $K$; $\k=\mathrm{Lie}(K)$ &  complex reductive algebraic group $G$; $\g=\mathrm{Lie}(G)$ & direct \\ 
				\hline
				orbit-type stratum $\Sigma\subset\k^*$ & decomposition class $\D\subset\g$ & direct \\
				\hline
				fundamental Weyl chamber $\mathfrak{t}^*_{+}\subset\k^*$ & principal Slodowy slice $\S\subset\g$ & approx. \\
				\hline
				open face $\sigma\subset\t^*_{+}$ & subset $e+\z(\l)_{\mathsf{gen}}\subset\g$ & approx. \\
				\hline
				compact torus $A(\xi)\coloneqq K_{\xi}/[K_{\xi},K_{\xi}]$, $\xi\in\k^*$ & algebraic group $A(x)\coloneqq G_{x}/[G_{x_{\semi}},G_{x_{\semi}}]_{x_{\nilp}}^{\circ}$, $x\in\g$ & direct \\
				\hline
				natural slice $\S_{\sigma}\subset\k^*$, $\sigma\subset\t^*_{+}$ open face & natural slice $\S_x\subset\g$, $x\in\g$ & approx.\\
				\hline
				Hamiltonian $K$-space $M$ & Hamiltonian $G$-variety / $G$-space $M$ & direct\\
				\hline
				principal stratum $\Sigma_M\subset\k^*$ & principal decomposition class $\D_M\subset\g$ & direct\\
				\hline
				principal open subset $\mu^{-1}(\Sigma_M)\subset M$ & open dense subset $\mu^{-1}(\D_M)\subset M$ & direct\\
				\hline
				principal face $\sigma_M\subset\mathfrak{t}^*_{+}$ & subset $e+\z(\l)_{\mathsf{gen}}\subset\g$ related to $\D_M$ & approx.\\
				\hline
				principal slice $\mu^{-1}(\sigma_M)\subset M$ & symplectic slice $\mu^{-1}(e+\z(\l)_{\mathsf{gen}})\subset M$ & approx.\\
				\hline
		\end{tabular}}
	\end{center}
	
	\vspace{10pt}
	
	\subsection{Poisson and symplectic slices}\label{Subsection: Poisson slices}
	Suppose that $M$ is a Poisson Hamiltonian $G$-space for a Lie group $G$. We define a \textit{Poisson slice} in $M$ to be a Poisson transversal $S\subset M$ that is transverse to the $G$-orbits in $M$. While this definition is in the smooth category, it has clear analogues in the holomorphic and complex-algebraic categories. If $M$ is symplectic, then a submanifold $S\subset M$ is a Poisson slice if and only if it is a symplectic submanifold that is transverse to the $G$-orbits in $M$. We adopt the term \textit{symplectic slice} for a Poisson slice in a (symplectic) Hamiltonian $G$-space; its analogues in the holomorphic and complex- algebraic categories are clear. Our notion of a \textit{Poisson slice} also turns out to generalize that in \cite{cro-roe:22}. Some prominent examples of Poisson slices are discussed below.
	
	\subsection{The case of Hamiltonian actions of compact Lie groups}\label{Subsection: The case of Hamiltonian actions of compact Lie groups} Our work is inspired by the following seminal results, together with their connections to Poisson and symplectic slices. Let $K$ be a compact connected Lie group with Lie algebra $\k$. Fix a maximal torus $T\subset K$ with Lie algebra $\t\subset\k$, as well as a fundamental Weyl chamber $\t^*_{+}\subset\k^*$. Elements of an open face $\sigma\subset\t^*_{+}$ have a uniform $K$-stabilizer, denoted $K_{\sigma}\subset K$. The following result stems from the works of Condevaux--Dazord--Molino \cite{con-daz-mol:88}, Hilgert--Neeb--Plank \cite{hil-nee-pla:94}, and Lerman--Meinrenken--Tolman--Woodward \cite{ler-mei-tol-woo}.
	
	\begin{theorem}[Principal faces, slices, strata, and open subsets]\label{thm:compactpcs} 
		Suppose that $M$ is a non-empty, connected Hamiltonian $K$-space with moment map $\mu: M \too \k^*$.
		\begin{itemize}
			\item[\textup{(i)}] There exists a unique open face $\sigma_M\subset\t_+^*$ with the property that $\mu(M)\cap \sigma_M$ is dense in $\mu(M)\cap \t_+^*$.
			\item[\textup{(ii)}] The pre-image $\mu^{-1}(\sigma_M)$ is a connected, $T$-invariant, symplectic submanifold of $M$, 
			and the restriction of $\mu$ to $\mu^{-1}(\sigma_M)$ is a moment map for the Hamiltonian $T$-action on $\mu^{-1}(\sigma_M)$.
			\item[\textup{(iii)}] The subgroup $[K_{\sigma_M},K_{\sigma_M}]\subset K$ acts trivially on $\mu^{-1}(\sigma_M)$.
			\item[\textup{(iv)}] The $K$-saturation $K\mu^{-1}(\sigma_M)$ is open and dense in $M$.
		\end{itemize}
	\end{theorem}

	The open face $\sigma_M\subset\t^*_{+}$ and submanifold $\mu^{-1}(\sigma_M)\subset M$ are called the \textit{principal face} and \textit{principal slice}, respectively. Open faces of $\t^*_{+}$ index orbit-type strata of the coadjoint representation of $K$ on $\k^*$ \cite{btd,dk}: the stratum indexed by an open face $\sigma\subset\t^*_{+}$ is the $K$-saturation $K\sigma\subset\k^*$. Theorem \ref{thm:compactpcs}(i) is thereby equivalent to the existence of a unique orbit-type stratum $\Sigma_M\subset\k^*$ for which $\mu(M)\cap\Sigma_M$ is open and dense in $\mu(M)$. We call $\Sigma_M\subset\k^*$ and $K\mu^{-1}(\sigma_M)$ the \textit{principal stratum} and \textit{principal open subset}, respectively. Note that $\Sigma_M = K\sigma_M$ and $K\mu^{-1}(\sigma_M)=\mu^{-1}(\Sigma_M)$. It follows that Part (iv) of Theorem \ref{thm:compactpcs} may be rephrased as $\mu^{-1}(\Sigma_M)$ being open and dense in $M$. Part (iii) of Theorem \ref{thm:compactpcs} is equivalent to $[K_{\xi},K_{\xi}]$ acting trivially on $\mu^{-1}(\xi)$ for all $\xi\in\Sigma_M$, where $K_\xi\subset K$ is the stabilizer of $\xi$.   It follows that the quotient $A(\xi)\coloneqq K_{\xi}/[K_{\xi},K_{\xi}]$ acts on $\mu^{-1}(\xi)$ for all $\xi\in\Sigma_M$. We thereby obtain the set-theoretic identity
	\begin{equation}\label{Equation: Reduction identity} M\sll{\xi}K=\mu^{-1}(\xi)/A(\xi)\end{equation} for all $\xi\in\Sigma_M$, where the left-hand side is the Hamiltonian reduction of $M$ by $K$ at level $\xi$.
	
	Theorem \ref{thm:compactpcs}(ii) has a direct generalization to any open face $\sigma\subset\t^*_{+}$. To this end, the $K_{\sigma}$-saturation $$
	\S_{\sigma} := K_\sigma\left(\bigcup_{\overline{\tau}\supset \sigma}\tau\right)$$ is a submanifold of $\k^*$; it is sometimes called a \textit{natural slice} \cite{ler-mei-tol-woo}. The following is proved by Lerman--Meinrenken--Tolman--Woodward \cite{ler-mei-tol-woo}.
	
	\begin{theorem}[Natural slices]\label{thm:compactscs} Suppose that $M$ is a connected Hamiltonian $K$-space with moment map $\mu: M \too \k^*$. For any open face $\sigma\subset\t^*_{+}$, $\mu^{-1}(\S_{\sigma})$ is a $K_\sigma$-invariant symplectic manifold of $M$, 
		and the $K_\sigma$-action on $\mu^{-1}(\S_{\sigma})$ is Hamiltonian.
	\end{theorem}
	
	Let us discuss Theorems \ref{thm:compactpcs} and \ref{thm:compactscs} in the context of Section \ref{Subsection: Poisson slices}. For any open face $\sigma\subset\t^*_{+}$, one finds that $\S_{\sigma}$ is a Poisson slice in $\k^*$. One also finds that $\mu^{-1}(\S_{\sigma})$ and $\mu^{-1}(\sigma_M)$ are symplectic slices in $M$. While $\sigma_M$ turns out not a Poisson slice in $\k^*$, it is a Poisson slice in the Poisson submanifold $\Sigma_M\subset\k^*$. 
	
	\subsection{Main results}\label{Subsection: Main results}
	We now outline the main results of this manuscript. Let $G$ be a connected complex reductive algebraic group with Lie algebra $\g$. We have the Lie algebra decomposition $\g=\z(\g)\oplus[\g,\g]$, where $\z(\g)$ is the center of $\g$. Fix a non-degenerate, symmetric bilinear form on $\z(\g)$. Together with the Killing form on $[\g,\g]$, it induces a symmetric, non-degenerate, $G$-invariant bilinear form on $\g$. We use this last form to freely identify $\g$ with $\g^*$ in what follows.
	
	Every $x\in\g$ admits a unique decomposition of the form $x = x_{\semi} + x_{\nilp}$, where $x_{\semi}\in\g$ is semisimple, $x_{\nilp}\in\g$ is nilpotent, and $[x_{\semi},x_{\nilp}]=0$; it is called the \textit{Jordan decomposition} of $x$. Note that $x_{\nilp}\in[\g_{x_{\semi}},\g_{x_{\semi}}]$ for all $x\in\g$. This fact allows one to consider $[G_{x_{\semi}},G_{x_{\semi}}]_{x_{\nilp}}$, the $[G_{x_{\semi}},G_{x_{\semi}}]$-stabilizer of $x_{\nilp}\in[\g_{x_{\semi}},\g_{x_{\semi}}]$. Let $[G_{x_{\semi}},G_{x_{\semi}}]_{x_{\nilp}}^{\circ}$ denote the identity component of $[G_{x_{\semi}},G_{x_{\semi}}]_{x_{\nilp}}$. At the same time, define an equivalence relation on $\g$ by $x\sim y$ if $y_{\nilp}=\mathrm{Ad}_g(x_{\nilp})$ and $\g_{y_{\semi}}=\mathrm{Ad}_g(\g_{x_{\semi}})$ for some $g\in G$. Equivalence classes are called \textit{decomposition classes} of $\g$, and are introduced in a work of Borho--Kraft \cite{bor-kra:79}. It turns out that decomposition classes are indexed by $G$-conjugacy classes of pairs $(\l,\O)$, where $\l\subset\g$ is a Levi subalgebra and $\O\subset\l$ is a nilpotent orbit; one assigns to the class of $(\l,\O)$ the decomposition class $G(\z(\l)_{\mathsf{gen}}+\O)$, where $\z(\l)\subset\l$ is the center and the \textit{generic locus} $\z(\l)_{\mathsf{gen}}\subset\z(\l)$ is defined by
	\begin{align*}
		&{ \z(\l)_{\mathsf{gen}} \coloneqq \{x\in\z(\l):\dim\g_x\leq\dim\g_y\text{ for all }y\in\z(\l)\} }.
	\end{align*}
	
	Our first main result is the following rough analogue of Theorem \ref{thm:compactpcs}; it appears in the main text as Corollary \ref{Corollary: Poisson slice} and Propositions \ref{Proposition: Symplectic subvariety}, \ref{Proposition: Characterizations}, \ref{Proposition: Holomorphic characterizations}, and \ref{Proposition: Trivial action}. 
	
	\begin{mtheorem}[Complex-geometric analogue of Theorem \ref{thm:compactpcs}]\label{thm:simplestpcs} Let $G$ be a connected complex reductive algebraic group with Lie algebra $\g$. Suppose that $M$ is a non-empty, connected holomorphic Hamiltonian $G$-space (resp. non-empty, irreducible Hamiltonian $G$-variety) with moment map $\mu:M\longrightarrow\g$. \begin{itemize}
			\item[\textup{(i)}] There exists a unique decomposition class $\D_M\subset \g$ with the property that $\mu(M)\cap\D_M$ is open and dense in $\mu(M)$. 
			\item[\textup{(ii)}] The pre-image $\mu^{-1}(\D_M)$ is an open, dense, $G$-invariant submanifold (resp. subvariety) of $M$. 
			\item[\textup{(iii)}] If $x\in\D_M$ has  Jordan decomposition $x=x_\semi+x_\nilp$,
			then $[G_{x_{\semi}},G_{x_{\semi}}]_{x_{\nilp}}^{\circ}$ acts trivially on $\mu^{-1}(x)$. 
		\end{itemize}
		Furthermore, let us fix a pair $(\l,\O)$ whose $G$-conjugacy class indexes $\D_M$. Let $L\subset G$ be the Levi subgroup integrating $\l$.   
		\begin{itemize}
			\item[\textup{(iv)}] Suppose that $e\in\O$, and let $L_e\subset L$ be the $L$-stabilizer of $e$. Then $e+\z(\l)_{\mathsf{gen}}$ is a Poisson slice in $\D_M$, $\mu^{-1}(e+\z(\l)_{\mathsf{gen}})$ is a symplectic slice in $M$, and the Hamiltonian actions of $G$ on $\g$ and $M$ restrict to Hamiltonian actions of $L_e$ on $e+\z(\l)_{\mathsf{gen}}$ and $\mu^{-1}(e+\z(\l)_{\mathsf{gen}})$, respectively.
		\end{itemize}
	\end{mtheorem}
	
	A comparison to Theorem \ref{thm:compactpcs} is warranted. Observe that $\D_M\subset\g$ and $\mu^{-1}(\D_M)\subset M$ are the complex-geometric analogues of the principal stratum $\Sigma_M\subset\k^*$ and principal open subset $\mu^{-1}(\Sigma_M)\subset M$, respectively. On the other hand, one should view $e+\z(\l)_{\mathsf{gen}}\subset\g$ and $\mu^{-1}(e+\z(\l)_{\mathsf{gen}})\subset M$ as ``approximate" analogues of the principal face $\sigma_M\subset\k^*$ and principal slice $\mu^{-1}(\sigma_M)\subset M$, respectively. One reason to include the adjective ``approximate" is that there is no direct, complex-geometric analogue of the fundamental Weyl chamber $\t^*_{+}\subset\k^*$. We discuss this issue in Proposition \ref{Proposition: No analogue}.
	
	\begin{remark}
		Recall the discussion of $A(\xi)$ in the context of Equation \eqref{Equation: Reduction identity}. Main Theorem \ref{thm:simplestpcs}(iii) motivates us to define $A(x)\coloneqq G_x/[G_{x_{\semi}},G_{x_{\semi}}]_{x_{\nilp}}^{\circ}$ for $x\in\g$. Observe that $$M\sll{x}G=\mu^{-1}(x)/A(x)$$ set-theoretically for all $x\in\D_M$. While $A(\xi)$ is a compact torus for all $\xi\in\k^*$, the structure of $A(x)$ is more subtle. It turns out to be a finite and potentially non-abelian extension of a complex torus; see Corollary \ref{Corollary: Simple} and Remark \ref{Remark: Non-abelian}. At the same time, we show that $A(x)$ is the smallest quotient $Q$ of $G_x$ with the following property: if $M$ is any non-empty, connected, holomorphic Hamiltonian $G$-space with moment map $\mu:M\longrightarrow\g$ for which $x\in\D_M$, then the $G_x$-action on $\mu^{-1}(x)$ descends to one of $Q$ on $\mu^{-1}(x)$. This is the statement of Corollary \ref{Corollary: Smallest quotient}.
	\end{remark}
	
	Our second main result is an analogue of Theorem \ref{thm:compactscs} for Hamiltonian $G$-actions. To this end, suppose that $x\in\g$. Consider the set of $G_{x_{\semi}}$-conjugacy classes of pairs $(\l,\O)$, where $\l\subset\g_{x_{\semi}}$ is a Levi subalgebra and $\O\subset\l$ is a nilpotent orbit. Write $[(\l,\O)]$ for the $G_{x_{\semi}}$-conjugacy class of such a pair, and consider the decomposition class in $\g_{x_{\semi}}$ given by $$\D_{\l,\O}^{G_{x_{\semi}}}\coloneqq G_{x_{\semi}}(\z(\l)_{\mathsf{gen}}+\O).$$ Define the \textit{natural slice} $\S_x$ at $x$ by
	$$\S_x\coloneqq\bigcup_{[(\l,\O)]\text{ s.t. } x\in\ol{\D^{G_{x_\semi}}_{\l,\O}}}\D^{G_{x_\semi}}_{\l,\O}.$$
	On the other hand, suppose that $\mathfrak{T}=(e,h,f)\in[\g_{x_{\semi}},\g_{x_{\semi}}]^{\times 3}$ is an $\mathfrak{sl}_2$-triple satisfying $e=x_{\nilp}$. Consider the Slodowy slice
	\[ \S_{\mathfrak{T}}\coloneqq e+(\g_{x_\semi})_f\subset \g_{x_\semi}. \]
	We define the \textit{complementary slice} $\S_{x,\mathfrak{T}}\subset\g_{x_{\semi}}$ by
	$$\S_{x,\mathfrak{T}}\coloneqq \S_x\cap\S_{\mathfrak{T}}.$$ Let us also write $\sss:T^*G\longrightarrow\g$ for the result of composing the left trivialization $T^*G\overset{\cong}\longrightarrow G\times\g$ with the projection $G\times\g\longrightarrow\g$. These considerations lead to second main result of our manuscript; it appears in the main text as Proposition \ref{Proposition: Semisimple case}, Proposition \ref{Proposition: Natural transversal}, Corollary \ref{Corollary: Poisson slices}, Proposition \ref{p:decompcontainDx}, Corollary \ref{c:slice-generalUx}, Corollary \ref{Corollary: Nice}, Proposition \ref{Proposition: Slodowy saturation}, and Corollary \ref{c:slice-general}.
	
	\begin{mtheorem}[Complex-geometric analogue of Theorem \ref{thm:compactscs}]\label{Theorem: General slices}
		Let $x\in\g$ and $\mathfrak{T}\in[\g_{x_{\semi}},\g_{x_\semi}]^{\times 3}$ be as discussed above. Suppose that $M$ is a holomorphic Hamiltonian $G$-space (resp. Hamiltonian $G$-variety) with moment map $\mu:M\longrightarrow\g$. The following statements are true:
		\begin{itemize}
			\item[\textup{(i)}] $\S_x$ and $\S_{x,\mathfrak{T}}$ are Poisson slices in $\g$; 
			\item[\textup{(ii)}] $\mu^{-1}(\S_x)$ and $\mu^{-1}(\S_{x,\mathfrak{T}})$ are symplectic slices in $M$;
			\item[\textup{(iii)}] $\sss^{-1}(\S_x)$ and $\sss^{-1}(\S_{x,\mathfrak{T}})$ are symplectic submanifolds (resp. subvarieties) of $T^*G$;
			\item[\textup{(iv)}] $G\S_x=G\S_{x,\mathfrak{T}}$ and $G\mu^{-1}(\S_x)=G\mu^{-1}(\S_{x,\mathfrak{T}}$); these are open in $\g$ and $M$, respectively;
			\item[\textup{(v)}] the $G$-action map induces symplectomorphisms
			\begin{equation}\label{Equation: First isomorphism}
				\frac{G\times \mu^{-1}(\S_x)}{\sim}\overset{\cong}\longrightarrow G\mu^{-1}(\S_x)
			\end{equation} 
			and 
			\begin{equation}\label{Equation: Second isomorphism}
				\frac{G\times \mu^{-1}(\S_{x,\mathfrak{T}})}{\sim}\overset{\cong}\longrightarrow G\mu^{-1}(\S_{x,\mathfrak{T}}),
			\end{equation} 
			where $(g_1,m_1)\sim (g_2,m_2)$ if $g_1m_1=g_2m_2$, and the symplectic structure on the left-hand side of \eqref{Equation: First isomorphism} is obtained by restricting the product symplectic structure on $\sss^{-1}(\S_x)\times \mu^{-1}(\S_x)$ to the fiber product $\sss^{-1}(\S_x)\times_{\S_x} \mu^{-1}(\S_x)\cong G\times \mu^{-1}(\S_x)$ and descending to the quotient; the symplectic structure on the left-hand side of \eqref{Equation: Second isomorphism} is obtained analogously.
		\end{itemize}
	\end{mtheorem}
	
	To formulate the abstract slice theorem, we note that there is a canonical bijection between Hamiltonian $G$-spaces and Hamiltonian spaces for the cotangent groupoid $T^*G\tto\g$ \cite{mik-wei:88}. It is therefore natural to consider an arbitrary symplectic groupoid $\G\tto X$; this is a groupoid object in the category of smooth manifolds, complex manifolds, or complex algebraic varieties that satisfies additional hypotheses. Let $\sss:\G\too X$ and $\ttt:\G\too X$ denote the source and target morphisms, respectively, and note that $X$ carries a unique Poisson structure for which $\ttt$ is a Poisson morphism. It is natural to define a \textit{Poisson slice} in this context as a Poisson transversal in $X$ that is transverse to the $\G$-orbits; in fact, the second property is redundant. Given a Poisson slice $S\subset X$, the restriction of $\G$ to $S$ is the symplectic subgroupoid
	$$\G^S_S:=\sss^{-1}(S) \cap \ttt^{-1}(S)\tto S$$ of $\G$. One may also consider the symplectic subvarieties or submanifolds
	\[ \G_S:=\sss^{-1}(S)\quad\text{and}\quad \G^S:=\ttt^{-1}(S)\] of $\G$. Each is a Hamiltonian $\G\times\G_{S}^S$-space.
	
	Suppose that $M$ is a Hamiltonian $\G$-space with moment map $\mu:M\longrightarrow X$. The underlying action map takes the form 
	$$\G\times_X M\longrightarrow M,\quad (g,p)\mapsto gp,$$ where
	$\G\times_X M\coloneqq\{(g,p)\in\G\times M: \sss(g)=\mu(p)\}$ is the fiber product with respect to $\sss$ and $\mu$. The $\G$-saturation of a subset $S\subset M$ is defined by 
	$$
	\G S\coloneqq\{gp:(g,p)\in \G\times_X S\}\subset M.
	$$
	At the same time, a Poisson slice $S\subset X$ has the property that $\mu^{-1}(S)$ is a symplectic submanifold or subvariety of $M$ with the property of being transverse to $\G$-orbits; we declare such a submanifold or subvariety of $M$ to be a \textit{symplectic slice}. One also finds that $\mu^{-1}(S)$ is a Hamiltonian $\G_S^S$-space, and that the fiber product $\G_{S}\times_S\mu^{-1}(S)$ inherits an action of $\G_S^S$. These considerations underlie the following result; it is stated in the smooth category as Theorem \ref{t:abs-symp-cr-sec}, and applies analogously in the holomorphic and complex-algebraic categories.
	
	\begin{mtheorem}[Abstract symplectic slices]\label{thm:abs-symp-cr-sec}
		Consider a symplectic groupoid $\G\tto X$ and Poisson slice $S\subset X$. Let $M$ be a Hamiltonian $\G$-space with moment map $\mu: M\longrightarrow X$. The following statements are true:
		\begin{itemize}
			\item[\textup{(i)}] $\mu^{-1}(S)$ is a symplectic slice in $M$ and a Hamiltonian $\G_{S}^S$-space; 
			\item[\textup{(ii)}] $\G\mu^{-1}(S)$ is an open subset of $M$;
			\item[\textup{(iii)}] the action map $\G\times_X M\longrightarrow M$ restricts and descends to a symplectomorphism $$
			(\G_S\times_S \mu^{-1}(S))/\G_S^S\overset{\cong}\longrightarrow
			\G\mu^{-1}(S),$$
			where the symplectic structure on the left-hand side is obtained by restricting the product symplectic structure on $\G_S\times\mu^{-1}(S)$ to $\G_S\times_S \mu^{-1}(S)$, and noting that it descends to $(\G_S\times_S \mu^{-1}(S))/\G_S^S$.
		\end{itemize}
	\end{mtheorem}
	
	\subsection{Organization}
	In Section \ref{Section: Preliminaries}, we outline some of the Lie-theoretic and Poisson-geometric conventions observed throughout this manuscript. Section \ref{sec:slices} then introduces Poisson and symplectic slices, reviews results of \cite{ler-mei-tol-woo} in this context, and illustrates the nuances that occur when attempting to extend these results to the holomorphic and complex algebro-geometric categories. We devote Section \ref{Section: Decomposition classes in Poisson geometry} to the role of Borho and Kraft's decomposition classes in our work, in addition to proving most of Main Theorem \ref{thm:simplestpcs}. In Section \ref{Section: Residual}, we obtain several results regarding residual group actions on level sets over a principal decomposition class. Section \ref{sec:symplectic slices} contains our results on slices for symplectic groupoid actions, as well as those on natural and complementary slices. The proofs of Main Theorems \ref{Theorem: General slices} and \ref{thm:abs-symp-cr-sec} appear in Section \ref{sec:symplectic slices}. 
	
	\subsection*{Acknowledgements} P.C. was supported by the National Science Foundation grant DMS-2454103 and Simons Foundation grant MP-TSM-00002292. R.G. was supported by the National Science Foundation grant DMS-2152312.
	
	\section{Preliminaries}\label{Section: Preliminaries}
	We now outline some of the conventions and definitions that persist throughout our manuscript, and review core results and constructions on Lie group actions in Poisson and symplectic geometry. Care is taken to ensure that specializations to the smooth, holomorphic, and complex-algebraic categories are clear from context.
	
	\subsection{Actions, orbits, and stabilizers}\label{Subsection: Actions, orbits, and stabilizers}
	In this manuscript, group actions are understood to be left group actions. Suppose that a group $G$ acts on a set $X$. We write $gx\in X$ for the action of $g\in G$ on $x\in X$. The orbit and stabilizer of $x\in X$ are denoted $Gx\subset X$ and $G_x\subset G$, respectively. We define the \textit{saturation} of a subset $Y\subset X$ by $$GY\coloneqq\{gx:g\in G,\text{ }x\in Y\}\subset X.$$ Note that $Gx=G\{x\}$ for all $x\in X$. On the other hand, define an equivalence relation on $X$ by $x\sim y$ if $G_x$ and $G_y$ are conjugate as subgroups of $G$. We refer to equivalence classes of $\sim$ as \textit{orbit-type strata}.
	
	Let $G$ be a real (resp. complex) Lie group with Lie algebra $\g$ and exponential map $\exp:\g\longrightarrow G$. We use the term \textit{$G$-manifold} for a smooth (resp. complex) manifold $M$ endowed with a smooth (resp. holomorphic) $G$-action. We then define the \textit{generating vector field} associated to $x\in\g$ by
	\begin{equation}\label{Equation: Generating vector field} (x_M)_p\coloneqq\frac{\mathrm{d}}{\mathrm{d}t}\bigg\vert_{t=0}\bigg(\exp(-tx)p\bigg)\in T_pM\end{equation} for all $p\in M$. Note that $$\g_p\coloneqq\{x\in\g:(x_M)_p=0\}\subset\g$$ is the Lie algebra of $G_p\subset G$ for all $p\in M$.
	
	\subsection{Basic Lie-theoretic preliminaries}\label{Subsection: Basic Lie-theoretic considerations} Let $G$ be a Lie group with Lie algebra $\g$. One may consider the adjoint representation $\mathrm{Ad}:G\longrightarrow\operatorname{GL}(\g)$ and coadjoint representation $\Ad^*:G\longrightarrow\operatorname{GL}(\g^*)$. By differentiating $\Ad$ (resp. $\Ad^*$) at $e\in G$, we obtain the adjoint representation $\ad:\g\longrightarrow\mathfrak{gl}(\g)$ (resp. coadjoint representation $\ad^*:\g\longrightarrow\mathfrak{gl}(\g^*)$) of $\g$. We have $$\g_x\coloneqq\{y\in\g:[y,x]=0\}\quad\text{and}\quad\g_{\xi}\coloneqq\{y\in\g:\xi([y,z])=0\text{ for all }z\in\g\}$$ for $x\in\g$ and $\xi\in\g^*$.
	
	\subsection{Compact Lie groups}\label{Subsection: Compact groups} Suppose that $K$ is a compact connected Lie group with Lie algebra $\k$. By complexifying the adjoint representation of $K$ (resp. $\k$), we obtain a representation of $K$ (resp. $\k$) on $\k_{\mathbb{C}}\coloneqq\k\otimes_{\mathbb{R}}\mathbb{C}$. Let $T\subset K$ be a maximal torus with Lie algebra $\t\subset\k$. Set $\t_{\mathbb{C}}\coloneqq\t\otimes_{\mathbb{R}}\mathbb{C}$. The \textit{roots} of $K$ with respect to $T$ constitute the finite subset $\Phi\subset\t^*\setminus\{0\}$ satisfying
	$$\k_{\mathbb{C}}=\t_{\mathbb{C}}\oplus\bigoplus_{\alpha\in\Phi}(\k_{\mathbb{C}})_{\alpha},$$ where $$(\k_{\mathbb{C}})_{\alpha}\coloneqq\{x\in\k_{\mathbb{C}}:[y,x]=-i\alpha(y)x\text{ for all }y\in\t\}$$ for $\alpha\in\Phi$. Choose a set of positive roots, i.e. a subset $\Phi^{+}\subset\Phi$ satisfying $\Phi=\Phi^{+}\sqcup(-\Phi^{+})$. The set $$\t_{+}\coloneqq\{x\in\t:\alpha(x)\geq 0\text{ for all }\alpha\in\Phi^{+}\}$$ is a fundamental domain for the adjoint action of $K$ on $\k$.
	
	Choose a $K$-invariant inner product on $\k$; it determines a $K$-module isomorphism between the adjoint and coadjoint representations of $K$. This isomorphism identifies $\t_{+}\subset\k$ with a fundamental domain $\t_{+}^*\subset\k^*$ for the coadjoint action of $K$. One calls $\t^*_{+}$ the \textit{fundamental Weyl chamber}. As a polyhedral cone, $\t_{+}^*$ is a set-theoretic disjoint union of the interiors of its faces. We refer to these interiors as the \textit{open faces} of $\t_{+}^*$. Given an open face $\sigma\subset\t^*_{+}$, the stabilizer subgroup $K_\xi\subset K$ is independent of $\xi\in\sigma$. We write $K_{\sigma}$ for this stabilizer subgroup.
	
	\subsection{Complex reductive algebraic groups}\label{Subsection: Complex reductive algebraic groups}
	Let $G$ be a connected complex reductive algebraic group with rank $\ell$, Lie algebra $\g$, and exponential map $\exp:\g\longrightarrow G$. Consider the open, dense, $G$-invariant subset $$\greg\coloneqq\{x\in\g:\dim\g_x=\ell\}$$ of \textit{regular} elements of $\g$. On the other hand, $x\in\g$ is called \textit{semisimple} (resp. \textit{nilpotent}) if the Zariski-closure of $\{\exp(tx):t\in\mathbb{C}\}\subset G$ is a torus (resp. isomorphic to the additive group $\mathbb{C}^n$ for some $n\in\mathbb{Z}_{\geq 0}$). We denote by $\g_{\mathsf{semi}}\subset\g$ and $\N\subset\g$ the $G$-invariant subsets of semisimple and nilpotent elements, respectively. Adjoint orbits contained in $\greg$ (resp. $\g_{\mathsf{semi}}$, $\mathcal{N}$) are called \textit{regular} (resp. \textit{semisimple}, \textit{nilpotent}). One defines $[\g,\g]_{\mathsf{semi}}$ analogously to $\g_{\mathsf{semi}}$, replacing $G$ and $\g$ with $[G,G]$ and $[\g,\g]$, respectively. One also has
	$$[\g,\g]_{\mathsf{semi}}=\{x\in[\g,\g]:\ad_x:[\g,\g]\longrightarrow[\g,\g]\text{ is diagonalizable}\}$$ and $$\mathcal{N}=\{x\in[\g,\g]:\ad_x:[\g,\g]\longrightarrow[\g,\g]\text{ is nilpotent}\}.$$ Further details on the last few sentences are available in \cite[Section 3.7]{oni-vin}.
	
	Each $x\in\g$ admits a Jordan decomposition, i.e. there exist unique elements $x_{\semi}\in\g_{\mathsf{semi}}$ and $x_{\nilp}\in\N$ satisfying $x=x_{\semi}+x_{\nilp}$ and $[x_{\semi},x_{\nilp}]=0$ \cite[Section 3.7]{oni-vin}. One calls $x_{\semi}$ and $x_{\nilp}$ the semisimple and nilpotent parts, respectively, of $x$. We also have $G_x=G_{x_{\semi}}\cap G_{x_{\nilp}}=(G_{x_{\semi}})_{x_{\nilp}}$ and $\g_x=\g_{x_{\semi}}\cap\g_{x_{\nilp}}=(\g_{x_{\semi}})_{x_{\nilp}}$ for all $x\in\g$. On the other hand, analogous considerations yield Jordan decompositions in $[\g,\g]$. The decomposition $\g=\z(\g)\oplus[\g,\g]$ now suggests comparing Jordan decompositions in $\g$ with those in $[\g,\g]$, where $\z(\g)\subset\g$ is the center. While we assume the following comparison to be a standard result, we include a proof for completeness. 
	
	\begin{proposition}\label{Proposition: Comparative Jordan decomposition}
		Write $x=z+y$ for the decomposition of $x\in\g$ into $z\in\z(\g)$ and $y\in[\g,\g]$. Consider the Jordan decompositions $x=x_{\semi}+x_{\nilp}$ of $x$ in $\g$, and $y=y_{\semi}+y_{\nilp}$ of $y$ in $[\g,\g]$. We have $x_{\semi}=z+y_{\semi}$ and $x_{\nilp}=y_{\nilp}$.
	\end{proposition} 
	
	\begin{proof}
		Observe that $x = (z + y_{\semi}) + y_{\nilp}$ and $[z + y_{\semi}, y_{\nilp}]=0$. It therefore suffices to prove that $z + y_{\semi}$ and $y_{\nilp}$ are semisimple and nilpotent in $\g$, respectively.
		
		Choose a maximal torus $T \subset [G,G]$ such that $\exp(ty_{\semi})\in T$ is for all $t\in\mathbb{C}$. Note that $T'\coloneqq Z(G)T$ is a maximal torus in $G$, where $Z(G)\subset G$ is the center. It is also clear that $\exp(t(z + y_{\semi})) = \exp(tz)\exp(ty_{\semi})\in T'$ for all $t\in\mathbb{C}$. This implies that the Zariski-closure of $\{\exp(t(z + y_{\semi})):t\in\mathbb{C}\}$ in $G$ is the Zariski-closure of the same set in $T'$. Since the latter closure is necessarily a torus, $z + y_{\semi}$ is semisimple in $\g$.
		
		The Zariski-closure of $\{\exp(ty_{\nilp}):t\in\mathbb{C}\}$ in $G$ is the Zariski-closure of the same set in $[G,G]$. Since $y_{\nilp}$ is nilpotent in $[\g,\g]$, the latter closure is an additive group. We conclude that $y_{\nilp}$ is nilpotent in $\g$.
	\end{proof}
	
	Some immediate consequences are that $$\mathcal{N}\subset[\g,\g]\quad\text{and}\quad \g_{\mathsf{semi}}=\z(\g)+[\g,\g]_{\mathsf{semi}}.$$
	
	Choose a non-degenerate bilinear form $\langle\cdot,\cdot\rangle_1:\z(\g)\otimes\z(\g)\longrightarrow\mathbb{C}$, and let $\langle\cdot,\cdot\rangle_2:[\g,\g]\otimes[\g,\g]\longrightarrow\mathbb{C}$ denote the Killing form on $[\g,\g]$. By means of the decomposition $\g=\z(\g)\oplus[\g,\g]$, $\langle\cdot,\cdot\rangle_1$ and $\langle\cdot,\cdot\rangle_2$ induce a bilinear form $\langle\cdot,\cdot\rangle:\g\otimes\g\longrightarrow\mathbb{C}$. Let us write $V_1^{\perp_1}\subset\z(\g)$ (resp. $V_2^{\perp_2}\subset[\g,\g]$, $V^{\perp}\subset\g$) for the annihilator of a subspace $V_1\subset\z(\g)$ (resp. $V_2\subset[\g,\g]$, $V\subset\g$) under $\langle\cdot,\cdot\rangle_1$ (resp. $\langle\cdot,\cdot\rangle_2$, $\langle\cdot,\cdot\rangle$). It follows that $$(V_1\oplus V_2)^{\perp}=V_1^{\perp_1}\oplus V_2^{\perp_2}$$ for all subspaces $V_1\subset\z(\g)$ and $V_2\subset[\g,\g]$. The form $\langle\cdot,\cdot\rangle$ is non-degenerate and $G$-invariant. As such, it induces a $G$-module isomorphism between $\g$ and $\g^*$. We use this isomorphism to freely identify $\g$ and $\g^*$ for the balance of this manuscript.
	
	We now recall complex group-theoretic analogues of the root space decomposition in Section \ref{Subsection: Compact groups}. Choose a maximal torus $T\subset G$ with Lie algebra $\t\subset\g$. One has
	$$\g=\t\oplus\bigoplus_{\alpha\in\Phi}\g_{\alpha}$$ for some finite subset $\Phi\subset\t^*\setminus\{0\}$, where
	$$\g_{\alpha}\coloneqq\{x\in\g:[y,x]=\alpha(y)x\text{ for all }y\in\t\}$$ for $\alpha\in\Phi$. Elements of $\Phi$ are called \textit{roots} of $G$ with respect to $T$.
	
	\subsection{Poisson and symplecto-geometric preliminaries}\label{Subsection: Poisson transversals} 
	We now outline some of our conventions in Poisson and symplectic geometry. Let $M$ be a smooth manifold, complex manifold, or smooth complex algebraic variety. In the first (resp. second, third) case, constructions and definitions are generally understood as occurring in the smooth (resp. holomorphic, complex-algebraic) category. We insert real-geometric, complex-geometric, and algebro-geometric adjectives only in cases where ambiguity may arise.
	
	Contracting a differential $2$-form $\omega\in\mathrm{H}^0(M,\wedge^2(T^*M))$ with tangent vectors yields a vector bundle morphism $\omega^{\vee}:TM\longrightarrow T^*M$. One calls $M$ \textit{symplectic} if it comes equipped with a differential $2$-form $\omega\in\mathrm{H}^0(M,\wedge^2(T^*M))$ for which $\mathrm{d}\omega=0$ and $\omega^{\vee}$ is a vector bundle isomorphism. In this case, $\omega$ is called a \textit{symplectic form}. Our convention is to define the canonical symplectic form on a cotangent bundle to be minus the exterior derivative of the tautological $1$-form.
	
	Write $\O_M$ for the structure sheaf of $M$. One calls $M$ \textit{Poisson} if $\O_M$ has been enriched to a sheaf of Poisson algebras. In this case, there exists a unique bivector field $\sigma\in\mathrm{H}^0(M,\wedge^2(TM))$ satisfying $\{f_1,f_2\}=\sigma(\mathrm{d}f_1\wedge\mathrm{d}f_2)$ for all $f_1,f_2\in\O_M$; it is called the \textit{Poisson bivector field}. Contracting $\sigma$ with cotangent vectors yields a vector bundle morphism $\sigma^{\vee}:T^*M\longrightarrow TM$, whose image is an integrable, singular distribution. One finds that $\sigma$ is tangent to each integral leaf $L\subset M$ of this distribution. It thereby induces the Poisson bivector field $\sigma_L\in \mathrm{H}^0(L,\wedge^2(TL))$ of a Poisson structure on $L$. There exists a unique symplectic form $\omega_L\in\mathrm{H}^0(L,\wedge^2(T^*L))$ satisfying $(\omega_L^{\vee})^{-1}=-\sigma_L^{\vee}$. The symplectic manifold $(L,\omega_L)$ is called a \textit{symplectic leaf} of $M$. If $M$ is symplectic with symplectic form $\omega$, then $M$ is Poisson with Poisson bivector field $\sigma^{\vee}=-(\omega^{\vee})^{-1}$.
	
	The following is a recurring example. Suppose that $M=\g^*$ for a finite-dimensional real or complex Lie algebra $\g$. Note that $\mathrm{d}f_{\xi}\in\g$ for all $f\in\O_{\g^*}$ and $\xi\in\g^*$. The canonical Poisson structure on $\g^*$ is the following Poisson bracket $\{\cdot,\cdot\}$ on $\O_{\g^*}$:
	$$\{f_1,f_2\}(\xi)\coloneqq\xi([(\mathrm{d}f_1)_{\xi},(\mathrm{d}f_2)_{\xi}])$$ for all $f_1,f_2\in\O_{\g^*}$ and $\xi\in\g^*$. Note that the symplectic leaves of $\g^*$ are the coadjoint orbits of any connected Lie group integrating $\g$.
	
	\subsection{Poisson transversals and pre-Poisson submanifolds}\label{Subsection: Pre-Poisson submanifolds} Suppose that $M$ is Poisson. Write $L_p\subset M$ for the symplectic leaf of $M$ containing $p\in M$. A submanifold or smooth, locally closed subvariety $S\subset M$ is called a \textit{Poisson transversal} if the following hold for all $p\in S$: $T_pM=T_p(L_p)+T_pS$ and $L_p\cap S$ is a symplectic submanifold of $L_p$. It turns out that $S$ inherits a canonical Poisson structure, with symplectic leaves given by the connected components of the intersections of $S$ with symplectic leaves of $M$ \cite[Lemma 3]{fre-mar:17}. If $M$ is symplectic, then a submanifold (resp. smooth, locally closed subvariety) $S\subset M$ is a Poisson transversal if and only if $S$ is a symplectic submanifold (resp. subvariety) of $M$. One also has the following; see \cite[Lemma 7]{fre-mar:17}.
	
	\begin{proposition}\label{Proposition: Basic transversal}
		Suppose that $M$ and $N$ are Poisson. Let $S\subset N$ be a Poisson transversal. If $\mu:M\longrightarrow N$ is a Poisson morphism, then $\mu^{-1}(S)\subset M$ is a Poisson transversal.
	\end{proposition}
	
	Suppose that $M$ is Poisson. Consider the binary operation $[\cdot,\cdot]:\Omega^1(X)\times\Omega^1(X)\longrightarrow\Omega^1(X)$ defined by
	$$[\alpha,\beta]\coloneqq\mathcal{L}_{\sigma^{\vee}(\alpha)}(\beta)-\mathcal{L}_{\sigma^{\vee}(\beta)}(\alpha)-\mathrm{d}(\sigma(\alpha\wedge\beta))$$ for all $\alpha,\beta\in\Omega^1(X)$. One finds that $[\cdot,\cdot]$ renders $T^*M$ a Lie algebroid with anchor map $\sigma^{\vee}$. Given a submanifold (resp. complex submanifold, smooth, locally closed subvariety) $S\subset M$, let $(TS)^{\circ}\subset (T^*M)\big\vert_{S}$ denote conormal bundle of $S$ in $M$. The following definition has obvious analogues in the holomorphic and complex-algebraic categories, and first arises in works of Cattaneo--Zambon \cite{cat-zam:07,cat-zam:09}.
	
	\begin{definition}[Cattaneo--Zambon]\label{Definition: Pre-Poisson submanifold}
		Suppose that $M$ is Poisson with bivector field $\sigma\in\mathrm{H}^0(M,\wedge^2(TM))$. A submanifold $S\subset M$ is called \textit{pre-Poisson} if $$L_S\coloneqq (\sigma^{\vee})^{-1}(TS)\cap (TS)^{\circ}\longrightarrow S$$ has constant fiber dimension.	
	\end{definition}
	
	There is a precise sense in which generic submanifolds of a Poisson manifold are pre-Poisson \cite[Proposition 2.6]{cro-may:22}. Examples of pre-Poisson submanifolds include coisotropic submanifolds. A submanifold $S$ of a Poisson manifold $M$ is pre-Poisson if and only if $L_S$ is a Lie subalgebroid of $T^*M$ \cite[Proposition 3.6]{cat-zam:09}.
	
	\subsection{Hamiltonian actions}\label{Subsection: Hamiltonian actions} Let $(M,\omega)$ be a symplectic manifold (resp. holomorphic symplectic manifold, complex symplectic variety). Suppose that $M$ carries a smooth (resp. holomorphic, algebraic) action of a Lie group (resp. complex Lie group, complex algebraic group) $G$. As above, real-geometric, complex-geometric, and algebro-geometric adjectives are used only in cases where ambiguity among the smooth, holomorphic, and complex-algebraic categories may arise. 
	
	Let $\g$ denote the Lie algebra of $G$.
	The $G$-action on $M$ is called \textit{Hamiltonian} if $\omega$ is $G$-invariant and there exists a $G$-equivariant morphism $\mu:M\longrightarrow\g^*$ satisfying $$\omega(x_M,\cdot)=-\mathrm{d}(\mu^x)$$ for all $x\in\g$, where $\mu^x\in\mathcal{O}_M$ is the pointwise pairing of $\mu$ with $x$. In this case, $\mu$ is called a \textit{moment map}. A \textit{Hamiltonian $G$-space} (resp. \textit{holomorphic Hamiltonian $G$-space}, \textit{ Hamiltonian $G$-variety}) is the data of a symplectic manifold (resp. holomorphic symplectic manifold, complex symplectic variety) $M$, a Hamiltonian $G$-action on $M$, and a prescribed moment map $\mu:M\longrightarrow\g^*$. Note that $\mu$ is necessarily a Poisson morphism \cite[Lemma 1.4.2]{chr-gin}. 
	
	There is a simple generalization of the above if $(M,\omega)$ is replaced with a Poisson manifold (resp. holomorphic Poisson manifold, smooth complex Poisson variety) $(M,\sigma)$, where $\sigma$ is the Poisson bivector field on $M$. The action of $G$ on $M$ is called \textit{Hamiltonian} if $\sigma$ is $G$-invariant and there exists a $G$-equivariant morphism $\mu:M\longrightarrow\g^*$, called a moment map, satisfying $$\sigma^{\vee}(\mathrm{d}(\mu^x))=x_M$$ for all $x\in\g$. A \textit{Poisson Hamiltonian $G$-space} (resp. \textit{holomorphic Poisson Hamiltonian $G$-space}, \textit{Poisson Hamiltonian $G$-variety}) is the data of a Poisson manifold (resp. holomorphic Poisson manifold, smooth complex Poisson variety) $M$, a Hamiltonian $G$-action on $M$, and a prescribed moment map $\mu:M\longrightarrow\g^*$.
	
	The following is a special case of the above. Suppose that $M=T^*N$ for a smooth manifold (resp. complex manifold, smooth complex algebraic variety) $N$ with a smooth (resp. holomorphic, algebraic) action of a Lie group (resp. complex Lie group, complex algebraic group) $G$. The $G$-action on $N$ lifts canonically to a smooth (resp. holomorphic, algebraic) action of $G$ on $T^*N$, where the symplectic form on $T^*N$ is the negated exterior derivative of the tautological one-form. This latter action is Hamiltonian with moment map $\mu:T^*N\longrightarrow\g^*$, defined by
	$$\mu(p,\phi)(x)\coloneqq-\phi((x_N)_p),\quad (p,\phi)\in T^*N,\text{ }x\in\g.$$
	
	We further specialize the above as follows. Let $G$ be a Lie group, complex Lie group, or complex algebraic group with Lie algebra $\g$. Consider the following action of $G\times G$ on $G$: $$(g_1,g_2)h\coloneqq g_1hg_2^{-1},\quad (g_1,g_2)\in G\times G,\text{ }h\in G.$$ There is an induced Hamiltonian action of $G\times G$ on $T^*G$. Using the left trivialization to identify $T^*G$ with $G\times\g^*$, this induced action becomes
	$$(g_1,g_2)(h,\xi)\coloneqq (g_1hg_2^{-1},\Ad_{g_2}^*(\xi)),\quad (g_1,g_2)\in G\times G,\text{ }(h,\xi)\in G\times\g^*=T^*G.$$ One finds that $$\mu:T^*G=G\times\g^*\longrightarrow\g^*\times\g^*=(\g\times\g)^*,\quad (g,\xi)\mapsto(\mathrm{Ad}_g^*(\xi),-\xi),\quad (g,\xi)\in G\times\g^*=T^*G$$ is a moment map.
	
	\subsection{Hamiltonian reduction}\label{Subsection: Hamiltonian reduction}
	Suppose that $G$ is a Lie group (resp. complex Lie group) with Lie algebra $\g$. Let $(M,\omega)$ be a Hamiltonian $G$-space (resp. holomorphic Hamiltonian $G$-space) with moment map $\mu:M\longrightarrow\g^*$. Given $\xi\in\g^*$, the subset $\mu^{-1}(\xi)\subset M$ is $G_{\xi}$-invariant. It turns out that $\xi$ is a regular value of $\mu$ if and only if the action of $G_{\xi}$ on $\mu^{-1}(\xi)$ is locally free. Assume that these equivalent conditions are satisfied, and that the quotient topological space $\mu^{-1}(\xi)/G_{\xi}$ carries a smooth (resp. complex) manifold structure for which the quotient map $\pi:\mu^{-1}(\xi)\longrightarrow\mu^{-1}(\xi)/G_{\xi}$ is a submersion; this structure is necessarily unique. Write $j:\mu^{-1}(\xi)\longrightarrow M$ for the inclusion. There exists a unique smooth (resp. holomorphic) symplectic form $\omega_{\xi}$ on $\mu^{-1}(\xi)/G_{\xi}$ satisfying $\pi^*\omega_{\xi}=j^*\omega$. The symplectic (resp. holomorphic symplectic) manifold $$M\sll{\xi}G\coloneqq(\mu^{-1}(\xi)/G_{\xi},\omega_{\xi})$$ is called the \textit{Hamiltonian reduction} of $M$ by $G$ at level $\xi$. This construction is largely due to Marsden--Weinstein \cite{mar-wei:74} and Meyer \cite{mey:73}.
	
	We now discuss a complex algebro-geometric counterpart of the above. To this end, let $G$ be a complex algebraic group with Lie algebra $\g$. Suppose that $(M,\omega)$ is a Hamiltonian $G$-variety with moment map $\mu:M\longrightarrow\g^*$. Fix $\xi\in\g^*$, and assume that the action of $G_{\xi}$ on $\mu^{-1}(\xi)$ is free. It follows that the closed subscheme $\mu^{-1}(\xi)\subset M$ is smooth. We also assume that a smooth geometric quotient variety $\mu^{-1}(\xi)/G_{\xi}$ exists. Consider the inclusion $j:\mu^{-1}(\xi)\longrightarrow M$ and quotient morphism $\pi:\mu^{-1}(\xi)\longrightarrow\mu^{-1}(\xi)/G_{\xi}$. Our final assumption is that $\mathrm{d}\pi_p:T_p(\mu^{-1}(\xi))\longrightarrow T_{\pi(p)}(\mu^{-1}(\xi)/G_{\xi})$ is surjective for all $p\in\mu^{-1}(\xi)$. There exists a unique algebraic symplectic form $\omega_{\xi}$ on $\mu^{-1}(\xi)/G_{\xi}$ such that $\pi^*\omega_{\xi}=j^*\omega$ \cite[Theorem C(iii)]{cro-may:22}. As above, we write $M\sll{\xi}G\coloneqq(\mu^{-1}(\xi)/G_{\xi},\omega_{\xi})$ for the Hamiltonian reduction of $M$ by $G$ at level $\xi$.
	
	\section{Slices for compact and complex reductive group actions}\label{sec:slices}
	This section introduces \textit{Poisson slices} and \textit{symplectic slices}, reviews the principal slice theorem for Hamiltonian actions of compact Lie groups, and sets the stage for analogues for complex reductive algebraic group actions. Emphasis is placed on specific differences between Hamiltonian actions of compact Lie groups and those of complex reductive algebraic groups.
	
	\subsection{Generalities on slices}
	Suppose that a Lie group $G$ acts smoothly on a smooth manifold $M$. The following definition of a \textit{slice} arises in this context; cf. \cite[Definition 3.6]{ler-mei-tol-woo} and \cite[Definition 1.1]{sja:95}.
	
	\begin{definition}\label{Definition: Slice compact} A \textit{slice} at $p\in M$ is a submanifold $S\subset M$ such that $p\in S$ and the following conditions are satisfied:
		\begin{itemize}
			\item[\textup{(i)}] $S$ is invariant under $G_p\subset G$;
			\item[\textup{(ii)}] $GS$ is open in $M$; 
			\item[\textup{(iii)}] the map $G\times_{G_p} S \longrightarrow GS$ given by $[g:q] \mapsto gq$ is a diffeomorphism, where $G\times_{G_p} S \longrightarrow GS$ is the quotient of $G\times S$ by the following $G_p$-action: $g(h,x)=(hg^{-1},gx)$ for $g\in G_p$ and $(h,x)\in G\times S$.
		\end{itemize}
	\end{definition}
	
	We wish to define slices in Poisson Hamiltonian $G$-spaces, and call them \textit{Poisson slices}. One approach would be to simply add appropriate Poisson-geometric conditions to Definition \ref{Definition: Slice compact}, e.g. that $S$ should be a Poisson transversal. At the same time, retaining Parts (i) and (iii) of this definition would preclude some \textit{Slodowy slices} from being Poisson slices in any reasonable way; see Section \ref{Subsection: Slodowy slices} for details on Slodowy slices, and Remark \ref{Remark: Slice issue} for an explanation of the issue that retaining Parts (i) and (iii) would create. These last few sentences inform our definition of a \textit{Poisson slice}.

	\subsection{Poisson and symplectic slices}\label{Subsection: Poisson slices new}
	Let $G$ be a Lie group (resp. complex Lie group, complex algebraic group). Suppose that $M$ is a Poisson Hamiltonian $G$-space (resp. holomorphic Poisson Hamiltonian $G$-space, Poisson Hamiltonian $G$-variety), as defined in Section \ref{Subsection: Hamiltonian actions}. Consider a submanifold (resp. complex submanifold, smooth locally closed subvariety) $S\subset M$.
	
	\begin{definition}\label{Definition: Slice}
		We call $S$ a \textit{Poisson slice} if it satisfies the following conditions:
		\begin{itemize}
			\item[\textup{(i)}] $S$ is a Poisson transversal in $M$;
			\item[\textup{(ii)}] $S$ is transverse to the $G$-orbits in $M$, i.e. $T_pM=T_pS+T_p(Gp)$ for all $p\in S$.
		\end{itemize}
		If $M$ is symplectic and $S$ is a Poisson slice, then $S$ is called a \textit{symplectic slice}.
	\end{definition} 
	
	\begin{remark}
		Suppose that $M$ is symplectic. It is straightforward to verify that $S$ is a symplectic slice if and only if $S$ is a symplectic submanifold (resp. holomorphic symplectic submanifold, symplectic subvariety) of $M$ with the property of being transverse to $G$-orbits.
	\end{remark}
	
	This definition generalizes that of the same term in \cite{cro-roe:22}, where $G$ is assumed to be a connected complex semisimple algebraic group, and Poisson slices are declared to be pre-images of Slodowy slices under moment maps for Hamiltonian $G$-actions. Before witnessing more examples of Poisson slices, we derive the following two properties.
	
	\begin{proposition}
		If $S\subset M$ is a Poisson slice, then the saturation $GS$ is open in $M$. 
	\end{proposition}
	
	\begin{proof}
		Consider the morphism
		$$\varphi:G\times S\longrightarrow M,\quad (g,p)\mapsto gp.$$ Since $S$ is transverse to the $G$-orbits in $M$, $\varphi$ is a submersion. We conclude that the image $GS$ of $\varphi$ is open if $M$ is a Poisson Hamiltonian $G$-space or holomorphic Poisson Hamiltonian $G$-space. On the other hand, suppose that $M$ is a Poisson Hamiltonian $G$-variety. The image $GS$ of the variety morphism $\varphi$ is necessarily constructible. One also knows that a constructible subset of an algebraic variety is open in the Euclidean topology if and only if it is open in the Zariski topology \cite[Section XII, Subsection 2, Corollary 2.3]{sga1}. This implies that $GS$ is open in the Zariski topology of $M$.
	\end{proof}
	
	\begin{proposition}\label{Proposition: Pullback}
		Let $M$ and $N$ be Poisson Hamiltonian $G$-spaces, holomorphic Poisson Hamiltonian $G$-spaces, or Poisson Hamilonian $G$-varieties. Suppose that $S\subset N$ is a Poisson slice. If $f:M\longrightarrow N$ is a $G$-equivariant Poisson morphism, then $f^{-1}(S)\subset M$ is a Poisson slice.
	\end{proposition}
	
	\begin{proof}
		By Proposition \ref{Proposition: Basic transversal}, $f^{-1}(S)$ is a Poisson transversal in $M$. It remains to prove that $f^{-1}(S)$ is transverse to the $G$-orbits in $M$. Suppose that $u\in T_pM$ for $p\in f^{-1}(S)$. Since $T_{f(p)}S$ is transverse to $G$-orbits, we may write $\mathrm{d}f_p(u)=v+w$, where $v\in T_{f(p)}S$ and $w=(x_N)_{f(p)}$ for some $x\in\g$; see Section \ref{Subsection: Actions, orbits, and stabilizers} for a definition of $x_N$. The equivariance of $f$ implies that $\mathrm{d}f_p((x_M)_p)=(x_N)_{f(p)}$. It follows that $\mathrm{d}f_p(u-(x_M)_p)=v$. We conclude that $u-(x_M)_p$ belongs to $(\mathrm{d}f_p)^{-1}(T_{f(p)}S)=T_p(f^{-1}(S))$. In other words, $u$ is the sum of an element in $T_{p}(f^{-1}(S))$ and an element in $T_p(Gp)$.
	\end{proof}
	
	\subsection{The principal slice theorem for compact Lie group actions}\label{Subsection: The principal slice theorem for compact group actions}
	Let $K$ be a compact connected Lie group with Lie algebra $\k$. As in Section \ref{Subsection: Compact groups}, we choose a maximal torus $T\subset K$ with Lie algebra $\t\subset\k$. Let us also choose a set of positive roots. Recall that $\t^*_{+}\subset\k^*$ denotes the fundamental Weyl chamber associated to this choice. Let us also recall that $K_{\sigma}$ denotes the $K$-stabilizer of any $x\in\sigma$, where $\sigma\subset\t^*_{+}$ is an open face. Versions of the following result are due to Condevaux--Dazord--Molino \cite[Section II.2.1]{con-daz-mol:88}, Guillemin--Sternberg \cite[Theorem 26.7]{gui-ste:90}, Hilgert--Neeb--Plank \cite[Lemmas 6.7--6.9]{hil-nee-pla:94}, and Lerman--Meinrenken--Tolman--Woodward \cite[Theorem 3.1]{ler-mei-tol-woo}; the formulation we give here is essentially that of \cite[Theorem 3.1]{ler-mei-tol-woo}.
	
	\begin{theorem}\label{Theorem: Slice theorem}
		Consider a non-empty, connected Hamiltonian $K$-space $M$ with moment map $\mu:M\longrightarrow\k^*$.
		\begin{itemize}
			\item[\textup{(i)}] There exists a unique open face $\sigma_M\subset\t^*_{+}$ for which $\mu(M)\cap\sigma_M$ is dense in $\mu(M)\cap\t^*_{+}$.
			\item[\textup{(ii)}] The inverse image $\mu^{-1}(\sigma_M)\subset M$ is a connected, $T$-invariant, symplectic submanifold.
			\item[\textup{(iii)}] The action of $[K_{\sigma_M},K_{\sigma_M}]$ on $\mu^{-1}(\sigma_M)$ is trivial.
			\item[\textup{(iv)}] The restriction $$\mu\big\vert_{\mu^{-1}(\sigma_M)}:\mu^{-1}(\sigma_M)\longrightarrow\t^*$$ is a moment map for the Hamiltonian action of $T$ on $\mu^{-1}(\sigma_M)$.
			\item[\textup{(v)}] The saturation $K\mu^{-1}(\sigma_M)\subset M$ is open and dense. 
		\end{itemize}
	\end{theorem}
	
	\begin{remark}
		Certain parts of Theorem \ref{Theorem: Slice theorem} warrant further discussion. 
		\begin{itemize}
			\item[\textup{(i)}] Part (iii) does not appear in the statement of \cite[Theorem 3.1]{ler-mei-tol-woo}; it instead appears in the proof of the aforementioned theorem.
			\item[\textup{(ii)}] The openness assertion in Part (v) is not included in the statement of \cite[Theorem 3.1]{ler-mei-tol-woo}; it follows easily from $\mu(M)\cap\sigma_M$ being dense in $\mu(M)\cap\t^*_{+}$, and $K\sigma_M$ being locally closed in $\k^*$.   
		\end{itemize}
		
	\end{remark}
	
	As per \cite[Section 3]{ler-mei-tol-woo}, we adopt the following definition.
	
	\begin{definition}
		Consider a non-empty, connected Hamiltonian $K$-space $M$ with moment map $\mu:M\longrightarrow\k^*$. The open face $\sigma_M\subset\t^*_{+}$ and Hamiltonian $T$-space $\mu^{-1}(\sigma_M)$ are called the \textit{principal face} and \textit{principal slice}, respectively.
	\end{definition}
	
	It is straightforward to check that $\sigma_M$ is a Poisson slice in the Poisson submanifold $K\sigma_M\subset\k^*$. By combining this fact with Proposition \ref{Proposition: Pullback} and Theorem \ref{Theorem: Slice theorem}, one can verify that $\mu^{-1}(\sigma_M)$ is a symplectic slice in $M$.
	
	\subsection{A partial reformulation via orbit-type strata}\label{Subsection: Orbit-type strata} Retain the setup of Section \ref{Subsection: The principal slice theorem for compact group actions}. Recall the discussion of orbit-type strata in Section \ref{Subsection: Actions, orbits, and stabilizers}. As discussed in Section \ref{Subsection: Compact groups}, the association $\sigma\mapsto K\sigma$ is a bijection from the set of open faces $\sigma\subset\t^*_{+}$ to the set of orbit-type strata for the coadjoint action of $K$. It follows that orbit-type strata are submanifolds of $\k^*$. With this in mind, Parts (i), (iii), and (v) of Theorem \ref{Theorem: Slice theorem} may be reformulated as follows.
	
	\begin{theorem}\label{Theorem: New slice theorem}
		Consider a non-empty, connected Hamiltonian $K$-space $M$ with moment map $\mu:M\longrightarrow\k^*$. The following statements are true.
		\begin{itemize}
			\item[\textup{(i)}] There exists a unique orbit-type stratum $\Sigma_M\subset\k^*$ for which $\mu(M)\cap\Sigma_M$ is dense in $\mu(M)$.
			\item[\textup{(ii)}] For all $\xi\in\Sigma_M$, the action of $[K_{\xi},K_{\xi}]$ on $\mu^{-1}(\xi)$ is trivial.
			\item[\textup{(iii)}] The inverse image $\mu^{-1}(\Sigma_M)\subset M$ is open and dense.
		\end{itemize}
	\end{theorem}
	
	\begin{definition}
		Retain the setup of Theorem \ref{Theorem: New slice theorem}. The orbit-type stratum $\Sigma_M\subset\k^*$ and inverse image $\mu^{-1}(\Sigma_M)\subset M$ are called the \textit{principal stratum} and \textit{principal open subset}, respectively.
	\end{definition}
	
	In Section \ref{Section: Decomposition classes in Poisson geometry}, we explain that decomposition classes are natural analogues of orbit-type strata for connected complex reductive algebraic groups.
	
	\subsection{Natural slices}\label{Subsection: Slices for compact group actions}
	Continue with the setup of Section \ref{Subsection: Orbit-type strata}. Given an open face $\sigma\subset\t^*_{+}$, let $$\mathrm{star}(\sigma)\coloneqq\bigcup_{\overline{\tau}\supset \sigma}\tau$$ be the union of the open faces $\tau\subset\mathfrak{t}^*_{+}$ that contain $\sigma$ in their closures. For $\xi\in\t^{*}_{+}$, it turns out that
	\begin{equation}\label{Equation: Star}\xi\in\mathrm{star}(\sigma)\Longleftrightarrow K_{\xi}\subset K_{\sigma};\end{equation} see \cite[Remark 3.7]{ler-mei-tol-woo}. The following object is introduced in \cite{ler-mei-tol-woo}.
	
	\begin{definition}
		The \textit{natural slice} associated to an open face $\sigma\subset\t^*_{+}$ is defined to be the $K_{\sigma}$-saturation $\S_{\sigma}$ of $\mathrm{star}(\sigma)$, i.e.
		$\S_{\sigma}\coloneqq K_{\sigma}\mathrm{star}(\sigma)\subset\k^*$.
	\end{definition}
	
	We now prove that $\S_{\sigma}$ is a slice in the sense of Definition \ref{Definition: Slice compact}. While we suspect that this is well-known to experts, we are unable to find an explicit proof in the literature. 
	
	\begin{lemma} 
		\label{Lemma : K-slice}
		If $\sigma\subset\t^*_{+}$ an open face, then $\S_{\sigma}\subset\k^*$ is a slice at any point of $\sigma$, with respect to the coadjoint action of $K$. 
	\end{lemma}
	\begin{proof}
		Suppose that $\xi\in\sigma$. Observe that $\S_{\sigma}$ contains $\xi$ and is invariant under $K_{\sigma}=K_{\xi}$. It is a also clear that 
		\begin{align*}
			K\S_{\sigma} = \bigcup_{\overline{\tau}\supset \sigma} 
			K\tau.
		\end{align*}
		This implies that $K\S_{\sigma} = \k^*\setminus Y$, where 
		$Y$ is the union of $K\tau$ such that 
		$\sigma\not\subset \overline{\tau}$.
		If 
		$\sigma\not\subset \overline{\tau}$, 
		then  
		$\sigma\not\subset \overline{\tau'}$
		for any face $\tau'\subset \ol\tau$. 
		Thus for $K\tau\subset Y$,  
		$$\ol{K\tau} = K\ol{\tau}=\bigcup_{\tau' \subset \ol\tau} K{\tau'} \subset Y,$$
		which implies that $Y$ is closed.   Thus $K\S_{\sigma}$ is open, as desired. 
		
		We now consider the smooth surjection $\varphi:K\times_{K_\sigma} \S_{\sigma} \longrightarrow K\S_{\sigma}$ given by $\varphi([k:\xi])=\Ad_k (\xi)$ for all $k\in K$ and $\xi\in \S_{\sigma}$. It remains only to prove that $\varphi$ is a diffeomorphism. To show that $\varphi$ is injective, suppose that 
		$\Ad_k(\xi)=\Ad_{k'}(\xi')$ for $[k:\xi],[k':\xi']\in K\times_{K_{\sigma}}\S_{\sigma}$. Recall as well that $\t^{*}_{+}$ is a fundamental domain for the coadjoint action of $K$ on $\k^*$. It follows that $\xi$ and $\xi'$ are $K$-conjugate to the same element $\eta\in\mathrm{star}(\sigma)\subset\t^{*}_{+}$. Since $\xi,\xi'\in\S_{\sigma}$, we must have $\xi=\Ad_{\ell}(\eta)$ and $\xi'=\Ad_{\ell'}(\eta)$ for some $\ell,\ell'\in K_{\sigma}$. We conclude that $\Ad_{k\ell}(\eta)=\Ad_{k'\ell'}(\eta)$, i.e. $(k\ell)^{-1}(k'\ell')\in K_{\eta}\subset K_{\sigma}$, where the inclusion follows from \eqref{Equation: Star}. This allows us to write $k'=km$ for some $m\in K_{\sigma}$, yielding $[k':\xi']=[km:\Ad_{m^{-1}}(\xi)]=[k:\xi]$. In other words, $\varphi$ is injective.
		
		The previous paragraph establishes that $\varphi$ is a smooth bijection. We are therefore reduced to proving that the differential of $\varphi$ is a vector space isomorphism at every point in $K\times_{K_{\sigma}}\S_{\sigma}$. This is easily seen to hold if and only if $\mathrm{d}\varphi_{[e:\xi]}:T_{[e:\xi]}(K\times_{K_{\sigma}}\S_{\sigma})\longrightarrow T_{\xi}(K\S_{\sigma})=\k^*$ is a vector space isomorphism for all $\xi\in\mathrm{star}(\sigma)$.
		
		Choose a $K$-invariant inner product $\langle\cdot,\cdot\rangle:\k\otimes_{\mathbb{R}}\k\longrightarrow\mathbb{R}$, and use it to $K$-equivariantly identify $\k^*$ with $\k$. We may regard $\sigma$ as an open face of $\t_{+}\subset\k$, where $\t_{+}$ is as defined in Section \ref{Subsection: Compact groups}. It also follows that $\mathrm{star}(\sigma)$ and $\S_{\sigma}$ are subsets of $\k$. Given $x\in\mathrm{star}(\sigma)$, our task is to prove that $\mathrm{d}\varphi_{[e:x]}:T_{[e:x]}(K\times_{K_{\sigma}}\S_{\sigma})\longrightarrow \k$ is a vector space isomorphism. We first note that $\S_{\sigma}$ is open in $\k_{\sigma}$, the Lie algebra of $K_{\sigma}$ \cite[Remark 3.7]{ler-mei-tol-woo}. This immediately implies that the domain and codomain of $\mathrm{d}\varphi_{[e:x]}$ have the same dimension. One is thereby reduced to proving that $\mathrm{d}\varphi_{[e:x]}$ is surjective. At the same time, the image of $\mathrm{d}\varphi_{[e:x]}$ is $T_x(\S_{\sigma})+[\k,x]=\k_{\sigma}+[\k,x]\subset\k$. The $K$-invariance of $\langle\cdot,\cdot\rangle$ implies that $\k=\k_{x}\oplus[\k,x]$, and \eqref{Equation: Star} yields $\k_x\subset\k_{\sigma}$. Hence $\k=\k_x\oplus[\k,x]\subset\k_{\sigma}+[\k,x]\subset\k$, i.e. $\k_{\sigma}+[\k,x]=\k$. This completes the proof.
	\end{proof}
	
	\begin{proposition}
		If $\sigma\subset\t^*_{+}$ is an open face, then $\S_{\sigma}$ is a Poisson slice with respect to the Poisson Hamiltonian $K$-space structure on $\k^*$.
	\end{proposition}
	
	\begin{proof}
		Lemma \ref{Lemma : K-slice} implies that $\S_{\sigma}$ is transverse to the coadjoint orbits of $K$. Our task is thereby reduced to proving that the submanifold $\S_{\sigma}\cap (K\xi)\subset K\xi$ is symplectic for all $\xi\in\S_{\sigma}$. It suffices to prove this assertion for $\xi\in\tau$, where $\tau\subset\t^*_{+}$ is an open face with $\sigma\subset\overline{\tau}$. Definition \ref{Definition: Slice compact} and Lemma \ref{Lemma : K-slice} imply that $\S_{\sigma}\cap (K\xi)=K_{\sigma}\xi$ under these hypotheses. This reduces us to proving that the submanifold $K_{\sigma}\xi\subset K\xi$ is symplectic.
		
		As in the proof of Lemma \ref{Lemma : K-slice}, we choose a $K$-invariant inner product $\langle\cdot,\cdot\rangle:\k\otimes_{\mathbb{R}}\k\longrightarrow\mathbb{R}$, and use it to $K$-equivariantly identify $\k^*$ with $\k$. This allows us to regard $\sigma$ as an open face of $\t_{+}\subset\k$, where $\t_{+}$ is as defined in Section \ref{Subsection: Compact groups}. Another consequence is that the symplectic structures on coadjoint orbits now amount to symplectic structures on adjoint orbits. Given $x\in\k$, the symplectic form $\omega$ on $Kx\subset\k$ satisfies
		$$\omega_x([y,x],[z,x])=\langle x,[y,z]\rangle$$ for all $y,z\in\k$. Our task is to prove the following: the submanifold $K_{\sigma}x\subset Kx$ is symplectic for all $x\in\tau$, where $\tau\subset\t_{+}$ is an open face with $\sigma\subset\overline{\tau}$.
		
		Note that $\langle\cdot,\cdot\rangle:\k\otimes_{\mathbb{R}}\k\longrightarrow\mathbb{R}$ restricts to a $K_{\sigma}$-invariant inner product on $\k_{\sigma}\coloneqq\mathrm{Lie}(K_{\sigma})$. In analogy with the above, adjoint orbits of $K_{\sigma}$ are symplectic manifolds. Given $x\in\k_{\sigma}$, the symplectic form $\omega$ on $K_{\sigma}x\subset\k$ satisfies
		$$\omega_x([y,x],[z,x])=\langle x,[y,z]\rangle$$ for all $y,z\in\k_{\sigma}$. We conclude that the above-described symplectic form on $Kx\subset\k$ restricts to this symplectic form on $K_{\sigma}x$. In particular, $K_{\sigma}x$ is a symplectic submanifold of $Kx$. By the last line of the previous paragraph, it suffices to observe that $\tau\subset\k_{\sigma}$ for all open faces $\tau\subset\t_{+}$ with $\sigma\subset\overline{\tau}$.
	\end{proof}
	
	\subsection{A first contrast between the compact and complex reductive cases}\label{Subsection: Fundamentals}
	Continue with the notation of Section \ref{Subsection: Slices for compact group actions}. Recall that $\t^*_{+}$ is a fundamental domain for the coadjoint action of $K$ on $\k^*$. This fact is implied by the existence of a continuous map $\pi:\k^*\longrightarrow\t^*_{+}$ satisfying $\pi^{-1}(\xi)=K\xi$ for all $\xi\in\t^*_{+}$.
	
	We now suppose that $G$ is a connected complex reductive algebraic group with Lie algebra $\g$. In light of the above, one wonders if there exist a topological subspace $S\subset\g^*$ and continuous map $\pi:\g^*\longrightarrow S$ that satisfy $\pi^{-1}(\xi)=G\xi$ for all $\xi\in S$. This turns out not to occur in general.
	
	\begin{proposition}\label{Proposition: No analogue}
		Let $G$ be a connected complex reductive algebraic group with Lie algebra $\g$. Equip $\g^*$ with any topology in which singletons are closed. There exist a topological subspace $S\subset\g^*$ and continuous map $\pi:\g^*\longrightarrow S$ satisfying $\pi^{-1}(\xi)=G\xi$ for all $\xi\in S$ if and only if $G$ is a torus.
	\end{proposition}
	
	\begin{proof}
		Suppose that $G$ is a torus. We may take $S=\g^*$ and $\pi$ to be the identity map on $\g^*$. 
		
		To establish the converse, assume that $G$ is not a torus. It follows that the semisimple Lie algebra $[\g,\g]$ is non-zero. Choose a non-zero nilpotent element $e\in[\g,\g]$, noting that the adjoint orbit $Ge\subset\g$ is not closed \cite[Theorem 2.3.1]{col-mcg:93}. Our $G$-module isomorphism between $\g$ and $\g^*$ identifies $Ge$ with a non-closed coadjoint orbit of $G$. On the other hand, note that singletons in any topological subspace $S\subset\g^*$ are closed in that subspace. If there existed a topological subspace $S\subset\g^*$ and continuous map $\pi:\g^*\longrightarrow S$ with $\pi^{-1}(\xi)=G\xi$ for all $\xi\in S$, the previous sentence would force all coadjoint orbits of $G$ to be closed. This would contradict the fact that $Ge$ is not closed, completing the proof.
	\end{proof}
	
	The Euclidean and Zariski topologies on $\g^*$ have the property that singletons are closed. It follows that no direct analogue of a fundamental Weyl chamber exists for the coadjoint action of a connected complex reductive algebraic group, regardless of whether one works in the complex-algebraic or holomorphic category. This is a first indication of important distinctions to be made between slice theorems for Hamiltonian actions of compact Lie groups and those of complex reductive algebraic groups. On the other hand, Proposition \ref{Proposition: True statements} gives us a a fundamental domain for ``most of" a complex reductive Lie algebra. This is detailed in Section \ref{Subsection: Slodowy slices}.
	
	\subsection{Slodowy slices in complex reductive Lie algebras}\label{Subsection: Slodowy slices}
	Let $G$ be a connected complex reductive algebraic group with Lie algebra $\g$. Recall that $\mathfrak{T}=(e,h,f)\in\g^{\times 3}$ is called an \textit{$\sl_2$-triple} if $[e,f]=h$, $[h,e]=2e$, and $[h,f]=-2f$. One calls $$\S_{\mathfrak{T}}\coloneqq e+\g_f\subset\g$$ the \textit{Slodowy slice} associated to $\mathfrak{T}$. Note that $e,h,f\in[\g,\g]$. 
	
	Each Slodowy slice turns out to carry a distinguished  $\mathbb{C}^{\times}$-action. To define it, let $\mathfrak{T}=(e,h,f)\in\g^{\times 3}$ be an $\sl_2$-triple. There exists a unique Lie algebra morphism $\phi_{\mathfrak{T}}:\sl_2\longrightarrow\g$ satisfying
	$$\phi_{\mathfrak{T}}\left(\begin{bmatrix}0 & 1\\ 0 & 0 \end{bmatrix}\right)=e,\quad \phi_{\mathfrak{T}}\left(\begin{bmatrix}1 & 0\\ 0 & -1 \end{bmatrix}\right)=h,\quad\text{and}\quad \phi_{\mathfrak{T}}\left(\begin{bmatrix}0 & 0\\ 1 & 0 \end{bmatrix}\right)=f.$$
	Since $\operatorname{SL}_2$ is simply-connected, there exists a unique algebraic group morphism $\varphi_{\mathfrak{T}}:\operatorname{SL}_2\longrightarrow G$ with differential at $e\in\operatorname{SL}_2$ equal to $\phi_{\mathfrak{T}}$. One may consider the cocharacter $$\lambda_{\mathfrak{T}}:\mathbb{C}^{\times}\longrightarrow G,\quad t\mapsto\varphi_{\mathfrak{T}}\left(\begin{bmatrix}t^{-1} & 0\\ 0 & t \end{bmatrix}\right).$$ It follows that $\mathbb{C}^{\times}$ acts on $\g$ by
	\begin{equation}\label{Equation: Dilation action}tx\coloneqq t^{2}\Ad_{\lambda_{\mathfrak{T}}(t)}(x),\quad t\in\mathbb{C}^{\times},\text{ }x\in\g.\end{equation} Observe that $\mathcal{S}_{\mathfrak{T}}$ is invariant under this $\mathbb{C}^{\times}$-action. One also finds that $$\lim_{t\rightarrow 0}(tx)=e$$ for all $x\in\mathcal{S}_{\mathfrak{T}}$. These considerations give context for the following definition.
	
	\begin{definition}\label{Definition: Contracting action}
		If $\mathfrak{T}\in\g^{\times 3}$ is an $\sl_2$-triple, then the restriction of \eqref{Equation: Dilation action} to $\mathcal{S}_{\mathfrak{T}}$ is called the \textit{contracting action} of $\mathbb{C}^{\times}$ on $\mathcal{S}_{\mathfrak{T}}$.
	\end{definition}
	
	Following Kostant, we call an $\sl_2$-triple $\mathfrak{T}=(e,h,f)\in\g^{\times 3}$ \textit{principal} if $e,h,f\in\greg$; see Section \ref{Subsection: Complex reductive algebraic groups} for a definition of $\greg$. The following result is well-known if $\g$ is semisimple \cite{kos:59,kos:63}, and readily deducible in the case of our reductive Lie algebra $\g$. For the sake of completeness, we include a proof in the reductive case.
	
	\begin{proposition}\label{Proposition: True statements}
		The following statements are true.
		\begin{itemize}
			\item[\textup{(i)}] There exist principal $\sl_2$-triples $\mathfrak{T}\in\g^{\times 3}$.
			\item[\textup{(ii)}] If $\mathfrak{T}=(e,h,f)\in\g^{\times 3}$ is a principal $\sl_2$-triple, then $\mathcal{S}_{\mathfrak{T}}\subset\greg$.
			\item[\textup{(iii)}] If $\mathfrak{T}=(e,h,f)\in\g^{\times 3}$ is a principal $\sl_2$-triple, then $\mathcal{S}_{\mathfrak{T}}$ is a fundamental domain for the adjoint action of $G$ on $\greg$.
		\end{itemize}
	\end{proposition}
	
	\begin{proof}
		Parts (i), (ii), and (iii) are well-known if $\g$ is semisimple \cite{kos:59,kos:63}. Since $\greg=\z(\g)+[\g,\g]_{\mathsf{reg}}$, a principal $\sl_2$-triple in $[\g,\g]$ is a principal $\sl_2$-triple in $\g$. This establishes (i) in our proposition. 
		
		We now prove (ii). Recall that $e,h,f\in [\g,\g]$ and $\greg=\z(\g)+[\g,\g]_{\mathsf{reg}}$. Since (ii) is known to hold for semisimple Lie algebras, $e+[\g,\g]_f \subset [\g,\g]_{\mathsf{reg}}$. We conclude $ \mathcal{S}_{\mathfrak{T}}\subset \z(\g)+[\g,\g]_{\mathsf{reg}} =\g_{\mathsf{reg}}$,  establishing (ii). 
		
		To verify (iii), suppose that $x\in\g$. Write $x= z+y$ with $z\in \z(\g)$ and $y\in  [\g,\g]$, noting that 
		$Gx = z+ [G,G]y.$ Since $\g_{\mathsf{reg}}=\z(\g) +[\g,\g]_{\mathsf{reg}}$, a fundamental domain of $G$ acting on $\g_{\mathsf{reg}}$ is given by $\z(\g)$ plus a fundamental domain for $[G,G]$ acting on $[\g,\g]_{\mathsf{reg}}.$  The fact that (iii) holds for semisimple Lie algebras tells us that $e+ [\g,\g]_f$ is a fundamental domain for the adjoint action of $[G,G]$ acting on $[\g,\g]_{\mathsf{reg}}$. It follows that $\mathcal{S}_{\mathfrak{T}} = \z(\g) + e+ [\g,\g]_f$ is a fundamental domain for the adjoint action of $G$ on $\g_{\mathsf{reg}}$, establishing (iii). 
	\end{proof}
	
	Slodowy slices also turn out to play distinguished roles in algebraic and holomorphic Poisson geometry. To this end, recall the $G$-invariant, non-degenerate bilinear form $\langle\cdot,\cdot\rangle:\g\otimes\g\longrightarrow\mathbb{C}$ in Section \ref{Subsection: Complex reductive algebraic groups}. As discussed in that section, we may use $\langle\cdot,\cdot\rangle$ to obtain a $G$-module isomorphism between $\g$ and $\g^*$. Equip $\g$ with the Poisson structure for which this isomorphism is Poisson. If $\g$ is semisimple, then Slodowy slices in $\g$ are known to be Poisson transversals \cite{gan-gin:02}. This result generalizes to our setting of a reductive Lie algebra $\g$.  
	
	\begin{proposition}\label{Proposition: Slodowy Poisson transversal}
		If $\mathfrak{T}=(e,h,f)\in\g^{\times 3}$ is an $\sl_2$-triple, then $\S_{\mathfrak{T}}$ is a Poisson transversal in $\g$.
	\end{proposition}
	
	\begin{proof}
		Suppose that $x\in\S_{\mathfrak{T}}$. We may write $x=z+e+y$ for $y\in[\g,\g]_f$ and $z\in\z(\g)$. 
		Observe that $Gx=z+[G,G](e+y)$. On the other hand, $\S_{\mathfrak{T}}=\z(\g)+e+[\g,\g]_f$. It follows that $$T_x(Gx)+T_x(\S_{\mathfrak{T}})=\z(\g)+[[\g,\g],x]+[\g,\g]_f.$$ As Slodowy slices in semisimple Lie algebras are known to be Poisson transversals, we have $[[\g,\g],x]+[\g,\g]_f=[\g,\g]$. 
		We conclude that
		$$T_x(Gx)+T_x(\S_{\mathfrak{T}})=\z(\g)+[\g,\g]=\g,$$ as required.
		It is also clear that $$T_x(Gx)\cap T_x(\S_{\mathfrak{T}})=[[\g,\g],e+y]\cap(\z(\g)+[\g,\g]_f)=[[\g,\g],e+y]\cap[\g,\g]_f.$$ Noting again that Slodowy slices in semisimple Lie algebras are Poisson transversals, $[[\g,\g],e+y]\cap[\g,\g]_f$ is a symplectic subspace of $T_{e+y}([G,G](e+y))$. One also has $T_{e+y}([G,G](e+y))=T_x(Gx)$, implying that $T_x(Gx)\cap T_x(\S_{\mathfrak{T}})$ is a symplectic subspace of $T_x(Gx)$. In other words, $S_{\mathfrak{T}}$ is a Poisson transversal in $\g$. 
	\end{proof} 
	
	\begin{corollary}
		If $\mathfrak{T}=(e,h,f)\in\g^{\times 3}$ is an $\sl_2$-triple, then $\S_{\mathfrak{T}}$ is a Poisson slice in $\g$.
	\end{corollary}
	
	\begin{proof}
		Since the symplectic leaves of $\g$ are the adjoint orbits of $G$, a Poisson transversal in $\g$ is necessarily a Poisson slice. Proposition \ref{Proposition: Slodowy Poisson transversal} now implies that $\S_{\mathfrak{T}}$ is a Poisson slice in $\g$.
	\end{proof}
	
	\begin{remark}\label{Remark: Slice issue}
		Suppose that $\mathfrak{T}=(e,h,f)\in\g^{\times 3}$ is an $\sl_2$-triple. The Slodowy slice $\S_{\mathfrak{T}}$ is often called a \textit{transverse slice} at $e\in\g$ \cite{pre,fu-jut-lev-som,kro}, as is consistent with Proposition \ref{Proposition: Slodowy Poisson transversal}. It is natural to ask if $\S_{\mathfrak{T}}$ is a \textit{slice} at $e\in\g$ in any reasonable complex-geometric analogue of Definition \ref{Definition: Slice compact}. The answer is decidedly negative; $S_{\mathfrak{T}}$ need not be preserved by the adjoint action of $G_e\subset G$, implying that the natural complex-geometric analogues of Parts (i) and (iii) in Definition \ref{Definition: Slice compact} do not apply.
	\end{remark}
	
	\section{Decomposition classes in Poisson geometry}\label{Section: Decomposition classes in Poisson geometry}
	In this section, $G$ is a connected complex reductive algebraic group with Lie algebra $\g$. We recall the definition of decomposition classes in $\g$, and develop useful properties thereof in the context of Hamiltonian group actions. Main Theorem \ref{thm:simplestpcs} is proved by combining Corollary \ref{Corollary: Poisson slice} and Propositions \ref{Proposition: Symplectic subvariety}, \ref{Proposition: Characterizations}, \ref{Proposition: Holomorphic characterizations}, and \ref{Proposition: Trivial action}. Most of our results apply in both the holomorphic and complex-algebraic categories.  
	
	\subsection{Decomposition classes}\label{Subsection: Decomposition classes}
	Retain the notation and conventions of Section \ref{Subsection: Complex reductive algebraic groups}. 
	Consider the equivalence relation on $\g$ defined by $x\sim y$ if there exists $g\in G$ satisfying $y_{\nilp}=\mathrm{Ad}_g(x_{\nilp})$ and $\g_{y_{\semi}}=\mathrm{Ad}_g(\g_{x_{\semi}})$. Introduced by Borho--Kraft \cite{bor-kra:79}, \textit{decomposition classes} in $\g$ are defined to be equivalence classes of $\sim$. It turns out that $x,y\in\g$ belong to the same decomposition class if and only if $G_x$ and $G_y$ are $G$-conjugate \cite[Theorem 3.7.1]{Broer}. Decomposition classes are thereby the orbit-type strata for the adjoint action of $G$. As such, one might expect there to be a \textit{principal decomposition classes} in complex-algebraic and holomorphic analogues of Theorem \ref{Theorem: New slice theorem}; see Section \ref{Subsection: Principal decomposition classes} for more details.  
	
	\begin{remark}
		While it is common to consider decomposition classes in reductive Lie algebras, some authors restrict to those in semisimple Lie algebras. This restriction is made harmless by Proposition \ref{Proposition: Comparative Jordan decomposition}: there is a bijective correspondence between decomposition classes in $\g$ and those in the derived subalgebra $[\g,\g]\subset\g$, where one associates to a $[G,G]$-decomposition class $\D\subset[\g,\g]$ 
		the $G$-decomposition class $\z(\g)+\D\subset\g$.
		In this way, many results about decomposition classes in semisimple Lie algebras easily generalize to the reductive case. Our manuscript is written accordingly; we sometimes invoke facts about the reductive case when the accompanying references only address the semisimple case and the generalization is clear.
	\end{remark}
	
	Each decomposition class $\D\subset\g$ is a $G$-invariant, smooth \cite[Corollary 3.8.1]{Broer}, locally closed \cite[Corollary 39.1.7]{Tauvel-Yu}, irreducible \cite[Section 3.3]{Broer} subvariety of $\g$. It also turns out that there exist only finitely many decomposition classes \cite[Theorem 3.5.2]{Broer}. To index them, consider the following set of pairs: $$\{(\l,\O):\l\subset\g\text{ is a Levi subalgebra and }\O\subset\l\text{ is a nilpotent orbit}\}.$$ Note that $G$ acts on this set by the formula $g(\l,\O)\coloneqq (\Ad_g(\l),\Ad_g(\O))$. We write $[(\l,\O)]$ for the $G$-orbit of $(\l,\O)$. At the same time, we consider the \textit{generic locus}
	$$\z(\l)_{\mathsf{gen}}\coloneqq\{x\in\z(\l):\dim\g_x\leq\dim \g_y\text{ for all }y\in\z(\l)\}
	$$ 
	for a Levi subalgebra $\l\subset\g$. The $G$-saturation $$\D_{\l,\O}\coloneqq G(\O+\z(\l)_{\mathsf{gen}})\subset\g$$ is a decomposition class, and is independent of the chosen pair $(\l,\O)\in [(\l,\O)]$.
	Furthermore, the association $[(\l,\O)]\mapsto\D_{\l,\O}$ is a bijection from the $G$-orbits of pairs $(\l,\O)$ to the decomposition classes in $\g$ \cite{bor-kra:79}.
	
	Another interesting fact is that the closure of a decomposition class is a union of decomposition classes \cite[Theorem 3.5.2]{Broer}. This fact allows us to consider the \textit{closure order} on the set of decomposition classes: $\mathcal{D}\leq\mathcal{E}$ if $\mathcal{D}\subset\overline{\mathcal{E}}$.
	
	\begin{example}[Regular semisimple elements]\label{Example: Regss}
		Let $\t\subset\g$ be a Cartan subalgebra. Note that $\{0\}\subset\t$ is the unique nilpotent orbit in $\t$. One finds that $\D_{(\t,\{0\})}=\greg\cap\g_{\mathsf{semi}}$. It follows that $\greg\cap\g_{\mathsf{semi}}$ is the unique maximal decomposition class in the above-defined partial order.
	\end{example}
	
	\begin{example}[Nilpotent orbits]\label{Example: Nilpotent orbits}
		If $\O\subset\g$ is a nilpotent orbit, then $\D_{(\g,\O)}=\z(\g)+\O$. If $\g$ is semisimple, this implies that nilpotent $G$-orbits are decomposition classes.	
	\end{example}

	\subsection{Decomposition classes as moment map images}\label{Subsection: Moment map images}
	We now associate to each decomposition class $\D\subset\g$ a holomorphic Hamiltonian $G$-space with $\D$ as its moment map image. The following lemma is a first step in this direction.
	
	\begin{lemma}\label{Lemma: Dilation}
		If $\D\subset\g$ is a decomposition class, then $\D$ is invariant under the dilation action of $\mathbb{C}^{\times}$ on $\g$.
	\end{lemma}
	
	\begin{proof}
		Suppose that $x\in\D$. Since $e\coloneqq x_{\nilp}\in[\g_{x_{\semi}},\g_{x_{\semi}}]$, the Jacobson--Morozov theorem implies the existence of $h,f\in[\g_{x_{\semi}},\g_{x_{\semi}}]$ for which $(e,h,f)$ is an $\mathfrak{sl}_2$-triple. Note that $$\mathrm{Ad}_{\exp(t h)}(x_{\semi})=x_{\semi}\quad\text{and}\quad \mathrm{Ad}_{\exp(t h)}(x_{\nilp})=e^{2t}x_{\nilp}$$ for all $t\in\mathbb{C}$. We conclude that $g=\exp(t h)$ satisfies $y_{\nilp}=\mathrm{Ad}_g(x_{\nilp})$ and $\g_{y_{\semi}}=\mathrm{Ad}_g(\g_{x_{\semi}})$ for $t\in\mathbb{C}$ and $y=e^{2t}x$. This implies that $x\sim e^{2t}x$ for all $\lambda\in\mathbb{C}$, where $\sim$ is the equivalence relation used to define decomposition classes.
	\end{proof}
	
	\begin{proposition}\label{Proposition: Canonical}
		Suppose that $\D\subset\g$ is a decomposition class. There exists a canonical holomorphic Hamiltonian $G$-space $M$ with $\mu(M)=\D$, where $\mu:M\longrightarrow\g$ is the moment map. 
	\end{proposition}
	
	\begin{proof}
		Consider the set $$M\coloneqq\bigsqcup_{x\in\D}G/N(x)^{\circ},$$ 
		where $N(x)^\circ$ is the identity component of $$N(x)\coloneqq [G_{x_{\semi}},G_{x_{\semi}}]_{x_{\nilp}}.$$ Note that $G$ acts on $M$ by left multiplication on each factor. As explained in \cite[Section 4.9]{cro-may:22}, the $G$-set $M$ is canonically a connected holomorphic Hamiltonian $G$-space with moment map
		$$\mu:M\longrightarrow\g,\quad ([g],x)\mapsto-\mathrm{Ad}_g(x),\quad x\in\D,\text{ }[g]\in G/N(x)^{\circ}.$$ Lemma \ref{Lemma: Dilation} implies that $\mu(M)=-\D=\D$.
	\end{proof}
	
	\begin{remark}
		While $G$ is taken to be semisimple in \cite[Section 4.9]{cro-may:22}, the results contained therein apply for a connected complex reductive algebraic group $G$. This justifies our use of \cite[Section 4.9]{cro-may:22} in the previous proof.
	\end{remark}
	
	\subsection{Poisson slices in decomposition classes}\label{Subsection: Poisson transversals in decomposition classes}
	Suppose that $\D=\D_{\l,\O}$ for a Levi subalgebra $\mathfrak{l}\subset\g$ and nilpotent orbit $\O\subset\mathfrak{l}$. Fixing $e\in\O$, one has $\mathcal{D}=G(e+\z(\l)_{\mathsf{gen}})$ \cite[Corollary 3.8.1(ii)]{Broer}. It is therefore natural to describe $(e+\z(\l)_{\mathsf{gen}})\cap Gx$ for each $x\in e+\z(\l)_{\mathsf{gen}}$. To this end, let $N_{(\mathfrak{l},\O)}\subset G$ denote the $G$-normalizer of $\mathfrak{z}(\mathfrak{l})+\O$. One has $N_{(\mathfrak{l},\O)}\subset N_G(L)$ \cite[Corollary 3.8.1]{Broer}, where $L\subset G$ is the Levi subgroup integrating $\l$. We conclude that $$\Gamma_{(\mathfrak{l},\O)}\coloneqq N_{(\mathfrak{l},\O)}/L$$ is a subgroup of the finite group $N_G(L)/L$. At the same time, the adjoint action of $N_G(L)\subset G$ preserves $\z(\l)_{\mathsf{gen}}$. We also note that $L\subset N_G(L)$ acts trivially on $\mathfrak{z}(\mathfrak{l})$. In this way, $N_G(L)/L$ and its subgroup $\Gamma_{(\mathfrak{l},\O)}$ act on $\z(\l)_{\mathsf{gen}}$. The following is an immediate consequence of \cite[Corollary 3.8.1(iii)]{Broer}.
	
	\begin{proposition}\label{Proposition: Variation}
		Suppose that $x,y\in\z(\l)_{\mathsf{gen}}$, and set $\Gamma\coloneqq\Gamma_{(\mathfrak{l},\O)}$. We have $G(e+x)=G(e+y)$ if and only if $\Gamma x=\Gamma y$.
	\end{proposition}
	
	\begin{example}
		Let $T\subset G$ be a maximal torus with Lie algebra $\mathfrak{t}\subset\g$ and associated Weyl group $W\coloneqq N_G(T)/T$. By Example \ref{Example: Regss}, $\D_{[(\t,0)]}=\greg\cap\g_{\mathsf{semi}}$. Set $\D\coloneqq\greg\cap\g_{\mathsf{semi}}$ and $\t_{\reg}\coloneqq\t\cap\greg$. In this case, Proposition \ref{Proposition: Variation} is the well-known result that two regular semisimple adjoint orbits coincide if and only if they intersect $\mathfrak{t}_{\text{reg}}$ in the same $W$-orbit.
	\end{example}
	
	Proposition \ref{Proposition: Variation} implies that each adjoint orbit in $\mathcal{D}$ has a finite intersection with the subvariety $e+\z(\l)_{\mathsf{gen}}\subset\mathcal{D}$. On the other hand, $\mathcal{D}$ is a Poisson variety with symplectic leaves given by the adjoint orbits of $\g$ contained in $\mathcal{D}$. One may therefore study Poisson transversals in $\mathcal{D}$, as per Section \ref{Subsection: Poisson transversals}. The following result arises in this context.
	
	\begin{proposition}\label{Proposition: Transversal}
		The subvariety $e+\z(\l)_{\mathsf{gen}}$ is a Poisson transversal in $\mathcal{D}$.
	\end{proposition}
	
	\begin{proof}
		Suppose that $x\in e+\z(\l)_{\mathsf{gen}}$. Since $\mathcal{D}=G(e+\z(\l)_{\mathsf{gen}})$, we have $T_x\mathcal{D}=[\g,x]+\mathfrak{z}(\mathfrak{l})$. We also have $\dim(T_x\mathcal{D})=\dim([\g,x])+\dim(\mathfrak{z}(\mathfrak{l}))$ \cite[Corollary 3.8.1(i)]{Broer}. It follows that $T_x\mathcal{D}=[\g,x]\oplus\mathfrak{z}(\mathfrak{l})$. This completes the proof.
	\end{proof}
	
	\begin{corollary}\label{Corollary: Poisson slice}
		The subvariety $e+\z(\l)_{\mathsf{gen}}$ is a Poisson slice in $\mathcal{D}$.
	\end{corollary}
	
	\begin{proof}
		This follows immediately from Proposition \ref{Proposition: Transversal} and the fact that the $G$-orbits in $\mathcal{D}$ are the symplectic leaves of $\mathcal{D}$.
	\end{proof}
	
	\begin{proposition}\label{Proposition: Symplectic subvariety}
		Let $M$ be a Hamiltonian $G$-variety (resp. holomorphic Hamiltonian $G$-space) with moment map $\mu:M\longrightarrow\g$. Suppose that $\mu(M)\subset\mathcal{D}$ for the decomposition class associated to $(\mathfrak{l},\O)$. If $e\in\O$, then $\mu^{-1}(e+\mathfrak{z}(\mathfrak{l})_{\mathsf{gen}})$ is a symplectic slice in $M$.
	\end{proposition}
	
	\begin{proof}
		By hypothesis, $\mu$ defines a $G$-equivariant Poisson morphism $M\longrightarrow\mathcal{D}$.  We may apply Proposition \ref{Proposition: Pullback} and Corollary \ref{Corollary: Poisson slice} to conclude that $\mu^{-1}(e+\mathfrak{z}(\mathfrak{l})_{\mathsf{gen}})$ is a symplectic slice in $M$.
	\end{proof}

	\subsection{Principal decomposition classes}\label{Subsection: Principal decomposition classes}
	We now introduce principal decomposition classes for Hamiltonian $G$-varieties and holomorphic Hamiltonian $G$-spaces.

	\begin{proposition}\label{Proposition: Characterizations}
		Let $M$ be a non-empty, irreducible Hamiltonian $G$-variety with moment map $\mu:M\longrightarrow\g$. There exists a unique decomposition class $\mathcal{D}_M\subset\g$ with the following properties:
		\begin{itemize}
			\item[\textup{(i)}]  $\mu^{-1}(\mathcal{D}_M)$ is open and dense in $M$;
			\item[\textup{(ii)}] $\D_M$ is maximal among the decomposition classes intersecting $\mu(M)$;
			\item[\textup{(iii)}] $\overline{\mathcal{D}_M\cap\mu(M)}=\overline{\mu(M)}$.
		\end{itemize}
		Furthermore, $\D_M$ is the unique decomposition class satisfying any one of these properties.
	\end{proposition}
	
	\begin{proof}
		Let $\D_1,...,\D_k\subset\g$ be the decomposition classes satisfying intersecting $\mu(M)$. Note that
		\[ M=\bigcup_{j=1}^k \mu^{-1}(\ol{\D_j}),\]
		and that $\mu^{-1}(\D_j)\ne \emptyset$ for all $j\in\{1,\ldots,k\}$. The irreducibility of $M$ implies the existence of $j_0\in\{1,\ldots,k\}$ such that $M=\mu^{-1}(\ol{\D})$ and $\D\coloneqq\D_{j_0}$ is maximal with respect to the closure order among $\D_1,...,\D_k$. If $j_0'\in\{1,\ldots,k\}$ and $\D'\coloneqq\D_{j_0'}$ is not comparable to $\D$ in the closure order, then $\D'\cap \ol{\D}=\emptyset$, $\mu^{-1}(\D')\ne \emptyset$, and $M=\mu^{-1}(\ol{\D})$; this is a contradiction. We conclude that $\D$ is the unique decomposition class satisfying (ii). On the other hand, $\D$ is open in $\ol{\D}$. This implies that $\mu^{-1}(\D)\subset M$ is open in $M$. Since $M$ is irreducible, $\mu^{-1}(\D)$ is also dense. It follows that $\D$ is the unique decomposition class satisfying (i).
		
		It remains to establish the existence of a unique decomposition class satisfying (iii), and that this class coincides with that in (i). We first observe that
		$$\mu(M)=\bigcup_{j=1}^k(\mathcal{D}_j\cap\mu(M)).$$ This yields
		$$\overline{\mu(M)}=\bigcup_{j=1}^k\overline{(\mathcal{D}_j\cap\mu(M))}.$$ On the other hand, the irreducibility of $M$ implies that of $\overline{\mu(M)}$. We conclude that $\overline{\mu(M)}=\overline{\mathcal{D}_{j_0}\cap\mu(M)}$ for some $j_0\in\{1,\ldots,k\}$. To prove that $\D_{j_0}$ is the unique decomposition class satisfying (iii), let $\mathcal{D},\mathcal{D}'\subset\g$ be decomposition classes satisfying $\overline{\mathcal{D}\cap\mu(M)}=\overline{\mu(M)}=\overline{\mathcal{D}'\cap\mu(M)}$. Note that $\mathcal{D}\cap\mu(M)$ and $\mathcal{D}'\cap\mu(M)$ are constructible subsets of $\g$. This implies the existence of open dense subsets $U,V\subset\overline{\mu(M)}$ satisfying $U\subset \mathcal{D}\cap\mu(M)$ and $V\subset \mathcal{D}'\cap\mu(M)$. As $\overline{\mu(M)}$ is irreducible and non-empty, $U\cap V\neq\emptyset$. We conclude that $\mathcal{D}\cap\mathcal{D}'\neq\emptyset$, i.e. $\D=\D'$.
		
		Our final task is to prove that the unique decomposition class in (iii) is that in (i). Let $\mathcal{D}\subset\g$ and $\mathcal{D}'\subset\g$ be the unique decomposition classes satisfying (i) and (iii), respectively. We have $\mu(M)\subset\overline{\mathcal{D}'}$. Since $\mathcal{D}'$ is open in $\overline{\mathcal{D}'}$, $\mu^{-1}(\mathcal{D}')$ is open in $M$. We also know that the constructible subset $\mathcal{D}'\cap\mu(M)$ contains an open dense subset of its closure $\overline{\mu(M)}$. This implies that $\mu^{-1}(\mathcal{D}')\neq\emptyset$. It follows that $\mu^{-1}(\mathcal{D})\cap\mu^{-1}(\mathcal{D}')\neq\emptyset$. As distinct decomposition classes are disjoint, we must have $\mathcal{D}=\mathcal{D}'$.
	\end{proof}
	
	\begin{proposition}\label{Proposition: Holomorphic characterizations}
		Let $M$ be a non-empty, connected holomorphic Hamiltonian $G$-space with moment map $\mu:M\longrightarrow\g$. There exists a unique decomposition class $\D_M\subset\g$ that satisfies the following properties:
		\begin{itemize}		
			\item[\textup{(i)}] $\mu^{-1}(\D_M)$ is open, dense, and path-connected in $M$; 
			\item[\textup{(ii)}]$\D_M$ is maximal among the decomposition classes intersecting $\mu(M)$; 
			\item[\textup{(iii)}] $\overline{\D_M\cap\mu(M)}=\overline{\mu(M)}$.
		\end{itemize}
		Moreover, $\D_M$ is the unique decomposition class satisfying any one of these properties.
	\end{proposition}
	\begin{proof}
		Write $\D_1,...,\D_k\subset\g$ for the decomposition classes intersecting $\mu(M)$. As in the proof of Proposition \ref{Proposition: Characterizations},
		\[ M=\bigcup_{j=1}^k \mu^{-1}(\ol{\D_j}),\] and $\mu^{-1}(\D_j)\ne \emptyset$ for $j\in\{1,\ldots,k\}$. Each closure $\ol{\D_j}$ is a subvariety of $\g$, and hence a complex analytic subset. It follows that $\mu^{-1}(\ol{\D_j})$ is a complex analytic subset of $M$ for all $j\in\{1,\ldots,k\}$. On the other hand, the complement of a proper complex analytic subset of $M$ is open, dense, and path-connected \cite[Lemma 21]{whitney:72}. If $\mu^{-1}(\ol{\D_j})\subsetneqq M$ for all $j\in\{1,\ldots,k\}$, the previous sentence implies that $\cap_j (M\backslash\mu^{-1}(\ol{\D_j}))=M\backslash (\cup_j \mu^{-1}(\ol{\D_j}))=\emptyset$ is open and dense in $M$. We deduce the existence of $\D \in \{\D_1,...,\D_k\}$, maximal with respect to the closure order, such that $M=\mu^{-1}(\ol{\D})$. Without loss of generality, we may suppose that $\D=\D_1$. If $\D'=\D_j$ is not comparable to $\D$ in the closure order, then $\D'\cap \ol{\D}=\emptyset$, $\mu^{-1}(\D')\ne \emptyset$, and $M=\mu^{-1}(\ol{\D})$; this is a contradiction. It follows that $\D$ is the unique decomposition class satisfying (ii).
		
		We now prove that $\D$ is the unique decomposition class satisfying (i). We first note that for $j>1$, $\D$ is not below $\D_j$ in the closure order. It follows that $\ol{\D_j}$ does not contain $\D$ for any $j>1$. Since $\ol{\D_j}$ is a union of decomposition classes, $\ol{\D_j}\cap \D=\emptyset$ for all $j>1$. We deduce that $\mu^{-1}(\ol{\D_j})$ is a proper analytic subset of $M$ for all $j>1$. This implies that
		\[ X\coloneqq\bigcup_{j=2}^k \mu^{-1}(\ol{\D_j}) \]
		is a proper analytic subset of $M$. The complement 
		\[ M\backslash X=\mu^{-1}(\D) \] 
		is therefore open, dense, and path connected in $M$. This combines with an argument analogous to one given in the proof of Proposition \ref{Proposition: Characterizations} to imply that $\D$ is the unique decomposition class satisfying (i). 
		
		Our final task is to prove that $\D$ is the unique decomposition class satisfying (iii). A first step is to show that $\mu(M)\cap \D$ is dense in $\ol{\mu(M)}$. Let $U$ be an open neighborhood of a point $p \in \ol{\mu(M)}$. Since $U$ is open, $U\cap \mu(M)\ne \emptyset$. It follows that $\mu^{-1}(U)$ is open and non-empty. Note that $\mu^{-1}(U)\cap \mu^{-1}(\D)\neq\emptyset$, as $\mu^{-1}(\D)$ is dense in $M$. Hence $\mu(\mu^{-1}(U)\cap \mu^{-1}(\D))=\mu(M)\cap U\cap \D$ is non-empty, as required. If $\D'$ is another decomposition class such that $\mu(M)\cap \D'$ is dense in $\ol{\mu(M)}$, then $\D\cap \ol{\D'}\supset \mu(M)\cap \D \ne \emptyset$. This implies that $\D\subset \ol{\D'}$, and likewise $\ol{\D}\cap \D'\supset\mu(M)\ne\emptyset$. We conclude that $\ol{\D}\supset \D'$, and so $\D=\D'$.
	\end{proof}
	
	\begin{definition}\label{Definition: Principal decomposition class}
		Let $M$ be a non-empty, irreducible Hamiltonian $G$-variety (resp. non-empty, connected holomorphic Hamiltonian $G$-space). We define the \textit{principal decomposition class} of $M$ to be the decomposition class $\mathcal{D}_M\subset\g$ referenced in Proposition \ref{Proposition: Characterizations} (resp. Proposition \ref{Proposition: Holomorphic characterizations}).
	\end{definition}
	
	\begin{remark}[Consistency between the algebraic and holomorphic definitions]
		Suppose that $M$ is a non-empty, irreducible Hamiltonian $G$-variety. Definition \ref{Definition: Principal decomposition class} yields a principal decomposition class $\D_M\subset\g$. On the other hand, we may regard $M$ as a connected holomorphic Hamiltonian $G$-space. As such, $M$ has a principal decomposition class. Proposition \ref{Proposition: Characterizations}(ii) and Proposition \ref{Proposition: Holomorphic characterizations}(ii) tell us that this principal decomposition class is $\D_M$.
	\end{remark}
	
	\begin{example}\label{Example: Weyl group}
		Choose a maximal torus $T\subset G$, and let $W\coloneqq N_G(T)/T$ denote the associated Weyl group. Suppose that $\O\subset\g^*$ is the coadjoint orbit of a regular semisimple element of $\g=\g^*$. The inclusion $\O\longrightarrow \g^*=\g$ is a moment map for the Hamiltonian $G$-action on $\O$. The only decomposition class $\D$ intersecting $\mu(\O)$ is $\g_{\reg}\cap\g_{\mathsf{semi}}$. In other words, $\D_{\O}=\g_{\reg}\cap\g_{\mathsf{semi}}$. Note that $\g_{\reg}\cap\g_{\mathsf{semi}}=G({\t\cap\g_{\reg}})$, where $\t\subset\g$ is the Lie algebra of $T$. We also note that $\mu^{-1}(\t\cap\g_{\reg})=\mathcal{O}\cap(\t\cap\g_{\reg})$ consists of $\vert W\vert$-many points.  
	\end{example}
	
	\begin{example}
		Recall from Example \ref{Example: Regss} that $\greg\cap\g_{\mathsf{semi}}$ is the unique maximal decomposition class in $\g$. It turns out to arise as the principal decomposition class in many examples. To this end, let $M$ be a non-empty, connected holomorphic Hamiltonian $G$-space with moment map $\mu:M\longrightarrow\g$. Suppose that $\mathrm{d}\mu_p:T_pM\longrightarrow\g$ is surjective for some $p\in M$. This implies that $\mu(M)$ contains a non-empty open subset of $\g$. Since $\greg\cap\g_{\mathsf{semi}}$ is dense for this topology, we must have $\mu(M)\cap(\greg\cap\g_{\mathsf{semi}})\neq\emptyset$. It follows that $\D_M=\greg\cap\g_{\mathsf{semi}}$.
	\end{example}
	
	\begin{example}[Arbitrary decomposition classes]
		Let $\D\subset\g$ be a decomposition class. Recall the holomorphic Hamiltonian $G$-space $M$ defined in the proof of Proposition \ref{Proposition: Canonical}. This proposition implies that $\D_{M}=\D$. In particular, every decomposition class is canonically realizable as a principal decomposition class.
	\end{example}
	
	\begin{example}[Springer resolutions and Richardson orbits]
		Let $P\subset G$ be a parabolic subgroup with Lie algebra $\p$. Consider the nilpotent radical $\mathfrak{u}(\p)\subset\p$. There exists a unique nilpotent orbit $\O_{P}\subset\g$ satisfying $G\mathfrak{u}(\p)=\overline{\O_{P}}$, called the \textit{Richardson orbit} associated to $P$. On the other hand, the action of $G$ on $G/P$ lifts to a Hamiltonian $G$-action on $T^*(G/P)$. The moment map $T^*(G/P)\longrightarrow\g$ is called the \textit{Grothendieck--Springer resolution} associated to $P$; its image turns out to be $\overline{\O_{P}}$. In light of Example \ref{Example: Nilpotent orbits}, we conclude that $\D_{T^*(G/P)}=\z(\g)+\O_P$. If $G$ is semisimple, then $\D_{T^*(G/P)}=\O_P$.  
	\end{example}

	\def\N{\ensuremath{\mathcal{N}}}
	\def\R{\ensuremath{\mathcal{R}}}
	\def\E{\ensuremath{\mathcal{E}}}
	\newcommand{\pair}[2]{\langle #1, #2 \rangle}
	\newcommand{\tn}[1]{\textnormal{#1}}
	\def\bR{\ensuremath{\mathbb{R}}}
	\def\bC{\ensuremath{\mathbb{C}}}
	\def\n{\ensuremath{\mathfrak{n}}}
	
	\section{The residual action of an almost-abelian group $A(x)$}\label{Section: Residual}
	We continue to assume that $G$ is a connected complex reductive algebraic group with Lie algebra $\g$. Given $x\in \g$, we define a finite extension $A(x)$ of a torus. This extension is shown to be the smallest quotient of $G_x$ with the following property: if $M$ is any holomorphic Hamiltonian $G$-space with moment map $\mu:M\longrightarrow\g$, then the $G_x$-action on $\mu^{-1}(x)$ descends to an action of $A(x)$.
	
	\subsection{Subquotients of $G$ associated to elements of $\g$}\label{Subsection: Subquotients}
	Fix $x\in\g$. The $G$-stabilizer of $x_{\semi}\in\g$ is the reductive, Levi subgroup $$L(x)\coloneqq G_{x_{\semi}}.$$ It follows that the derived subgroup $$L(x)'\coloneqq[L(x),L(x)]=[G_{x_{\semi}},G_{x_{\semi}}]$$ is semisimple. We also note that $$T(x)\coloneqq L(x)/L(x)'=G_{x_{\semi}}/[G_{x_{\semi}},G_{x_{\semi}}]$$ is a torus of rank equal to $\rank(G)-\rank(L(x)')$. On the other hand, the Lie algebra $[\g_{x_{\semi}},\g_{x_{\semi}}]$ of $L(x)'$ contains $x_{\nilp}$. This fact allows one to consider the $L(x)'$-stabilizer of $x_{\nilp}$, i.e. the group $$N(x)\coloneqq [G_{x_{\semi}},G_{x_{\semi}}]_{x_{\nilp}}.$$ Let us write $\mathfrak{l}(x)$, $\mathfrak{l}(x)'$, $\mathfrak{t}(x)$, and $\mathfrak{n}(x)$ for the Lie algebras of $L(x)$, $L(x)'$, $T(x)$, and $N(x)$, respectively. It follows that
	$\mathfrak{l}(x)=\g_{x_{\semi}}$, $\mathfrak{l}(x)'=[\g_{x_{\semi}},\g_{x_{\semi}}]$, $\mathfrak{t}(x)=\g_{x_{\semi}}/[\g_{x_{\semi}},\g_{x_{\semi}}]$, and $\mathfrak{n}(x)=[\g_{x_{\semi}},\g_{x_{\semi}}]_{x_{\nilp}}$.
	
	\begin{remark}\label{Remark: Specializations}
		Some of the above-defined groups admit simpler descriptions in certain cases. We have $L(x)=G_x$ if $x\in\g$ is semisimple. If $G$ is semisimple and $x\in\g$ is nilpotent, then $L(x)=G=L(x)'$, $T(x)=\{e\}$, and $N(x)=G_x$.
	\end{remark}
	
	\begin{example} It is possible for $G_x$ to be connected and $N(x)$ to be disconnected. To this end, suppose that $G=\operatorname{PSL}_3$. Let $$s=\begin{bmatrix} 1 & 0 & 0\\ 0 & 1 & 0 \\ 0 & 0 & -2
		\end{bmatrix}\in\mathfrak{sl}_3=\g.$$ Observe that $[G_s,G_s]$ is the image of the algebraic group embedding
		$$\operatorname{SL_2}\longrightarrow\mathrm{PSL}_3,\quad A\mapsto \left[\begin{bmatrix} A & 0 \\ 0 & 1\end{bmatrix}\right].$$ At the same time, choose a nilpotent element $n\in\mathfrak{sl}_3$ satisfying $[s,n]=0$ and $x\coloneqq s+n\in(\mathfrak{sl}_3)_{\text{reg}}$; an example is $n=E_{12}$. We have $x_{\semi}=s$ and $x_{\nilp}=n$. Note that $n$ is regular in $[\g_s,\g_s]\cong\mathfrak{sl}_2$. We also note that the $\operatorname{SL}_2$-stabilizer of a regular nilpotent element in $\mathfrak{sl}_2$ is disconnected. These considerations imply that $N(x)$ is disconnected. On the other hand, observe that $\operatorname{PSL}_3$ is the adjoint group of $\mathfrak{sl}_3$. It follows that the $\operatorname{PSL}_3$-stabilizer of a regular element in $\mathfrak{sl}_3$ is connected \cite[Proposition 14]{kos:63}. In particular, $G_x$ is connected.
	\end{example}
	
	\subsection{The almost-abelian group $A(x)$}\label{Subsection: Almost abelian} 
	As in Section \ref{Subsection: Subquotients}, we fix $x\in\g$. Let us also consider the identity component $N(x)^{\circ}$ of $N(x)$. It is clear that we have inclusions $N(x)^{\circ}\subset N(x)\subset G_x\subset L(x)$. This gives context for the following result.
	
	\begin{lemma}
		If $x\in\g$, then the subgroup $N(x)^{\circ}\subset G_x$ is normal.
	\end{lemma}
	
	\begin{proof}
		This follows immediately from the fact that $G_x=(G_{x_{\semi}})_{x_{\nilp}}=Z(G_{x_{\semi}})N(x)$.
	\end{proof}
	
	The quotient group
	$$A(x)\coloneqq G_x/N(x)^{\circ}=G_x/[G_{x_{\semi}},G_{x_{\semi}}]_{x_{\nilp}}^{\circ}$$ and component group 
	$$C(x)\coloneqq\pi_0(N(x))=N(x)/N(x)^{\circ}=[G_{x_{\semi}},G_{x_{\semi}}]_{x_{\nilp}}/[G_{x_{\semi}},G_{x_{\semi}}]_{x_{\nilp}}^{\circ}$$ feature prominently in what follows. As one might expect, we write $\a(x)$ for the Lie algebra of $A(x)$. Note that $\a(x)=\g_x/[\g_{x_{\semi}},\g_{x_{\semi}}]_{x_{\nilp}}$.
	
	\begin{remark}
		If $x\in\g$ is semisimple, then $A(x)=T(x)$ and $C(x)=\{e\}$. If $G$ is semisimple and $x\in\g$ is nilpotent, then $A(x)=C(x)$ and $T(x)=\{e\}$. The groups $A(x)$ and $C(x)$ are more subtle when one no longer requires $x$ to be semisimple or nilpotent; in Proposition \ref{Proposition: Sequence}, we show $A(x)$ to be a finite extension of $T(x)$ by $C(x)$. 
	\end{remark}
	
	We now describe the group $A(x)$ in two ways. For the first, let $Z(L(x))$ and $Z(L(x)')$ denote the centers of $L(x)$ and $L(x)'$, respectively. Note that $Z(L(x)')$ acts on $Z(L(x))\times C(x)$ by
	$$g\cdot(h,[k])\coloneqq (hg^{-1},[gk]),\quad g\in Z(L(x)'),\text{ }(h,[k])\in Z(L(x))\times C(x).$$ Write $Z(L(x))\times_{Z(L(x)')}C(x)$ for the quotient by this $Z(L(x)')$-action. Observe that $Z(L(x))\times_{Z(L(x)')}C(x)$ is the quotient of $Z(L(x))\times C(x)$ by the normal subgroup
	\begin{equation}\label{Equation: Kernel}\{(g,[g^{-1}]):g\in Z(L(x)')\}\cong Z(L(x)').\end{equation} In particular, $Z(L(x))\times_{Z(L(x)')}C(x)$ is an algebraic group. Our first description of $A(x)$ is the following.
	
	\begin{proposition}\label{Proposition: Strange morphism}
		Suppose that $x\in\g$. The map \begin{equation}\label{Equation: Map}Z(L(x))\times C(x)\longrightarrow A(x),\quad (g,[h])\mapsto [gh]\end{equation} descends to an algebraic group isomorphism $$Z(L(x))\times_{Z(L(x)')}C(x)\overset{\cong}\longrightarrow A(x).$$
	\end{proposition}
	
	\begin{proof}
		Since elements of $Z(L(x))$ and $N(x)$ commute with one another, it is clear that \eqref{Equation: Map} is an algebraic group morphism. This morphism is surjective, as $G_x=Z(L(x))N(x)$. It therefore suffices to establish that \eqref{Equation: Kernel} is the kernel of \eqref{Equation: Map}. To this end, suppose that $g\in Z(L(x))$ and $h\in N(x)$. We have $gh\in N(x)^{\circ}$ only if $g\in Z(L(x))\cap N(x)=Z(L(x)')$. It follows that $gh\in N(x)^{\circ}$ if and only if $g\in Z(L(x)')$ and $[h]=[g^{-1}]\in C(x)$. This establishes that \eqref{Equation: Kernel} is the kernel of \eqref{Equation: Map}.
	\end{proof}
	
	Our second description of $A(x)$ is as the following finite extension of $T(x)$.
	
	\begin{proposition}\label{Proposition: Sequence}
		If $x\in\g$, then there is a canonical short exact sequence $$\{e\}\longrightarrow C(x)\longrightarrow A(x)\longrightarrow T(x)\longrightarrow\{e\}$$ of algebraic groups.
	\end{proposition}
	
	\begin{proof}
		The surjective quotient morphism
		$$A(x)=G_x/N(x)^{\circ}\longrightarrow G_x/N(x)$$ has $N(x)/N(x)^{\circ}=C(x)$ as its kernel. It therefore remains only to construct a canonical algebraic group isomorphism $$G_x/N(x)\overset{\cong}\longrightarrow T(x).$$
		
		Observe that $G_x=Z(L(x))N(x)$ and $Z(L(x))\cap N(x)=Z(L(x)')$. It follows that the inclusion $Z(L(x))\subset G_x$ descends to an algebraic group isomorphism $$Z(L(x))/Z(L(x)')\overset{\cong}\longrightarrow G_x/N(x).$$ At the same time, the inclusion $Z(L(x))\subset L(x)$ descends to an algebraic group isomorphism $$Z(L(x))/Z(L(x)')\overset{\cong}\longrightarrow T(x).$$ These last two sentences complete the proof.
	\end{proof}
	
	\begin{corollary}
		If $x\in\g$, then the inclusion $\g_x\subset\g_{x_{\semi}}$ descends to a Lie algebra isomorphism $$\a(x)=\g_x/[\g_{x_{\semi}},\g_{x_{\semi}}]_{x_{\nilp}}\overset{\cong}\longrightarrow\g_{x_{\semi}}/[\g_{x_{\semi}},\g_{x_{\semi}}]=\mathfrak{t}(x).$$
	\end{corollary}
	
	\begin{proof}
		This follows from Proposition \ref{Proposition: Sequence} and the fact that $C(x)$ is finite.
	\end{proof}
	
	\begin{corollary}\label{Corollary: Simple}
		If $x\in\g$, then $A(x)$ is abelian if and only if $C(x)$ is abelian.
	\end{corollary}
	
	\begin{proof}
		The forward implication follows from Proposition \ref{Proposition: Sequence}, while the backward implication is a consequence of Proposition \ref{Proposition: Strange morphism}.  
	\end{proof}
	
	\begin{proposition}\label{Proposition: Lie types}
		Suppose that $x\in\g$. If every simple ideal of $[\g_{x_{\semi}},\g_{x_{\semi}}]$ is of Lie type $A$ or $C$, then $A(x)$ is abelian.
	\end{proposition}
	
	\begin{proof}
		In light of Corollary \ref{Corollary: Simple}, it suffices to prove that $C(x)$ is abelian. Let $H$ be the simply-connected semisimple affine algebraic group with Lie algebra $[\g_{x_{\semi}},\g_{x_{\semi}}]$. It follows that $L(x)'=H/S$ for some subgroup $S$ of the finite group $Z(H)$. The quotient map $H\longrightarrow L(x)'$ restricts and descends to a surjective group morphism $$\pi_0(H_{x_{\nilp}})=H_{x_{\nilp}}/H_{x_{\nilp}}^{\circ}\longrightarrow N(x)/N(x)^{\circ}=C(x).$$ It therefore suffices to prove that $\pi_0(H_{x_{\nilp}})$ is abelian. In other words, we may assume that $L(x)'$ is simply-connected. We may also assume $L(x)'$ to be simple.
		
		Consider the nilpotent orbit $\mathcal{O}\coloneqq L(x)'\cdot x_{\nilp}\subset[\g_{x_{\semi}},\g_{x_{\semi}}]_{x_{\nilp}}$. Since $L(x)'$ is connected and simply-connected, we have a group isomorphism $C(x)\cong\pi_1(\mathcal{O})$. We also know that $[\g_{x_{\semi}},\g_{x_{\semi}}]$ is of Lie type $A$ or $C$. This fact combines with \cite[Corollary 6.1.6]{col-mcg:93} to imply that $\pi_1(\mathcal{O})$ is abelian. It follows that $C(x)$ is abelian.
	\end{proof}
	
	\begin{corollary}
		If $G$ is of Lie type $A$, then $A(x)$ is abelian for all $x\in\g$. 
	\end{corollary}
	
	\begin{proof}
		If $\l\subset\g$ is a Levi subalgebra, then every simple ideal of $[\l,\l]$ has Lie type $A$. The result now follows immediately from Proposition \ref{Proposition: Lie types}. 
	\end{proof}
	
	\begin{remark}\label{Remark: Non-abelian}
		The group $C(x)$ is non-abelian in general. Let $G$ be simple and simply-connected. If $x\in\g$ is nilpotent, then $C(x)\cong\pi_1(\mathcal{O})$ for $\mathcal{O}=Gx\subset\g$. An inspection of the tables in \cite[Section 8.4]{col-mcg:93} reveals that $\pi_1(\mathcal{O})$ need not be abelian. By Corollary \ref{Corollary: Simple}, $A(x)$ is not abelian in general.
	\end{remark}
	
	\subsection{Induced actions of $A(x)$}\label{Subsection: Actions on level sets}
	We now obtain an action of $A(x)$ on $\mu^{-1}(x)$ for $x\in\D_M$. Our first step is the following complex-geometric analogue of Theorem \ref{Theorem: New slice theorem}(ii).
	
	\begin{proposition}\label{Proposition: Trivial action}
		Suppose that $M$ is a non-empty, irreducible Hamiltonian $G$-variety (resp. non-empty, connected holomorphic Hamiltonian $G$-space) with moment map $\mu:M\longrightarrow\g$. If $x\in\mathcal{D}_M$, then $N(x)^{\circ}\subset G_x$ acts trivially on $\mu^{-1}(x)$.
	\end{proposition}
	
	\begin{proof}
		As an open, $G$-invariant subvariety (resp. submanifold) of $M$, $\mu^{-1}(\mathcal{D}_M)$ is a Hamiltonian $G$-variety (resp. holomorphic Hamiltonian $G$-space) with moment map $\mu\big\vert_{\mu^{-1}(\mathcal{D}_M)}:\mu^{-1}(\mathcal{D}_M)\longrightarrow\g$. It follows that the image of $\mathrm{d}\mu_p:T_pM\longrightarrow\g$ is contained in $T_x\mathcal{D}_M$ for all $p\in\mu^{-1}(x)$. The image is precisely $(\g_p)^{\perp}\subset\g$, where ``$\perp$" is with respect to the form $\langle\cdot,\cdot\rangle:\g\otimes\g\longrightarrow\mathbb{C}$ in Section \ref{Subsection: Complex reductive algebraic groups}. This yields $(T_x\mathcal{D}_M)^{\perp}\subset\g_p$ for all $p\in\mu^{-1}(x)$. 
		
		We now compute $(T_x\mathcal{D}_M)^{\perp}$. To this end, write $x=z+y$ with $z\in\z(\g)$ and $y\in[\g,\g]$. Note also that $\D_M=\z(\g)+\D$ for a $[G,G]$-decomposition class $\D\subset[\g,\g]$. One deduces that $T_x\D_M=\z(\g)\oplus T_y\D$ and $(T_x\D_M)^{\perp}=(T_y\D)^{\perp}$, where the latter is the annihilator of $T_y\D\subset[\g,\g]$ with respect to the Killing form on $[\g,\g]$. The proof of \cite[Proposition 4.35]{cro-may:22} tells us that $$(T_y\D)^{\perp}=[[\g,\g]_{y_{\semi}},[\g,\g]_{y_{\semi}}]_{y_n},$$ where $y_{\semi}$ (resp. $y_{\nilp}$) is the semisimple (resp. nilpotent) part of $y$ in $[\g,\g]$. On the other hand, $\g_{x_{\semi}}=\z(\g)+[\g,\g]_{y_{\semi}}$ and $x_{\nilp}=y_{\nilp}$; see Proposition \ref{Proposition: Comparative Jordan decomposition}. These last few sentences imply that $(T_x\D_M)^{\perp}=[\g_{x_{\semi}},\g_{x_{\semi}}]_{x_{\nilp}}$.
		
		The last sentences of the previous two paragraphs yield $[\g_{x_{\semi}},\g_{x_{\semi}}]_{x_{\nilp}}\subset\g_p$ for all $p\in\mu^{-1}(x)$. In other words, $N(x)^{\circ}\subset G_p$ for all $p\in\mu^{-1}(x)$. 
	\end{proof}
	
	\begin{remark}
		Retain the setup of Proposition \ref{Proposition: Trivial action}. It is natural to wonder if $N(x)$ acts trivially on $\mu^{-1}(x)$ for all $x\in\D_M$. This turns out to be false in general. To obtain an explicit counter-example, note that the defining representation of $G=\operatorname{Sp}_{2n}$ on $M=\IC^{2n}$ preserves the standard symplectic form on the latter. This action is Hamiltonian with moment map $\mu:\IC^{2n}\longrightarrow\mathfrak{sp}_{2n}$ satisfying $\mu(\IC^{2n})= \overline{\mathcal{O}_{\text{min}}}$ and $\mu^{-1}(0)=\{0\}$ \cite[Example 4.5]{bry-kos:94}, where $\mathcal{O}_{\text{min}}\subset\mathfrak{sp}_{2n}$ is the minimal nilpotent orbit. It follows that $\D_M=\mathcal{O}_{\text{min}}$ and $\mu^{-1}(\mathcal{D}_M)=\mathbb{C}^{2n}\setminus\{0\}$. On the other hand, note that $N(x)=(\mathrm{Sp}_{2n})_x$ for all $x\in\mathcal{D}_M=\mathcal{O}_{\text{min}}$. We also note that $-I_{2n}\in(\mathrm{Sp}_{2n})_x$ for all $x\in\mathfrak{sp}_{2n}$, and that this element does not stabilize any element of $\mathbb{C}^{2n}\setminus\{0\}$. This implies that $N(x)$ acts non-trivially on $\mu^{-1}(x)$ for all $\mathcal{D}_M$.
	\end{remark}
	
	Let $M$ be a non-empty, connected holomorphic Hamiltonian $G$-space with moment map $\mu:M\longrightarrow\g$. Proposition \ref{Proposition: Trivial action} implies that the $G_x$-action on $\mu^{-1}(x)$ descends to one of $A(x)$ for all $x\in\mathcal{D}_M$. Given $x\in\D_M$, one might wonder if a smaller quotient of $G_x$ can be shown to inherit an action on $\mu^{-1}(x)$. We use the following two results to answer this question.
	
	\begin{proposition}\label{Proposition: Largest}
		Suppose that $x\in\g$. Let $\D\subset\g$ be the decomposition class containing $x$. The algebraic group $N(x)^{\circ}$ is the largest closed subgroup $H\subset G_x$ with the following property: if $M$ is any non-empty, connected holomorphic Hamiltonian $G$-space with moment map $\mu:M\longrightarrow\g$ that satisfies $\D_M=\D$, then $H$ acts trivially on $\mu^{-1}(x)$.
	\end{proposition} 
	
	\begin{proof}
		Proposition \ref{Proposition: Trivial action} reduces us to finding a non-empty, connected holomorphic Hamiltonian $G$-space $M$ with moment map $\mu:M\longrightarrow\g$, such that $\D_M=\D$ and the kernel of the $G_x$-action on $\mu^{-1}(x)$ is contained in $N(x)^{\circ}$. To this end, let $M$ and $\mu$ be as defined in the proof of Proposition \ref{Proposition: Canonical}. The same proposition implies that $x=\mu([h],y)$ for some $y\in\D$ and $[h]\in G/N(y)^{\circ}$, i.e. $x=\Ad_h(-y)$. 
		
		Suppose that $g\in G_x$ belongs to the kernel of the $G_x$-action on $\mu^{-1}(x)$. In particular, $g([h],y)=([h],y)$. It follows that $gh=hk$ for some $k\in N(y)^{\circ}$, or equivalently $g=hkh^{-1}$. Since $x=\Ad_h(-y)$, we have $N(x)^{\circ}=hN(-y)^{\circ}h^{-1}$. It is also clear that $N(-y)=N(y)$. These last three sentences imply that $x\in N(x)^{\circ}$. 
	\end{proof}
	
	We now answer the question posed above; it is the following immediate consequence of Proposition \ref{Proposition: Largest}.
	
	\begin{corollary}\label{Corollary: Smallest quotient}
		Suppose that $x\in\g$. Let $\D\subset\g$ be the decomposition class containing $x$. The algebraic group $A(x)$ is the smallest quotient $Q$ of $G_x$ with the following property: if $M$ is any non-empty, connected holomorphic Hamiltonian $G$-space with moment map $\mu:M\longrightarrow\g$ that satisfies $\D_M=\D$, then the $G_x$-action on $\mu^{-1}(x)$ descends to one of $Q$ on $\mu^{-1}(x)$.
	\end{corollary}
	
	\section{Hamiltonian actions of symplectic groupoids and symplectic slices}\label{sec:symplectic slices}
	
	In this section, we prove symplectic slice theorems for a Hamiltonian action of a complex reductive algebraic group $G$ on a holomorphic Hamiltonian $G$-space or Hamiltonian $G$-variety $M$. Parts of these results are obtained as applications of an {\it abstract symplectic slice theorem} for Hamiltonian actions of symplectic groupoids on symplectic manifolds. We begin with an account of Lie groupoids and Hamiltonian actions of symplectic groupoids. This allows us to prove Main Theorem \ref{thm:abs-symp-cr-sec} as Theorem \ref{t:abs-symp-cr-sec}. Main Theorem \ref{Theorem: General slices} is proved by combining Proposition \ref{Proposition: Semisimple case}, Proposition \ref{Proposition: Natural transversal}, Corollary \ref{Corollary: Poisson slices}, Proposition \ref{p:decompcontainDx}, Corollary \ref{c:slice-generalUx}, Corollary \ref{Corollary: Nice}, Proposition \ref{Proposition: Slodowy saturation}, and Corollary \ref{c:slice-general}.
	
	\subsection{Lie groupoids}\label{sec:Liegroupoids}
	A Lie groupoid (resp. holomorphic Lie groupoid, complex algebraic groupoid) is a groupoid object $\G\tto X$ in the category of smooth manifolds (resp. complex manifolds, complex algebraic varieties). We impose the additional requirement that $\G$ and $X$ be smooth if $\G\tto X$ is to be called a complex algebraic groupoid. Regardless of the category in question, this groupoid structure amounts to having morphisms with various properties. These morphisms are denoted $\sss:\G\longrightarrow X$, $\ttt:\G\longrightarrow X$, $\mmm:\G\fp{\sss}{\ttt}\G\longrightarrow\G$, $\iii:\G\longrightarrow\G$, and $\mathsf{1}:X\longrightarrow\G$, and called the source, target, multiplication, inversion, and identity bisection morphisms, respectively; the notation $\G\fp{\sss}{\ttt}\G$ indicates the subset of $\G\times\G$ given by pairs $(g, h)$ such that $\sss(g)=\ttt(h)$.
	Set $gh\coloneqq\mmm(gh)$ for $(g,h)\in\G\fp{\sss}{\ttt}\G$ and $g^{-1}\coloneqq\iii(g)$ for $g\in\G$. The aforementioned properties of these morphisms are as follows:
	\begin{itemize}
		\item[\textup{(i)}] $\sss(gh)=\sss(h)$ and $\ttt(gh)=\ttt(g)$ for all $(g,h)\in\G\fp{\sss}{\ttt}\G$;
		\item[\textup{(ii)}] $(gh)k=g(hk)$ for all $g,h,k\in\G$ with $\sss(g)=\ttt(h)$ and $\sss(h)=\ttt(k)$;
		\item[\textup{(iii)}] $\sss(\mathsf{1}(x))=x=\ttt(\mathsf{1}(x))$ for all $x\in X$;
		\item[\textup{(iv)}] $\mathsf{1}(\ttt(g))g=g$ and $g\mathsf{1}(\sss(g))=g$ for all $g\in\G$;
		\item[\textup{(v)}] $\sss(g^{-1})=\ttt(g)$ and $\ttt(g^{-1})=\sss(g)$ for all $g\in\G$;
		\item[\textup{(vi)}] $gg^{-1}=\mathsf{1}(\ttt(g))$ and $g^{-1}g=1(\sss(g))$ for all $g\in\G$.
	\end{itemize} 
	Note that a Lie group (resp. complex Lie group, complex algebraic group) is a Lie groupoid (resp. holomorphic Lie groupoid, complex algebraic groupoid) $\G\tto X$ for which $X$ is a singleton.

	\begin{example}[Action groupoids]\label{Example: Action groupoids}
		Let $G$ be a Lie group (resp. complex Lie group, complex algebraic group). Suppose that $M$ is a smooth manifold (resp. complex manifold, smooth complex algebraic variety) carrying a smooth (resp. holomorphic, algebraic) action of $G$. The \textit{action groupoid} associated to this situation takes the form $G\times M\tto M$, with the following maps and operations: source $\sss(g,p)\coloneqq p$ for all $(g,p)\in G\times M$, target $\ttt(g,p) \coloneqq gp$ for all $(g,p)\in G\times M$,  identity  $\mathsf{1}(p)\coloneqq(e,p)$ for all $p\in M$, multiplication
		$$(g_1,p_1)(g_2,p_2)\coloneqq (g_1g_2,p_2)$$ for all $((g_1,p_1),(g_2,p_2))\in (G\times M)\fp{\sss}{\ttt} (G\times M)$, and inversion $(g,p)^{-1} \coloneqq (g^{-1}, gp)$ for all $(g,p)\in G\times M$.
		In this way, $G\times M\tto M$ is a Lie groupoid (resp. holomorphic Lie groupoid, complex algebraic groupoid).
	\end{example}
	
	Let $\G\tto X$ be a Lie groupoid (resp. holomorphic Lie groupoid, complex algebraic groupoid). Given an open subset $U\subset X$, we have the restriction
	$$\G_{U}^U\coloneqq\sss^{-1}(U)\cap\ttt^{-1}(U)\tto U.$$
	Note that $\sss$ and $\ttt$ are necessarily submersions. Let us also note that $\mathsf{i}:X\longrightarrow\G$ includes $X$ as a submanifold (resp. complex submanifold, smooth closed subvariety) of $\G$. The \textit{Lie algebroid} of $\G$ is $$\mathrm{Lie}(\G)\coloneqq\left(\mathrm{ker}(\mathrm{d}\sss)\big\vert_X\longrightarrow X\right),$$ with anchor map obtained by restricting $\mathrm{d}\ttt:T\G\longrightarrow TX$ to $\mathrm{Lie}(\G)\subset T\G$. 
	
	An \textit{action} of a Lie groupoid $\G\tto X$ on a smooth manifold (resp. complex manifold, complex algebraic variety) $M$ consists of morphisms $\mu:M\longrightarrow X$ and $$\G\times_X M\coloneqq\G\fp{\sss}{\mu}M\longrightarrow M,\quad(g,p)\mapsto gp$$ that satisfy the following properties: $\mu(gp)=\ttt(g)$ for all $(g,p)\in\G\times_X M$ and $\mathsf{1}(\mu(p))p=p$ for all $p\in M$. One calls $\mu$ the \textit{moment map} for this action.
	
	\subsection{Symplectic groupoids} The following definition is written in such a way that its specializations to the smooth, holomorphic, and complex-algebraic categories are clear; see \cite{wei:87,cos-daz-wei:87} for its origins.
	
	\begin{definition}
		Let $\G\tto X$ be a Lie groupoid. One calls $\G$ a \textit{symplectic groupoid} if it comes equipped with a symplectic form $\Omega$ such that the graph of multiplication $$\{(g,h,gh): (g,h)\in\G\fp{\sss}{\ttt}\G\}$$ is Lagrangian in $\G\times\G\times\G^{-}$, where $\G^{-}$ denotes the result of equipping $\G$ with $-\Omega$.
	\end{definition}
	
	\begin{example}[Cotangent groupoids of Lie groups]\label{Example: Cotangent groupoids}
		Let $G$ be a Lie group with Lie algebra $\g$. The negated exterior derivative of the tautological $1$-form yields a symplectic form on $T^*G$. On the other hand, we may use the left trivialization to identify $T^*G$ with $G\times\g^*$. The morphisms $$\sss:T^*G\longrightarrow\g^*,\quad (g,\xi)\mapsto \xi,\quad\ttt:T^*G\longrightarrow\g^*,\quad (g,\xi)\mapsto \Ad_g^*(\xi),$$
		and $$T^*G\fp{\sss}{\ttt}T^*G\longrightarrow T^*G,\quad ((g,\xi),(h,\eta))\mapsto(gh,\eta)$$ constitute the source, target, and multiplication morphisms of a Lie groupoid structure on $T^*G$. 
		
		This Lie groupoid structure on $T^*G$ combines with the aforementioned symplectic structure to render $T^*G$ a symplectic groupoid, called the \textit{cotangent groupoid} of $G$. Note that this definition of a cotangent groupoid applies analogously if $G$ is a complex Lie group or complex algebraic group, yielding cotangent groupoids in holomorphic and complex-algebraic contexts.
	\end{example}
	
	\begin{example}[The Lie algebroid of a symplectic groupoid]\label{Example: The Lie algebroid}
		If $\G\tto X$ is a symplectic groupoid, then $X$ carries a unique Poisson structure for which $\ttt$ is a Poisson morphism and $\sss$ is an anti-Poisson morphism. The Poisson structure on $X$ renders $T^*X\longrightarrow X$ a Lie algebroid; see Section \ref{Subsection: Pre-Poisson submanifolds}. It turns out that the map
		$$\mathrm{Lie}(\G)\longrightarrow T^*X,\quad (x,v)\mapsto\Omega_x(v,\cdot)\big\vert_{T_xX}$$ is a Lie algebroid isomorphism, where $\Omega$ is the symplectic form on $\G$. With this in mind, we freely identify $\mathrm{Lie}(\G)$ with $T^*X$.
	\end{example}
	
	In light of Section \ref{Subsection: Pre-Poisson submanifolds} and Example \ref{Example: The Lie algebroid}, we may consider pre-Poisson submanifolds in the base space of a symplectic groupoid. This leads to the following key ingredient of \cite{cro-may:22}.
	
	\begin{definition}\label{Definition: Stabilizer subgroupoid}
		Consider a symplectic groupoid $\G\tto X$ and pre-Poisson submanifold $S\subset X$. An isotropic Lie subgroupoid $\H\tto S$ of $\G$ with Lie algebroid $L_S\subset T^*X$ is called a \textit{stabilizer subgroupoid}.
	\end{definition}
	
	\subsection{Hamiltonian actions of symplectic groupoids}
	We now give a definition of Hamiltonian symplectic groupoid actions that lends itself easily to specializations to the smooth, holomorphic, and complex-algebraic categories. Let $M$ be a symplectic manifold, holomorphic symplectic manifold, or smooth complex symplectic variety. The following is due to Mikami--Weinstein \cite{mik-wei:88}.
	
	\begin{definition}\label{Definition: Hamiltonian groupoid action}
		Suppose that a symplectic groupoid $\G$ acts on $M$ via a moment map $\mu:M\longrightarrow X$. One calls $M$ a \textit{Hamiltonian $\G$-space} if $$\{(g,p,gp):(g,p)\in\G\fp{\sss}{\mu}M\}$$ is coisotropic in $\G\times M\times M^{-}$, where $M^{-}$ results from equipping $M$ with the negated symplectic form.
	\end{definition}
	
	\begin{example}[Recasting Hamiltonian actions of Lie groups]\label{Example: Recasting}
		Let $G$ be a Lie group with Lie algebra $\g$. Recall our description of the cotangent groupoid $T^*G\tto\g^*$ in Example \ref{Example: Cotangent groupoids}. Consider a Hamiltonian $G$-space $M$ with moment map $\mu:M\longrightarrow\g^*$. It follows that $T^*G=G\times\g^*$ acts on $M$ via $\mu$, i.e.
		$$(g,\xi)p\coloneqq gp\quad\text{if }\xi=\mu(p).$$ Furthermore, this action renders $M$ a Hamiltonian $T^*G$-space in the sense of Definition \ref{Definition: Hamiltonian groupoid action}. These last three sentences define a bijective correspondence between Hamiltonian $G$-spaces and Hamiltonian $T^*G$-spaces; see \cite{mik-wei:88}. We have analogous correspondences in the holomorphic and complex-algebraic categories.
	\end{example}
	
	The following theorem is proved in \cite{cro-may:22}. Its analogues in the holomorphic and algebraic categories are \cite[Theorem B]{cro-may:22} and \cite[Theorem C]{cro-may:22}, respectively.
	
	\begin{theorem}\label{Theorem: Reduction}
		Consider a symplectic groupoid $\G\tto X$ and stabilizer subgroupoid $\H\tto S$ over a pre-Poisson submanifold $S\subset X$. Suppose that $M$ is a Hamiltonian $\G$-space with symplectic form $\omega$ and moment map $\mu:M\longrightarrow X$. Consider the inclusion $j:\mu^{-1}(S)\longrightarrow M$ and quotient map $\pi:\mu^{-1}(S)\longrightarrow\mu^{-1}(S)/\H$. If $\H$ acts freely and properly on $\mu^{-1}(S)$, then there exists a unique symplectic form $\overline{\omega}$ on $\mu^{-1}(S)/\H$ satisfying $\pi^*\overline{\omega}=j^*\omega$.
	\end{theorem}
	
	\begin{definition}\label{Definition: Generalized reduction}
		Suppose that the hypotheses of Theorem \ref{Theorem: Reduction} are satisfied. We call $$M\sll{S,\H}{\G}\coloneqq (\mu^{-1}(S)/\H,\overline{\omega})$$ the \textit{Hamiltonian reduction} of $M$ along $S$ with respect to $\H$. 
	\end{definition}
	
	\subsection{An abstract symplectic slice theorem}\label{Subsection: Abstract symplectic slice theorem}
	In this section, we work in the smooth category. All definitions and results apply analogously in the holomorphic and complex-algebraic categories. We frame the discussion in the setting of Hamiltonian actions of symplectic groupoids \cite{mik-wei:88}, and use the general reduction process of Theorem \ref{Theorem: Reduction}.
	
	Let $\G\tto X$ be a symplectic groupoid with symplectic form $\Omega$. Recall that there is a unique Poisson structure $\sigma$ on $X$ such that $\ttt\colon \G\longrightarrow X$ is a Poisson morphism.
	The orbits of $\G$ in $X$ are immersed submanifolds; if $\G$ is source-connected, the orbits coincide with the symplectic leaves of the Poisson structure, and more generally they are disjoint unions of symplectic leaves. In particular, a Poisson transversal $S\subset X$ is automatically transverse to $\G$-orbits. In analogy with Definition \ref{Definition: Slice}, we refer to such a transversal as a \textit{Poisson slice} in $X$.
	Write $\ol{X}$ for $X$ with the Poisson structure; the map $(\ttt,\sss)\colon \G\longrightarrow X\times \ol{X}$ is then a Poisson morphism. Given a Poisson transversal $S\subset X$, it follows that
	\[ \G_S:=\sss^{-1}(S), \quad \G^S:=\ttt^{-1}(S),\quad\text{and}\quad \G^S_S:=\sss^{-1}(S)\cap\ttt^{-1}(S) \]
	are symplectic submanifolds of $\G$. We also see that $\G^S_S\tto S$ is a symplectic subgroupoid of $\G$. With this in mind, $\G_S$ and $\G^S$ are Hamiltonian $\G\times\G_S^S$-spaces.
	
	Let $(M,\omega,\mu)$ be a Hamiltonian $\G$-space with action map $a\colon \G\times_X M\longrightarrow M$. The Hamiltonian condition is equivalent to the following multiplicative property:
	\begin{equation} 
		\label{e:multprop}
		a^*\omega=\pr_1^*\Omega+\pr_2^*\omega, 
	\end{equation}
	where $\pr_1\colon \G\times_X M\longrightarrow \G$ and $\pr_2\colon \G\times_X M\longrightarrow M$ are the projection maps to the two factors; cf. \cite{Mol2024}. A submanifold $S\subset M$ is called a {\it symplectic slice} if it is symplectic and transverse to the $\G$-orbits.
	
	\begin{theorem}[Abstract symplectic slices]
		\label{t:abs-symp-cr-sec}
		Consider a symplectic groupoid $\G\tto X$ and Hamiltonian $\G$-space $(M,\omega,\mu)$. Suppose that $S\subset X$ is a Poisson slice. The following statements are true:
		\begin{itemize}
			\item[\textup{(i)}] $(\mu^{-1}(S),\omega\big\vert_{\mu^{-1}(S)},\mu\big\vert_{\mu^{-1}(S)})$ is a symplectic slice and a Hamiltonian $\G_S^S$-space;
			\item[\textup{(ii)}] the $\G$-saturation $\G\mu^{-1}(S)\subset M$ of $\mu^{-1}(S)$ is open in $M$;
			\item[\textup{(ii)}] the action map $\G\fp{\sss}{\mu}M\longrightarrow M$ restricts and descends to an isomorphism
			\[ (\G_S\times_S \mu^{-1}(S))/\G_S^S\overset{\cong}\longrightarrow \G\mu^{-1}(S) \] of Hamiltonian $\G$-spaces,
			where the symplectic structure on $(\G_S\times_S \mu^{-1}(S))/\G_S^S$ is obtained by restricting and descending the product symplectic structure on $\G_S\times \mu^{-1}(S)$ to $(\G_S\times_S \mu^{-1}(S))/\G_S^S$.
		\end{itemize}
	\end{theorem}
	\begin{proof}
		Since the pre-image of a Poisson transversal under a Poisson morphism is a Poisson transversal \cite[Lemma 7]{fre-mar:17}, $Y\coloneqq\mu^{-1}(S)\subset M$ is a symplectic submanifold. A proof similar to that of Proposition \ref{Proposition: Pullback}, with the element $x$ replaced by a section of the Lie algebroid of $\G$, shows that $Y$ is transverse to the $\G$-orbits in $M$; hence it is a symplectic slice. It is also clear that the restriction $\mu\big\vert_S: Y\longrightarrow S$ is a Poisson morphism. By construction, the action of $\G_S^S$ preserves $Y$. Pulling back equation \eqref{e:multprop} to the submanifold $\G_S\times_S Y$ of $\G\times_X M$ shows that the pair $(\omega|_Y,\Omega|_{\G_S^S})$ satisfies the multiplicativity property \eqref{e:multprop} for the action of the Lie groupoid $\G_S^S$ on $Y$. It follows that the action of $\G_S^S$ on $Y$ is Hamiltonian.
		
		The fiber product $\G_S\times_S Y$ is smooth since $\sss: \G_S\longrightarrow S$ is a submersion. We also see that $\G_S\times_S Y$ carries an action of $\G\times \G_S^S$, where $\G$ acts by left multiplication on $\G_S$ with moment map $\ttt\colon \G_S\longrightarrow X$, and $\G_S^S$ acts diagonally on $\G_S$ by right multiplication and on $Y$. The action of $\G_S^S$ on $\G_S$ is free and proper, implying that the quotient $(\G_S\times_S Y)/\G_S^S$ is a smooth manifold with an induced smooth $\G$-action. Since $S$ is a Poisson transversal, $\G S$ is open, and thus $\G Y=\mu^{-1}(\G S)$ is open. The action $a$ restricts to a map $\G_S\times_S S\longrightarrow \G S$, which descends to a smooth $\G$-equivariant morphism
		\[ \ol{a}\colon (\G_S\times_S S)/\G_S^S\longrightarrow \G S.\]
		By the definition of $\G_S^S$, $\ol{a}$ is a bijection. It is then seen to be a diffeomorphism by the inverse function theorem. The equivariance of $\mu$ now implies that the corresponding map, induced by the restriction of the action $a\colon \G\times_X M\longrightarrow M$ to $\G_S\times_S Y \longrightarrow \G Y$,
		\[ \ol{a}\colon(\G_S\times_S Y)/\G_S^S\longrightarrow \G Y, \]
		is also a $\G$-equivariant diffeomorphism.
		
		To describe the Hamiltonian $\G$-space structure on $(\G_S\times_S Y)/\G_S^S$, we use the general procedure for reduction along a submanifold developed in \cite{cro-may:22}. The submanifold $\G_S=\sss^{-1}(S)$ is symplectic. Note that the product $\G_S \times Y$ carries a Hamiltonian action of the product symplectic groupoid $\ol{\G_S^S}\times \G_S^S\tto \ol{S}\times S$ with moment map
		\[ \sss\times \mu_Y \colon \G_S\times Y\longrightarrow \ol{S}\times S,\] 
		where $\ol{\G_S^S}$ acts on $\G_S$ by right multiplication.
		
		We claim that the diagonal subgroupoid $\Delta \G_S^S\tto \Delta S$ is a stabilizer Lie subgroupoid of $\ol{\G_S^S}\times \G_S^S\tto \ol{S}\times S$, in the sense of Definition \ref{Definition: Stabilizer subgroupoid}. To see this, note that $\Delta \G_S^S\subset \ol{\G_S^S}\times \G_S^S$ is a Lagrangian Lie subgroupoid of $\ol{\G_S^S}\times \G_S^S\tto \ol{S}\times S$. We also note that
		\[ ((-\sigma)^{\vee}\times \sigma^{\vee})^{-1}(T\Delta S)\cap (T\Delta S)^\perp\]
		is the conormal bundle to the diagonal, where $\sigma$ is the Poisson bivector field on $X$. This is a smooth vector bundle and the Lie algebroid of $\Delta \G_S^S$. This establishes our claim. Theorem \ref{Theorem: Reduction} now implies that
		\[ (\G_S\times_S Y)/\G_S^S \]
		is symplectic. The commuting Hamiltonian action of $\G$ on $\G_S$ by left multiplication induces a Hamiltonian action of $\G$ on this quotient; cf. \cite[Proposition 1.4.19]{Mol-thesis}.
		
		It remains only to prove that the $\G$-equivariant diffeomorphism $\ol{a}\colon(\G_S\times_S Y)/\G_S^S\longrightarrow \G Y$ is an isomorphism of Hamiltonian $\G$-spaces. We first note that $\ol{a}$ intertwines the moment maps. We also recall that the symplectic structure on $(\G_S\times_S Y)/\G_S^S$ is obtained by restricting $\pr_1^*\Omega+\pr_2^*\omega$ to $\G_S\times_S Y$, and then descending to the quotient. Taking the pullback of \eqref{e:multprop} to $\G_S\times_S Y$ and descending to the quotient shows that $\ol{a}^*\omega$ equals the symplectic structure on $(\G_S\times_S Y)/\G_S^S$. This shows that $\ol{a}$ also intertwines symplectic structures.
	\end{proof}
	
	\subsection{The natural slice at a semisimple element}\label{Subsection: The natural}
	We now introduce complex-geometric analogues of the natural slices discussed in Section \ref{Subsection: Slices for compact group actions}. Let $G$ be a connected complex reductive algebraic group with Lie algebra $\g$. Suppose that $x\in\g$ is semisimple. It follows that $\g_x\subset\g$ is a reductive subalgebra. We define the {\it natural slice} $\S_x$ at $x$ to be
	\[ \S_x\coloneqq \{y \in \g\mid \g_y\subset \g_x\} \subset \g_x.\]
	One has the following consequences of \cite[Corollary 35.3.2, Lemma 39.4.2, and Lemma 39.4.3]{Tauvel-Yu}:
	
	\begin{proposition}
		\label{p:Tauvel-Yu-equiv-cond}
		Suppose that $x\in\g$ is semisimple. The subset $\S_x\subset\g$ is $G_x$-invariant and open. Moreover, the following conditions are equivalent:
		\begin{itemize}
			\item[\textup{(i)}] $y \in \S_x$;
			\item[\textup{(ii)}] $x \in \z(\g_y)$; 
			\item[\textup{(iii)}] $[\g,x]\subset [\g,y]$; 
			\item[\textup{(iv)}] $\z(\g_x)\subset \z(\g_y)$;  
			\item[\textup{(v)}] $\g=[\g,y]+\g_x$.
		\end{itemize}
	\end{proposition}
	Recall that $\g_{\mathsf{semi}}\subset\g$ and $\N\subset \g$ are the subsets of semisimple and nilpotent elements in  
	$\g$, respectively. For a Levi subalgebra $\l\subset \g_x$, set
	\[ \N_\l\coloneqq\N\cap \l=\N\cap [\l,\l].\]
	We may decompose $\N_\l$ as a union of nilpotent $L$-orbits $$\N_\l=\bigcup_{\O\subset \N_\l} \O,$$ where $L\subset G$ is the Levi subgroup integrating $\l$. Given a Levi subalgebra $\l\subset \g_x$ and nilpotent $L$-orbit $\O\subset\N_{\l}$, set
	\[ \D^{G_x}_{\l,\O}\coloneqq G_x(\z(\l)_{\mathsf{gen}}+\O).\]
	We have that
	\[ G\D^{G_x}_{\l,\O}=G(\z(\l)_{\mathsf{gen}}+\O)=\D_{\l,\O} \]
	is the $G$-decomposition class associated to the pair $(\l,\O)$. It is also clear that $\D^{G_x}_{\l,\O}$ only depends on the $G_x$-conjugacy class of the pair $(\l,\O)$.
	\begin{proposition}
		\label{p:jordan}
		Suppose that $x\in\g$ is semisimple and $y \in \S_x$. The following statements are true:
		\begin{itemize}
			\item[\textup{(i)}] $y_\semi \in \S_x$;
			\item[\textup{(ii)}] one has
			\begin{equation} 
				\label{e:jordanUx}
				\S_x=\bigcup_{s \in \S_x\cap \g_{\mathsf{semi}}} (s+\N_{\g_s})=\bigcup_{\l} (\z(\l)_{\mathsf{gen}}+\N_\l)=\bigcup_{[(\l,\O)]} \D^{G_x}_{\l,\O},
			\end{equation}
			where $\l$ ranges over the Levi subalgebras of $\g_x$, and $[(\l,\O)]$ ranges over the $G_x$-conjugacy classes of pairs $(\l,\O)$ with $\O\subset \N_\l$ a nilpotent $L$-orbit.
		\end{itemize}
	\end{proposition}
	\begin{proof}
		Proposition \ref{p:Tauvel-Yu-equiv-cond}(iv) tells us that $\z(\g_x)\subset \z(\g_y)$. By \cite[Proposition 39.1.1]{Tauvel-Yu},
		\[ \z(\g_y)=\z(\g_{y_\semi})\oplus \z([\g_{y_\semi},\g_{y_\semi}]_{y_\nilp}).  \]
		The same result implies that the two summands are precisely the subsets of semisimple and nilpotent elements, respectively, of $\z(\g_y)$. Since $\z(\g_x)$ consists entirely of semisimple elements, we must have $\z(\g_x)\subset \z(\g_{y_\semi})$. Applying Proposition \ref{p:Tauvel-Yu-equiv-cond}(iv) a second time, we deduce $y_\semi \in \S_x$.
		
		It remains only to prove (ii). For the first equality in \eqref{e:jordanUx}, note that $y=y_\semi+y_\nilp \in y_\semi+(\N\cap \g_{y_\semi})$. We also note that $y_\semi \in \S_x\cap \g_{\mathsf{semi}}$ by (i). It follows that the left-hand side of \eqref{e:jordanUx} is contained in the right-hand side. Conversely, suppose that $s \in \S_x\cap \g_{\mathsf{semi}}$ and $s+n \in s+\N\cap \g_s$. This implies that $y=s+n$ is the Jordan decomposition of $y$, so that $\g_y=\g_s\cap \g_n\subset \g_x$. We conclude that $y \in \S_x$. For the second equality in \eqref{e:jordanUx}, recall that the stabilizer of a semisimple element $s$ is the Levi subalgebra $\g_s=\l$. We also know that $\z(\l)_{\mathsf{gen}}$ is precisely the set of elements of $\g$ with stabilizer equal to $\l$.
	\end{proof}
	
	\begin{proposition}\label{Proposition: Semisimple case}
		If $x\in\g$ is semisimple, then $\S_x$ is a Poisson transversal in $\g$ and $T_x\S_x\oplus [\g,x]=\g$.
	\end{proposition}
	\begin{proof}
		Consider an element $y \in \S_x$. We have $T_y(Gy)=[\g,y]$ and $T_y(\S_x)=\g_x$. By Proposition \ref{p:Tauvel-Yu-equiv-cond}(v), $T_y(Gy)+T_y\S_x=\g$. It follows that $Gy$ intersects $\S_x$ transversely at $y$. Since $x$ is semisimple, we have $\g=\g_x\oplus \g_x^\perp$, where $\g_x^\perp$ is the annihilator of $\g_x$ with respect to the $G$-invariant pairing on $\g$ in Section \ref{Subsection: Complex reductive algebraic groups}. Noting that $y \in \g_y\subset \g_x$, the adjoint action of $y$ on $\g$ preserves the splitting $\g=\g_x\oplus \g_x^\perp$. The tangent space of the intersection $Gy\cap \S_x$ at $y$ is $[\g,y]\cap \g_x=([\g_x,y]\oplus [\g_x^\perp,y])\cap \g_x=[\g_x,y]$. As the tangent space of the orbit $G_xy\subset \g_x$, it is a symplectic vector space. It follows that the symplectic leaves of $\g$ intersect $\S_x$ transversely in symplectic submanifolds, i.e. $\S_x$ is a Poisson transversal.
	\end{proof}
	
	Applying Theorem \ref{t:abs-symp-cr-sec} to the action of the symplectic groupoid $\G=T^*G=G\times \g^*\tto \g^*=X$ on a Hamiltonian $G$-space $(M,\omega,\mu)$, we obtain the following symplectic slice theorem. The reader may wish to consult our discussion of action groupoids in Example \ref{Example: Action groupoids}.
	
	\begin{corollary}
		\label{c:slice-semisimple}
		Suppose that $(M,\omega,\mu)$ is a holomorphic Hamiltonian $G$-space and $x \in \mu(M)\cap\g_{\mathsf{semi}}$. Then $\mu^{-1}(\S_x)\subset M$ is a symplectic slice and Hamiltonian $\G_{\S_x}^{\S_x}$-space, and there is an isomorphism of Hamiltonian $G$-spaces
		\[ (G\times\mu^{-1}(\S_x))/\G_{\S_x}^{\S_x}\overset{\cong}\longrightarrow G\mu^{-1}(\S_x).\]
		The action groupoid $N_G(G_x)\times \S_x\tto \S_x$ is a subgroupoid of $\G_{\S_x}^{\S_x}$, and $\mu^{-1}(\S_x)$ is a Hamiltonian $N_G(G_x)$-space.
	\end{corollary}
	\begin{proof}
		In this case, $\G_{\S_x}=G\times \S_x\subset G\times \g^*=T^*G$. It follows that $\G_{\S_x}\times_{\S_x}\mu^{-1}(\S_x)\cong G\times \mu^{-1}(\S_x)$. For the last claim, note that the source fiber of $\G_{\S_x}^{\S_x}$ at $y \in \S_x$ is $\{(g,y)\mid gy\in \S_x\}$. The latter contains $N_G(G_x)\times \{y\}$, as $N_G(G_x)$ preserves $\S_x$.
	\end{proof}
	
	\begin{remark}
		Beware that the map $G\times_{G_x}\S_x\longrightarrow G\S_x$ is often not injective. For example, suppose that $\t \subset \g_x$ is a Cartan subalgebra. Then $\t_\reg\coloneqq\t\cap\greg \subset \S_x$ is preserved by $N_G(\t)$, but also $N_G(\t)\subsetneqq G_x$ in general. It can even happen that $N_G(\t)\subsetneqq N_G(G_x)$, in which case $G\times_{N_G(G_x)}\S_x\rightarrow G\S_x$ also fails to be injective. One may remedy this situation as follows. Fix a fundamental domain $\mathfrak{d}\subset \t$ for the action of the Weyl group on $\t$; cf. \cite[Section 5.2]{crooks2019complex} and \cite[Section 2.2]{col-mcg:93}. Let
		\[ \S_x^{\mathfrak{d}}=\{y\in \S_x\mid y_\semi \in G_x\mathfrak{d}\}.\] 
		Arguing as in Lemma \ref{Lemma : K-slice}, one checks that $G\S_x^{\mathfrak{d}}=G\S_x$ and that $G\times_{G_x}\S_x^{\mathfrak{d}}\longrightarrow G\S_x^{\mathfrak{d}}$ is injective.
	\end{remark}
	
	\begin{remark}
		The restricted Lie groupoid $\G_{\S_x}^{\S_x}$ is typically not an action groupoid. On the other hand, $\G_{\S_x}^{\S_x}$ is nearly an action Lie groupoid in the following sense. Let $y \in \S_x$. The source fiber of $y$ in $\G_{\S_x}^{\S_x}$ can be identified with the subset
		\[ G_{\g_y\subset \g_x}\coloneqq\{g\in G\mid g(\g_y)\subset \g_x\}\]
		of $G$. Observe that $G_{\g_y\subset \g_x}$ is preserved under left multiplication by elements of $N_G(\g_x)=N_G(G_x)$. We claim that the quotient of $G_{\g_y\subset \g_x}$ by the $N_G(G_x)$-action is finite; the same will also be true for $G_x$ in place of $N_G(G_x)$ since $N_G(G_x)/G_x$ is finite and isomorphic to $N_W(W_x)/W_x$, where $W$ and $W_x$ are the Weyl groups of $G$ and $G_x$ respectively. To see this, let $\t\subset \g_x$ be a Cartan subalgebra containing $y_\semi$. Let $g \in G_{\g_y\subset \g_x}$ and choose $g' \in G_x$ such that $\Ad_{g'g}(\t)=\t$. Then $g'g\in G_{\g_y\subset \g_x}\subset G$, and it preserves $\t$. In other words, $g'g \in N_G(\t)$. It follows that $G_{\g_y\subset \g_x}/N_G(G_x)$ can be identified with a subset of $N_G(\t)/(N_G(G_x)\cap N_G(\t))\simeq W/N_W(W_x)$, a finite set.
	\end{remark}
	
	We now provide a description of $\S_x$ in combinatorial terms comparable to \cite[Remark 3.7]{ler-mei-tol-woo}. We begin by describing the semisimple elements of $\S_x$. Let $\t \subset \g_x$ be a Cartan subalgebra, and let $\Phi\subset\t^*$ denote the corresponding set of roots. For any $h \in \t$, let $\Phi_h\subset\Phi$ be the subset of roots that vanish on $h$; it is the root system of $\g_h$ with respect to $\t$.
	\begin{proposition}
		\label{p:Uxintersecth}
		Adopt the objects and notation of the previous paragraph. We have $$\S_x\cap \g_{\mathsf{semi}}=G_x(\S_x\cap \t)$$ and
		\[ \S_x \cap \t=\t\setminus\big(\bigcup_{\alpha \in \Phi\backslash \Phi_x}\ker(\alpha)\big). \]
	\end{proposition}
	\begin{proof}
		Since $\g_x$ is reductive with Cartan subalgebra $\t$, any element of $\S_x\cap \g_{\mathsf{semi}}$ is $G_x$-conjugate to some element $h \in \t$. As $\S_x$ is $G_x$-invariant, $h \in \S_x\cap \t$. It follows that $\S_x\cap\g_{\mathsf{semi}}=G_x(\S_x\cap \t)$. For the second claim, note that $h \in \S_x\cap \t$ if $h \in \t$ and $\g_h\subset \g_x$. We also have
		\[ \g_x=\t\oplus \bigoplus_{\alpha \in \Phi_x} \g_\alpha \quad\text{and}\quad \g_h=\t\oplus \bigoplus_{\alpha \in \Phi_h} \g_\alpha. \]
		For $h \in \t$, it follows that $\g_h\subset \g_x$ if and only if $\Phi_h\subset \Phi_x$. The latter condition is equivalent to the following: $h \notin \ker(\alpha)$ for each $\alpha \in \Phi\backslash \Phi_x$.
	\end{proof}
	
	We continue with the notation of the paragraph preceding Proposition \ref{p:Uxintersecth}. Choose a collection $\Delta \subset \Phi$ of simple roots. Each subset $I\subset \Delta$ determines a subset $$\Phi_I\coloneqq\tn{span}_{\mathbb{Z}}(\Delta)\cap \Phi$$ and Levi subalgebra
	\[ \g_I\coloneqq\t\oplus \bigoplus_{\alpha \in \Phi_I} \g_\alpha\]
	of $\g$. Note that $I$ is a choice of simple roots for $\g_I$. Every Levi subalgebra of $\g$ is conjugate to $\g_I$ for some subset $I\subset \Delta$. Set $\N_I\coloneqq\N_{\g_I}$, $W_I\coloneqq W_{G_I}$, and for $J\subset I$ set $\D^I_{J,\O}\coloneqq \D^{G_I}_{\g_J,\O}$. Up to conjugating by $G$, we may assume $\g_x=\g_I$ for a subset $I\subset \Delta$.
	
	\begin{corollary}
		\label{c:UxIJ}
		Suppose that $x\in\t$. Let $I\subset \Delta$ be the subset satisfying $\g_x=\g_I$. We have
		\[ \S_x=G_I\left(\bigcup_{J\subset I} \z(\g_J)_{\mathsf{gen}}+\N_J\right)=\bigcup_{J\subset I}\bigcup_{\O\subset \N_J} \D^I_{J,\O}\quad\text{and} \quad \S_x\cap \t=W_I \bigcup_{J\subset I}\z(\g_J)_{\mathsf{gen}},\]
		where $\D^I_{J,\O}=G_I(\z(\g_J)_{\mathsf{gen}}+\O)$.
	\end{corollary}
	\begin{proof}
		Let $\l\subset \g_x=\g_I$ be a Levi subalgebra. Then $\l$ is $G_I$-conjugate to $\g_J$ for some subset $J$ of the simple roots $I$ of $\g_I$. The descriptions of $\S_x$ now follow from equation \eqref{e:jordanUx}. For the intersection $\S_x\cap \t$, note that $\t\cap \N=\{0\}$. We therefore have
		\[ \S_x \cap \t=\bigg(G_I\bigcup_{J\subset I} \z(\g_J)_{\mathsf{gen}}\bigg)\cap \t=W_I \bigcup_{J\subset I}\z(\g_J)_{\mathsf{gen}},\]
		where we used the fact that $h \in \t$ implies $(G_Ih)\cap \t=W_Ih$.
	\end{proof}
	
	\begin{remark}
		We compare this briefly with the setting in \cite[Remark 3.7]{ler-mei-tol-woo}. Let $\k$ be the Lie algebra of a compact connected Lie group $K$, $\t\subset \k$ the Lie algebra of a maximal torus, and $\t_+$ the positive chamber corresponding to a choice of simple roots $\Delta\subset \Phi^{+}$. Let $J\subset \Delta$ be a subset of the simple roots, and $\k_J$ the Lie subalgebra consisting of $x\in \k$ such that $[x,t]=0$ for all $t\in \bigcap_{\alpha \in J}\ker(\alpha|_{\t})$. The direct analogue of $\z(\g_J)_{\mathsf{gen}}$ is the subset $\z(\k_J)_{\mathsf{gen}}\subset \z(\k_J)$ of elements with centralizer $\k_J$. Observe that $\z(\k_J)_{\mathsf{gen}}$ is not connected, being the complement in $\z(\k_J)$ of the finite collection of real hyperplanes. In \cite{ler-mei-tol-woo}, the authors choose the connected component
		\[ \sigma_J=\t_+\cap \z(\k_J)_{\mathsf{gen}}. \]
		Let $x \in \sigma_I\subset \t_+$. Since $\t_+$ is a fundamental domain for the action of the Weyl group $W$ on $\t$, one obtains a slice $K_x(\cup_{J\subset I}\sigma_J)$; see Section \ref{Subsection: Slices for compact group actions}. Definition \ref{Definition: Slice compact}(iii) implies that the restriction of the action Lie groupoid to this slice is simply the $K_x$-action groupoid.
	\end{remark}
	
	\subsection{The natural slice at a general element}
	Suppose that $x \in \g$. We have $x \in \S_{x_\semi}$, where $\S_{x_\semi}$ is defined in Section \ref{Subsection: The natural}. It follows that $\S_{x_\semi}$ is a $G_{x_\semi}$-invariant Poisson transversal containing $x$. One issue is that this transversal only depends on $x_{\semi}$. In what follows, we describe a modest refinement $\S_x$ of $\S_{x_\semi}$. The decomposition class of $x$ turns out to be minimal in the closure order amongst those decomposition classes that meet $\S_x$.
	
	Recall that
	\[ \S_{x_\semi}=\bigcup_{[(\l,\O)]} \D^{G_{x_\semi}}_{\l,\O} \quad\text{and}\quad \D^{G_{x_\semi}}_{\l,\O}=G_{x_\semi}(\z(\l)_{\mathsf{gen}}+\O),\]
	where $[(\l,\O)]$ ranges over the $G_{x_\semi}$-conjugacy classes of pairs $(\l,\O)$ of a Levi subalgebra $\l\subset \g_{x_\semi}$ and nilpotent orbit $\O\subset\N_{\l}$.

	\begin{definition}\label{Definition: Natural transversal}
		We define the \textit{natural slice} at $x\in\g$ to be
		\begin{equation} 
			\label{e:defUx}
			\S_x\coloneqq\bigcup_{[(\l,\O)], x\in \ol{\D^{G_{x_\semi}}_{\l,\O}}} \D^{G_{x_\semi}}_{\l,\O}\subset \S_{x_\semi}.
		\end{equation}
	\end{definition}
	
	\begin{proposition}\label{Proposition: Natural transversal}
		If $x\in\g$, then the natural slice $\S_x$ contains $x$ and is a $G_{x_\semi}$-invariant open subset of $\S_{x_\semi}$.
	\end{proposition}
	\begin{proof}
		It is clear that $\S_x$ is $G_{x_\semi}$-invariant. Since $x_\semi \in \z(\g_{x_\semi})_{\mathsf{gen}}$ and $x_\nilp \in \O_{x_\nilp}=G_{x_\semi}x_{\nilp}$, we have $x=x_\semi+x_\nilp\in \z(\g_{x_\semi})_{\mathsf{gen}}+\O_{x_\nilp}=\D^{G_{x_\semi}}_{\g_{x_\semi},\O_{x_\nilp}}\subset \S_x$. To see that $\S_x$ is open in $\g_{x_\semi}$, let $\D^{G_{x_\semi}}_{\l,\O}\subset \S_x$ be one of the subsets appearing on the right-hand side of \eqref{e:defUx}. We claim that the union $\C_{\l,\O}$ of all $\D^{G_{x_\semi}}_{\l',\O'}$ satisfying $\D^{G_{x_\semi}}_{\l,\O}\subset \ol{\D^{G_{x_\semi}}_{\l',\O'}}$ contains an open neighborhood of $\D^{G_{x_\semi}}_{\l,\O}$. If not, there is a sequence in the complement of $\C_{\l,\O}$ converging to a point in $\D^{G_{x_\semi}}_{\l,\O}$. As there are finitely many subsets $\D^{G_{x_\semi}}_{\l'',\O''}$ in total, we can pass to a convergent subsequence and obtain a sequence contained in a single $\D^{G_{x_\semi}}_{\l'',\O''}$ that converges to a point in $\D^{G_{x_\semi}}_{\l,\O}$. It follows that $\ol{\D^{G_{x_\semi}}_{\l'',\O''}} \supset \D^{G_{x_\semi}}_{\l,\O}$. This contradicts $\D^{G_{x_\semi}}_{\l'',\O''}\cap \C_{\l,\O}=\emptyset$, proving the claim.
		
		In light of the above, let $(\l',\O')$ be as above. We have
		\[ \D^{G_{x_\semi}}_{\g_{x_\semi},\O_{x_\nilp}}\subset \ol{\D^{G_{x_\semi}}_{\l,\O}}\subset \ol{\D^{G_{x_\semi}}_{\l',\O'}}, \]
		and hence $\D^{G_{x_\semi}}_{\l',\O'}\subset \S_x$. It follows that $\S_x$ contains an open neighborhood of $\D^{G_{x_\semi}}_{\l,\O}$ for any $\D^{G_{x_\semi}}_{\l,\O}\subset \S_x$, and so $\S_x$ is open.
	\end{proof}
	
	\begin{corollary}\label{Corollary: Poisson slices}
		If $x\in\g$, then $\S_x$ is a Poisson slice in $\g$. 
	\end{corollary}
	
	\begin{proof}
		This follows immediately from Propositions \ref{Proposition: Semisimple case} and \ref{Proposition: Natural transversal}.
	\end{proof}
	
	Since $G\D^{G_{x_\semi}}_{\l,\O}=\D_{\l,\O}$, the $G$-saturation of $\S_x$ is
	\begin{equation} 
		\label{e:GUx}
		G\S_x=\bigcup_{[(\l,\O)], x\in \ol{\D^{G_{x_\semi}}_{\l,\O}}} \D_{\l,\O},
	\end{equation}
	a union of $G$-decomposition classes. Let $\D_x=G(\z(\g_{x_\semi})_{\mathsf{gen}}+x_\nilp)=G(\z(\g_{x_\semi})_{\mathsf{gen}}+\O_{x_\nilp})$ be the decomposition class of $x$, where $\O_{x_\nilp}\coloneqq G_{x_\semi}x_\nilp$.
	\begin{proposition}
		\label{p:decompcontainDx}
		If $x\in\g$, then the following statements are true:
		\begin{itemize}
			\item[\textup{(i)}] $G\S_x$ is an open subset of $\g$;
			\item[\textup{(ii)}] $G\S_x$ is equal to the union of those $G$-decomposition classes whose closure contains $\D_x$.
		\end{itemize}
	\end{proposition}
	\begin{proof}
		It is clear that the decomposition classes appearing on the right-hand side of \eqref{e:GUx} contain $\D_x$ in their closures, as each of the subsets $\D^{G_\semi}_{\l,\O}$ already contains $x$ in its closure. Since $T_u(G\S_x)=[\g,u]+\g_{x_\semi}=\g$ for any $u\in \S_x\subset \S_{x_\semi}$, $G\S_x\subset \g$ is open. Noting that $G\S_x$ also contains $\D_x$, $G\S_x$ is an open neighborhood of $\D_x$. On the other hand, suppose that $\D$ is a decomposition class such that $\ol{\D}\supset \D_x$. The openness of $G\S_x$ in $\g$ implies that $\D\cap G\S_x \ne \emptyset$. As $G\S_x$ is a disjoint union of decomposition classes, $\D$ must be one of the decomposition classes appearing on the right-hand side of \eqref{e:GUx}.
	\end{proof}
	Applying Theorem \ref{t:abs-symp-cr-sec} to the transversal $\S_x$, we obtain the following result.
	\begin{corollary}
		\label{c:slice-generalUx}
		Let $(M,\omega,\mu)$ be a holomorphic Hamiltonian $G$-space and $x \in \mu(M)$. The following statements are true:
		\begin{itemize}
			\item[\textup{(i)}] $\mu^{-1}(\S_x)\subset M$ is a symplectic slice and Hamiltonian $\G_{\S_x}^{\S_x}$-space;
			\item[\textup{(ii)}] there is an isomorphism of Hamiltonian $G$-spaces
			\[ (G\times \mu^{-1}(\S_x))/\G_{\S_x}^{\S_x}\longrightarrow G\mu^{-1}(\S_x);\]
			\item[\textup{(iii)}] the action groupoid $G_{x_\semi}\times \S_x\tto \S_x$ is a subgroupoid of $\G_{\S_x}^{\S_x}$, and $\mu^{-1}(\S_x)$ is a Hamiltonian $G_{x_\semi}$-space.
		\end{itemize}
		
	\end{corollary}
	\begin{remark}
		\label{r:specialtoprin}
		Let $(M,\omega,\mu)$ be a connected, non-empty holomorphic Hamiltonian $G$-space. When $x \in \D_M$, the description of $\mu^{-1}(\S_x)$ simplifies: it follows from Propositions \ref{Proposition: Characterizations}, \ref{Proposition: Holomorphic characterizations}, and \ref{p:decompcontainDx} that $\mu^{-1}(G_{x_\semi}(\z(\l)_{\mathsf{gen}}+\O))=\emptyset$ for any $(\l,\O)$ appearing in \eqref{e:defUx}, other than $(\g_{x_\semi},\O_{x_\nilp})$ itself. Hence $\mu^{-1}(\S_x)=\mu^{-1}(\z(\g_{x_\semi})_{\mathsf{gen}}+\O_{x_{\nilp}})$.
	\end{remark}
	
	\subsection{A root system-theoretic description of $\S_x$}\label{sec:root system description}
	Suppose that $x\in\g$. We now provide a description of $\S_x$ in more combinatorial terms, as in Corollary \ref{c:UxIJ}. Fix a Cartan subalgebra $\t\subset \g_{x_\semi}$ and choice of simple roots $\Delta\subset \Phi$ with $\g_{x_\semi}=\g_I$ for $I\subset\Delta$, as in Corollary \ref{c:UxIJ}. Let $\Phi_I,\Phi_+,\Phi_{I,+}=\Phi_I\cap \Phi_+$ denote the roots of $\g_I$, positive roots of $\g$, and positive roots of $\g_I$, respectively. We have 
	\[ \S_x=\bigcup_{J\subset I}\bigcup_{x\in \ol{\D^I_{J,\O}}} \D^I_{J,\O}.\]
	
	A brief digression on \textit{Lusztig--Spaltenstein induction} is necessary; see \cite[Section 7.1]{col-mcg:93}, \cite[Section 4]{losev2022deformations} for more details on this procedure. To this end, $M$ be a connected complex reductive algebraic group with Lie algebra $\mathfrak{m}$. Consider a Levi subalgebra $\mathfrak{l}\subset\mathfrak{m}$ and nilpotent orbit $\mathcal{O}\subset\mathfrak{l}$. Choose a parabolic subalgebra $\mathfrak{p}\subset\mathfrak{m}$ that contains $\mathfrak{l}$ as a Levi factor, and let $P\subset M$ be the parabolic subgroup integrating $\mathfrak{p}$. The $M$-saturation $M(\overline{\mathcal{O}}+\mathfrak{u}(\p))\subset\mathfrak{m}$ is equal to the closure of a unique nilpotent orbit $\mathrm{Ind}_{\mathfrak{l}}^{\mathfrak{m}}(\mathcal{O})\subset\mathfrak{m}$. This orbit does not depend on the choice of $\p$ \cite[Lemma 4.1]{losev2022deformations}, and is called the \textit{nilpotent orbit induced by $(\mathfrak{l},\mathcal{O})$}.
	
	The previous two paragraphs give context for the following result.
	\begin{proposition}
		\label{Proposition: Uxcombinatorial}
		Suppose that $x\in\g$. The natural slice $\S_x$ is the union
		\begin{equation} 
			\label{e:IJUx}
			\S_x=\bigcup_{J\subset I}\sideset{}{'}\bigcup_{\O \subset \N_J}G_I(\z(\g_J)_{\mathsf{gen}}+\O) 
		\end{equation}
		where the prime $'$ indicates that the union runs over nilpotent $G_J$-orbits $\O$ such that the induced nilpotent $G_I$-orbit $\tn{Ind}_{\g_J}^{\g_I}(\O)$ contains $\O_{x_\nilp}\coloneqq G_Ix_\nilp$ in its closure.
	\end{proposition}
	\begin{proof}
		We have
		\[ \ol{\D^I_{J,\O}}=G_I(\z(\g_J)+\ol{\O}+\n_{I,J}), \]
		where $\n_{I,J}$ is the nilpotent radical of 
		\[ \p_{I,J}\coloneqq\g_J\oplus \bigoplus_{\alpha \in \Phi_{I,+}\backslash \Phi_J} \g_\alpha, \] 
		the standard parabolic of $\g_I$ containing $\g_J$; the proof is similar to \cite[Section 39.2.2]{Tauvel-Yu}. Since $J\subset I$ implies that $\z(\g_I)\subset \z(\g_J)$, the condition $\ol{\D^I_{J,\O}}\supset \D^I_{I,\O_{x_\nilp}}=\z(\g_I)_{\mathsf{gen}}+\O_{x_\nilp}$ amounts to
		\[ (\ol{\O}+\n_{I,J})\cap \O_{x_\nilp} \ne \emptyset.\]
		The result follows, as $\ol{\O}+\n_{I,J}$ is the closure of $\tn{Ind}_{\g_J}^{\g_I}(\O)$.
	\end{proof}
	
	\subsection{The complementary slice at a general element}
	The transversal $\S_x$ is $G_x$-invariant (even $G_{x_\semi}$-invariant), and yet generally does not satisfy $T_x\S_x\cap T_x(Gx)=\{0\}$. We now describe a transversal $\S_{x,\mathfrak{T}}$ that satisfies the latter property; it is not $G_x$-invariant in general. Note that when $G_x$ is not reductive, the tangent space $T_x(Gx)$ need not have a $G_x$-invariant complement.
	
	Suppose that $x \in \g$. Let $\mathfrak{T}=(e,h,f)\in [\g_{x_\semi},\g_{x_\semi}]^{\times 3}$ be an $\mathfrak{sl}_2$-triple with $e=x_{\nilp}$. Let $G_{x,\mathfrak{T}}\subset G_{x_\semi}$ be the subgroup fixing each of $e,h,f$; it is reductive \cite[Lemma 3.7.3]{col-mcg:93}. By Proposition \ref{Proposition: Slodowy Poisson transversal}, the $G_{x,\mathfrak{T}}$-invariant affine subspace 
	\[ \S_{\mathfrak{T}}\coloneqq e+(\g_{x_\semi})_f\subset \g_{x_\semi} \] is a Poisson transversal in $\g_{x_\semi}$. We also have
	\begin{equation}
		\label{e:Axproperties}
		\S_{\mathfrak{T}}\cap G_{x_\semi}x_\nilp=\{x_\nilp\}\quad\text{and}\quad T_{x_\nilp}\S_{\mathfrak{T}}\oplus [\g_{x_\semi},x_\nilp]=\g_{x_{\semi}};
	\end{equation}
	cf. \cite[Section 3.7]{chr-gin}.
	Since $f \in \g_{x_\semi}$, $\z(\g_{x_\semi})\subset (\g_{x_\semi})_f$. It follows that $\S_{\mathfrak{T}}$ is invariant under translation by $\z(\g_{x_\semi})$, implying that $x=x_\semi+x_\nilp \in \S_{\mathfrak{T}}$. One further implication is that
	\begin{equation}
		\label{e:Axproperties2}
		\S_{\mathfrak{T}}\cap G_{x_\semi}x=\{x\}\quad\text{and}\quad T_x\S_{\mathfrak{T}}\oplus [\g_{x_\semi},x]=\g_{x_{\semi}}.
	\end{equation}
	Recall the \textit{natural slice} $\S_x\subset \g_{x_\semi}$ introduced in Definition \ref{Definition: Natural transversal}.
	
	\begin{definition}\label{Definition: Transversal}
		Suppose that $x\in\g$. Consider an $\sl_2$-triple $\mathfrak{T}=(e,h,f)\in[\g_{x_\semi},\g_{x_\semi}]^{\times 3}$ with $e=x_\nilp$. We define the \textit{complementary slice} associated to $x$ and $\mathfrak{T}$ to be $$\S_{x,\mathfrak{T}}\coloneqq \S_x\cap\S_{\mathfrak{T}}.$$
	\end{definition}
	
	\begin{corollary}\label{Corollary: Nice}
		If $x$ and $\mathfrak{T}$ are as in Definition \ref{Definition: Transversal}, then $\S_{x,\mathfrak{T}}$ is a Poisson slice in $\g$.	
	\end{corollary}
	
	\begin{proof}
		Note that $\S_{x,\mathfrak{T}}$ is a Poisson transversal in $\S_x$. Corollary \ref{Corollary: Poisson slices} implies that the latter is a Poisson transversal in $\g$. It follows that $\S_{x,\mathfrak{T}}$ is a Poisson transversal in $\g$. We conclude that $\S_{x,\mathfrak{T}}$ is a Poisson slice in $\g$.
	\end{proof}
	
	Let $x$ and $\mathfrak{T}$ be as in Definition \ref{Definition: Transversal}. By means of Definition \ref{Definition: Contracting action}, $\mathfrak{T}$ determines a contracting action of $\mathbb{C}^{\times}$ on $\mathcal{S}_{\mathfrak{T}}$. The following result arises in this context.
	\begin{proposition}
		\label{Proposition: Slodowy saturation}
		Let $x\in\g$ and $\mathfrak{T}\in [\g_{x_\semi},\g_{x_\semi}]^{\times 3}$ be as in Definition \ref{Definition: Transversal}.
		\begin{itemize}
			\item[\textup{(i)}] The open subset $\S_{x,\mathfrak{T}}\subset \S_{\mathfrak{T}}$ is invariant under the contracting action of $\mathbb{C}^{\times}$ on $\mathcal{S}_{\mathfrak{T}}$.
			\item[\textup{(ii)}] We have $G_{x_\semi}\S_{x,\mathfrak{T}}=\S_x$. 
			\item[\textup{(iii)}] If $x$ is nilpotent, then $\S_{x,\mathfrak{T}}=\S_{\mathfrak{T}}$; otherwise, $\S_{x,\mathfrak{T}}\subsetneqq \S_{\mathfrak{T}}$.
		\end{itemize}
	\end{proposition}
	\begin{proof}
		The subsets $\D^{G_{x_\semi}}_{\l,\O}=G_{x_\semi}(\z(\l)_{\mathsf{gen}}+\O)$ appearing in \eqref{e:defUx} are $G_{x_\semi}$-invariant, as well as invariant under non-zero scaling. It follows that $\S_{x,\mathfrak{T}}$ is invariant under the $\bC^\times$-contracting action, proving (i). If $x_\semi=0$, then $x_\nilp=x\in \S_{x,\mathfrak{T}}$. It is also clear that $\S_{x,\mathfrak{T}}$ is an open neighborhood of $x$ in $\S_{\mathfrak{T}}$, invariant under the $\bC^\times$-contracting action. We conclude that $\S_{x,\mathfrak{T}}=\S_{\mathfrak{T}}$. On the other hand, if $x_{\semi} \ne 0$, then $x_{\nilp}\notin \S_{x,\mathfrak{T}}$. Hence $\S_{x,\mathfrak{T}}\subsetneqq \S_{\mathfrak{T}}$. This proves (iii).
		
		Turning to the proof of (ii), it is clear that $G_{x_\semi}\S_{x,\mathfrak{T}}\subset \S_x$. For the reverse inclusion, we must show every element of $\S_x$ is $G_{x_\semi}$-conjugate to an element of $\S_{x,\mathfrak{T}}$. Note that if $y \in \S_x$ is $G_{x_\semi}$-conjugate to an element $y'$ of $\S_{\mathfrak{T}}$, the $G_{x_\semi}$-invariance of $\S_x$ implies that $y'$ belongs to $\S_{x,\mathfrak{T}}\subset \S_{\mathfrak{T}}$ automatically. It therefore suffices to show that $\S_x\subset G_{x_\semi}\S_{\mathfrak{T}}$, or equivalently that
		\[ U:=\S_x-x_\semi \subset G_{x_\semi}\S_{\mathfrak{T}}-x_\semi=G_{x_\semi}\S_{\mathfrak{T}}.\] 
		Slodowy slices have the property that $G_{x_\semi}\S_{\mathfrak{T}}$ contains an open neighborhood $U$ of the orbit $G_{x_\semi}x_\nilp$, and is invariant under the $\bC^\times$-scaling action of $\g$. We claim the following: it suffices to show that for any $y \in U$, the closure of the $G_{x_\semi}\times \bC^\times$-orbit through $y$ contains $e=x_\nilp$; indeed, if this is the case, then since $G_{x_\semi}\S_{\mathfrak{T}}$ is an open neighborhood of $e=x_\nilp$, $(G_{x_\semi}\times \bC^\times)y\cap G_{x_\semi}\S_{\mathfrak{T}}\ne \emptyset$. Then $y \in G_{x_\semi}\S_{\mathfrak{T}}$ by the $(G_{x_\semi}\times \bC^\times)$-invariance of $G_{x_\semi}\S_{\mathfrak{T}}$. 
		
		It is convenient to fix a Cartan subalgebra $\t\subset \g_{x_\semi}$ containing $y_\semi$, as well as a choice of simple roots $\Delta$ such that $\g_{x_\semi}=\g_I$ with $I\subset \Delta$. Let $y_\semi'=y_\semi+x_\semi$ and let $\O=G_{y_\semi}y_\nilp$. Then 
		\[ y=y_\semi+y_\nilp=y_\semi'-x_\semi+y_\nilp\in (\z(\g_J)_{\mathsf{gen}}-x_\semi)+\O\subset U \quad\text{and}\quad y_\semi'+y_\nilp\in \S_x.\] 
		Proposition \ref{Proposition: Uxcombinatorial} tells us that $e=x_\nilp$ belongs to the closure of $\tn{Ind}_{\g_J}^{\g_I}(\O)$, which is the subset $G_I(\ol{\O}+\n_{I,J})\subset \g_I$. Up to conjugating by an element of $G_I$, we may assume $x_\nilp=e=e'+e'' \in \ol{\O}+\n_{I,J}$. Choose a norm on the vector space $\g_I$ for convenience, and choose $g'\in G_J$ such that $y_\nilp':=g'y_\nilp$ is less than a distance $\epsilon>0$ from $e'$. Note that $y_\semi',y_\nilp'$ are the semisimple and nilpotent components of $y'=y_\semi'+y_\nilp'$. We conclude that $\g_{y'}\subset \g_{y_\semi'}$. Moreover, $y_\semi'\in \z(\g_J)_{\mathsf{gen}}$. It follows that $\g_{y_\semi'}=\g_J$. Hence $\g_{y'}\cap \n_{I,J}\subset \g_J\cap \n_{I,J}=0$. Since $\ker(\ad_{y'})=\g_{y'}$, this proves that $\ker(\ad_{y'})\cap \n_{I,J}=0$. We also note that $\ad_{y'}(\n_{I,J})\subset \n_{I,J}$, as $\n_{I,J}$ is normalized by $\g_J$ and $y' \in \g_J$. Noting that $\ker(\ad_{y'})\cap \n_{I,J}=0$, it follows that $\ad_{y'}(\n_{I,J})=\n_{I,J}$; compare this to \cite[Section 39.2.2]{Tauvel-Yu}). The same argument shows $\ad_{y'_\lambda}(\n_{I,J})=\n_{I,J}$ for $y'_\lambda=\lambda y_\semi'+y_\nilp'$ and any $\lambda \in \bC^\times$. By \cite[Lemma 32.2.5]{Tauvel-Yu}, $\exp(\n_{I,J})(\lambda y_\semi'+y_\nilp')=\lambda y_\semi'+y_\nilp'+\n_{I,J}$. We may therefore choose $g''_\lambda\in \exp(\n_{I,J})$ such that $g''_\lambda (\lambda y_\semi'+y_\nilp')=\lambda y_\semi'+y_\nilp'+e''$. We also observe that $y_\nilp' \in \O$ implies $\lambda^{-1}y_\nilp' \in \O$, allowing us to choose $g_\lambda'\in G_J$ such that $g_\lambda'y_\nilp'=\lambda^{-1}y_\nilp'$. Then
		\[ \lambda g''_\lambda g_\lambda'g'y=\lambda g''_\lambda g_\lambda'(y_\semi'-x_\semi+y_\nilp')=\lambda g''_\lambda(y_\semi'-x_\semi+\lambda^{-1}y_\nilp')=g''_\lambda(\lambda y_\semi'+y_\nilp'-\lambda x_\semi)=\lambda y_\semi'+y_\nilp'+e''-\lambda x_\semi. \]
		Letting $\lambda \longrightarrow 0$ gives $y_\nilp'+e''=g'y_\nilp+e''$, which is less than distance $\epsilon$ of $x_\nilp=e=e'+e''$ by construction. Since $\epsilon>0$ was arbitrary, this completes the proof.
	\end{proof}
	
	Applying Theorem \ref{t:abs-symp-cr-sec} to $\S_{x,\mathfrak{T}}$ we obtain a symplectic slice theorem. Let $\G=T^*G=G\times \g^*\tto \g^*$.
	\begin{corollary}
		\label{c:slice-general}
		Let $(M,\omega,\mu)$ be a holomorphic Hamiltonian $G$-space and let $x \in \mu(M)$. Then $\mu^{-1}(\S_{x,\mathfrak{T}})\subset M$ is a symplectic slice and Hamiltonian $\G_{\S_{x,\mathfrak{T}}}^{\S_{x,\mathfrak{T}}}$-space, and there is an isomorphism of Hamiltonian $G$-spaces
		\[ (G\times \mu^{-1}(\S_{x,\mathfrak{T}}))/\G_{\S_{x,\mathfrak{T}}}^{\S_{x,\mathfrak{T}}}\longrightarrow G\mu^{-1}(\S_{x,\mathfrak{T}}).\]
		The action groupoid $G_{x,\mathfrak{T}}\times \S_{x,\mathfrak{T}}\tto \S_{x,\mathfrak{T}}$ is a subgroupoid of $\G_{\S_{x,\mathfrak{T}}}^{\S_{x,\mathfrak{T}}}$ and $\mu^{-1}(\S_{x,\mathfrak{T}})$ is a Hamiltonian $G_{x,\mathfrak{T}}$-space.
	\end{corollary}
	\begin{remark}
		\label{r:specialtoprin2}
		Continuing Remark \ref{r:specialtoprin}, when $x\in \D_M$, the description of $\mu^{-1}(\S_{x,\mathfrak{T}})$ simplifies: since $\O_{x_{\nilp}}\cap \S_{\mathfrak{T}}=\{x_\nilp\}$, we have $\mu^{-1}(\S_{x,\mathfrak{T}})=\mu^{-1}(\z(\g_{x_\semi})_{\mathsf{gen}}+x_\nilp)$, recovering the symplectic subvariety in Proposition \ref{Proposition: Symplectic subvariety}.
	\end{remark}
	
	\bibliographystyle{acm}
	\bibliography{decomposition}
\end{document}